\documentclass[11pt]{amsart}
\usepackage{amscd}
\usepackage{amssymb}
\usepackage{graphicx}
\newcounter{mtheorem}

\newtheorem{theorem}{Theorem}[section]
\newtheorem{lemma}[theorem]{Lemma}
\newtheorem{prop}[theorem]{Proposition}
\newtheorem{corollary}[theorem]{Corollary}

\theoremstyle{definition}
\newtheorem{definition}[theorem]{Definition}
\newtheorem{example}[theorem]{Example}

\theoremstyle{remark}
\newtheorem{remark}[theorem]{Remark}

\numberwithin{equation}{section}
\newcommand{\E}{\mathcal{E}}
\newcommand{\cone}{\mathcal{C}}

\newcommand{\unitary}[1]{\textrm{U({#1})}}
\newcommand{\sunitary}[1]{\textrm{SU({#1})}}

\newcommand{\C}{\mathbb{C}}
\newcommand{\R}{\mathbb{R}}

\newcommand{\N}{\mathbb{N}}

\newcommand{\Sph}{\mathbb{S}}

\newcommand{\Imag}{\operatorname{Im}}
\newcommand{\Real}{\operatorname{Re}}

\newcommand{\tnabla}{{\widetilde{\nabla}}}
\newcommand{\tg}{{\tilde{g}}}

\newcommand{\bt}{\boldsymbol{t}}
\title[SL conifolds, II]{Special Lagrangian conifolds, II: \linebreak Gluing constructions in $\C^m$}
\author[T.~Pacini]{Tommaso~Pacini}
\address{Scuola Normale Superiore, Pisa} \email{tommaso.pacini@sns.it}
\date{\today}
\begin{document}
\begin{abstract}
We prove two gluing theorems for special Lagrangian (SL) conifolds in $\C^m$. Conifolds are a key ingredient in the compactification problem for moduli spaces of compact SLs in Calabi-Yau manifolds. 

In particular, our theorems yield the first examples of smooth SL conifolds with 3 or more planar ends and the first (non-trivial) examples of SL conifolds which have a conical singularity but are not, globally, cones. We also obtain: (i) a desingularization procedure for transverse intersection and self-intersection points, using ``Lawlor necks''; (ii) a construction which completely desingularizes any SL conifold by replacing isolated conical singularities with non-compact asymptotically conical (AC) ends; (iii) a proof that there is no upper bound on the number of AC ends of a SL conifold; (iv) the possibility of replacing a given collection of conical singularities with a completely different collection of conical singularities and of AC ends. 

As a corollary of (i) we improve a result by Arezzo and Pacard \cite{arezzopacard} concerning minimal desingularizations of certain configurations of SL planes in $\C^m$, intersecting transversally. 
\end{abstract}
\maketitle
\tableofcontents


\section{Introduction}\label{s:intro}

Let $M$ be a Calabi-Yau (CY) manifold. Roughly speaking, a submanifold
$L\subset M$ is \textit{special Lagrangian} (SL) if it is both minimal and
Lagrangian with respect to the ambient Riemannian and symplectic structures. 

SLs are examples of ``calibrated submanifolds'' \cite{harveylawson} and are thus interesting from the point of view of Geometric Measure Theory. They also have other nice geometric properties, including smooth moduli spaces. Research in this field is largely guided by several conjectures relating SLs to Mirror Symmetry \cite{kontsevich}, \cite{syz} and to the search for invariants for CY manifolds \cite{joyce:3spheres}. 

The simplest example of a CY manifold is $\C^m$, endowed with its standard structures. Recall that $\C^m$ cannot admit compact minimal submanifolds. In this ambient space it is therefore necessary to study non-compact SLs. It is also important to be able to work with singular objects. The simplest most natural class of SLs in $\C^m$ is thus the class of SL \textit{conifolds}: submanifolds admitting both isolated ``conical singularities'' (CS) and non-compact  ``asymptotically conical'' (AC) ends: the former modelled on the ``tip'', the latter on the ``large end'', of SL cones. SL cones are, of course, the most basic example of conifolds. If a SL conifold is smooth then, by definition, it can have only AC ends: we will refer to it as an AC SL.

The main motivation for studying SL conifolds in $\C^m$ stems from the fact that, up to first order and in appropriate coordinate systems, any CY manifold $M$ is modelled on $\C^m$. Likewise, a SL submanifold $L$ in $M$ with an isolated conical singularity is locally modelled on a SL cone in $\C^m$. Assume there exists an AC SL $\hat{L}$ in $\C^m$, asymptotic to that cone. The work of Joyce \cite{joyce:III}, \cite{joyce:IV}, \cite{joyce:V} then shows, under appropriate assumptions, how to glue $\hat{L}$ into a neighbourhood of the singularity obtaining a 1-parameter family $L_t$ of smooth SLs in $M$ which converges to $L$ as $t\rightarrow 0$. Although this is an impressive result, there are two important limits to Joyce's work. The first is the fact that it applies only to compact SLs; this prevents us from applying his results to the case $M=\C^m$. The second is the lack of examples to which to apply his results. More specifically, (i) we have no way of producing SLs with isolated conical singularities in general CY manifolds, and (ii) we have very few examples of AC SLs in $\C^m$. We do instead have many examples of SL cones in $\C^m$: \cite{harveylawson}, \cite{cheng} and recent work \cite{carberrymc}, \cite{haskins}, \cite{haskinskapouleas}, \cite{joyce:symmetries} have produced many classes of examples of SL cones. Some of these are known however not to admit AC SL desingularizations \cite{haskinspacini}, and in general we do not know which do.

The goal of this paper is to define a gluing construction which produces new examples of AC SLs and of SL conifolds in $\C^m$. It is also Part II of a multi-step project aiming to set up a general theory of special Lagrangian conifolds. Two other papers related to this project are currently available: \cite{pacini:ics}, \cite{pacini:sldefs} (see also \cite{pacini:sldefsextended}). The first of these papers provides the analytic foundations for our gluing construction, the second provides the geometric foundations, but actually each paper is self-contained and has its own, independent, focus.

One of the most basic examples to which our construction applies is the following. Given a pair of transverse SL planes satisfying certain ``angle conditions'', Lawlor \cite{lawlor} constructed an AC SL interpolating between them: these submanifolds are known as \textit{Lawlor necks}. Assume given a finite number of SL planes in $\C^m$, such that each intersection satisfies Lawlor's conditions. We then prove that it is possible to glue a rescaled copy of the appropriate Lawlor neck into a neighbourhood of each intersection point obtaining a family of AC SL submanifolds, parametrized by the ``size'' of the necks, which converges to the initial configuration of planes as the parameters tend to zero, cf. Example \ref{ex:SL_planes_desing}. This result extends previous work by Arezzo-Pacard \cite{arezzopacard} which had produced minimal (but not Lagrangian) desingularizations of similar configurations of SL planes under additional technical hypotheses. In particular, our construction produces the first examples of smooth SL conifolds in $\C^m$ with 3 or more planar ends.
 
In the above example one should think of transverse intersections as special types of isolated conical singularities and of Lawlor necks as special types of local desingularizations of the singularity. Our main results, Theorem \ref{thm:ACSLgluing} and Theorem \ref{thm:conifold_gluing}, then generalize this example in two ways. Theorem \ref{thm:ACSLgluing} is a gluing result involving arbitrary singularities and corresponding AC SL desingularizations. The final product is a family of new AC SLs. Example \ref{ex:SL_planes_desing}, described above, is a corollary of this result. As further corollaries, Example \ref{ex:SL_desing} generalizes Example \ref{ex:SL_planes_desing} by desingularizing transverse intersections of arbitrary AC SLs. It also shows how to attach an arbitrary number of new ends onto a given AC SL. Example \ref{ex:SL_doubling} shows how to replace arbitrary isolated conical singularities with AC ends, thus transforming any singular conifold into a smooth AC SL.

Example \ref{ex:SL_doubling} is interesting in that it completely desingularizes any SL conifold. Notice however this is only true in a rather weak sense: it replaces compact conical singularities with new non-compact ends, thus changing the nature of the initial conifold rather drastically. In the light of \cite{haskinspacini}, if the goal is indeed to obtain a smooth object, this may actually be the best possible result: in general one should expect the existence of conifolds $L$ some of whose singularities do admit AC SL desingularizations, but others do not. The alternative solution for dealing with such conifolds is to work only with the ``smoothable'' singularities, leaving the others alone. The result will only be a partial desingularization of $L$, but this outcome may be  preferable to that of Example \ref{ex:SL_doubling}. This type of gluing forces us however to work, from start to finish, with singular conifolds. The additional complications are dealt with in Theorem \ref{thm:conifold_gluing}. The final product is a family of new SL conifolds. In particular, this theorem  gives the option of adding new planar ends to a singular conifold, \textit{e.g.} to a cone, thus generating the first (non-trivial) examples of SL conifolds which have a conical singularity but are not, globally, cones: cf. Example \ref{ex:SL_desing+sings} for details. The theorem also gives the possibility of taking a singular conifold $L$, cutting out one singularity, and gluing in a different singular conifold $\hat{L}$. This is potentially very interesting: coupling such a construction with Joyce's gluing results would produce a way to jump between different compact singular SLs inside a CY manifold $M$. Unfortunately, as in the work of Joyce mentioned above, this type of construction is still largely theoretical due to the lack of interesting examples of  conifolds $L$, $\hat{L}$ to apply it to.

Theorems \ref{thm:ACSLgluing} and \ref{thm:conifold_gluing} may very well be optimal. Our technical hypotheses concern only the parameters used to set up the gluing process: these parameters disappear in the final result, so they are of no real importance. We also obtain very good control over the final asymptotics of the conifold, near the CS and AC ends. Our theorems set only two restrictions on the SL conifolds themselves, as follows. Theorem \ref{thm:conifold_gluing} requires the remaining singularities to be ``stable'': this assumption is rather natural and has already appeared in previous work of Joyce and Haskins, see also \cite{ohnita} and \cite{pacini:sldefs}. The second restriction is as follows. The first step in the gluing process concerns the construction of certain ``approximate solutions'' to the gluing problem: basically, it is necessary to choose a Lagrangian interpolation between the initially given SL conifolds $L$, $\hat{L}$. Using an additional assumption as in Joyce \cite{joyce:III} we can reduce this problem to a choice of ``interpolating function''. Roughly speaking, this assumption is that the conifold $\hat{L}$ converges fairly quickly to its asymptotic cone. This is a strong assumption: basically only one category of examples is known, cf. Example \ref{e:C^m_examples} for details. Work on how to remove this restriction, as in Joyce \cite{joyce:IV}, is currently in progress. Notice however that: (i) the new examples obtained in our paper show that, even with the current assumption, these few known ingredients can be combined to produce interesting new AC SLs and SL conifolds; (ii) many of the examples we produce have the same strong rate of convergence so they can be fed back into the same machine to yield new, more complicated, submanifolds; 
(iii) even removing this restriction will add only a few more AC SLs to our list of possible ingredients, bringing us back to the general issue that the known number of AC SLs is very small.

The above presentation should convince the reader that the choice of working in the flat ambient space $\C^m$ is based on precise goals rather than on technical convenience. As already explained, this choice forces us to work with non-compact submanifolds. Generally speaking, non-compact submanifolds are more difficult to work with than compact ones. In our case non-compactness is partially compensated for by the assumption of working with conifolds. The theory concerning ``large ends'' is the same as that concerning isolated conical singularities, so in this paper we can rely on several ideas introduced by Joyce in his work on compact CS SLs in CYs. This is particularly true when dealing with the more geometric aspects of these gluing theorems. In particular, the Lagrangian neighbourhoods we use in Section \ref{ss:conifold_defs} to set up the gluing problem go mostly back to Joyce, with minor adaptations and changes introduced in \cite{pacini:sldefsextended}. The quadratic estimates of Section \ref{s:SL_defs} also originate in the work of Joyce, though a certain amount of extra work is needed to adapt them to the presence of CS and AC ends.

From the analytic point of view, however, the situation is rather different. Joyce formulated his results using non-weighted Sobolev spaces. In the non-compact setting, weighted Sobolev spaces become inevitable. This is not as inconvenient as it may seem: using weighted spaces simplifies other issues and they have thus become the standard choice in other gluing problems. In this paper we get two main benefits out of using weighted spaces:
(i) as proved in \cite{pacini:ics}, coupling weighted spaces with careful choices of how to parametrize the ``neck regions'' of our manifolds allows us to obtain estimates which are completely uniform with respect to the gluing parameter; (ii) the presence of non-compact ends allows us to choose certain weights ``at infinity'' which kill the kernel of our linearized operator, thus leading to invertibility. In particular, when dealing with conifolds in $\C^m$ the number of components of $L$ makes no difference. 

Together, (i) and (ii) streamline the gluing process considerably, compared to analogous results for compact SLs. It is thus interesting to compare our techniques with those used by Joyce and, earlier, by Butscher \cite{butscher} and Lee \cite{lee}. Of course, it is important to emphasize that in the compact setting straight-forward invertibility of the linearized operator is not possible: each connected component of $L$ introduces new ``approximate kernel''. Some invertibility results contained in \cite{pacini:ics} apply however to compact manifolds. We thus believe it is possible to combine the methods used in this paper with the results of \cite{pacini:ics} to obtain stronger and simpler results for compact SLs than those currently known.  

It is also interesting to compare the SL gluing problem with gluing problems involving other classes of calibrated submanifolds. A special feature of the SL case is that the SL condition can be decoupled into two, weaker, conditions. In particular, SL submanifolds are Lagrangian so by restricting to the subspace of Hamiltonian deformations one can reduce some aspects of SL geometry from a system of PDEs to a scalar PDE. Analogous reductions do not exist for other calibrations, so one is forced to work with systems of PDEs throughout. We refer to \cite{lotay1}, \cite{lotay2} for gluing results concerning compact ``coassociative'' submanifolds and to \cite{nordstrom} for gluing results concerning compact ``associative'' submanifolds. The techniques of \cite{pacini:ics} should be applicable to these classes of submanifolds, as well as to a variety of other gluing problems.

\ 

\textbf{Important remarks: }Throughout this paper we will often encounter chains of inequalities of the form
\begin{equation*}
 |e_0|\leq C_1|e_1|\leq C_2|e_2|\leq\dots
\end{equation*}
The constants $C_i$ will often depend on factors that are irrelevant within the given context. In this case we will sometimes simplify such expressions by omitting the subscripts of the constants $C_i$, \textit{i.e.} by using a single constant $C$. Furthermore, to simplify certain arguments, we always assume that our manifolds satisfy the dimension constraint $m\geq 3$. We refer to Joyce \cite{joyce:I} Section 2 for a presentation of some of the issues which arise in the underlying linear elliptic theory when $m=2$.


\section{Review of Lagrangian conifolds}\label{s:con_review}

\begin{definition} \label{def:lagr}
Let $(M^{2m},\omega)$ be a symplectic manifold. An embedded or immersed
submanifold $\iota:L^m\rightarrow M$ is \textit{Lagrangian} if
$\iota^*\omega\equiv 0$. The immersion allows us to view the tangent bundle $TL$
of $L$ as a subbundle of $TM$ (more precisely, of $\iota^*TM$). When $M$ is
K\"ahler with structures $(g,J,\omega)$ it is simple to check that $L$ is
Lagrangian if and only if $J$ maps $TL$ to the normal bundle $NL$ of $L$, \textit{i.e.}
$J(TL)=NL$. 

We will denote by $\tilde{g}, \tilde{J},\tilde{\omega}$ the standard Euclidean, complex and symplectic structures on $\R^{2m}=\C^m$. A subset $\mathcal{C}$ of $\R^{2m}$ is a \textit{cone} if it is invariant under dilations of $\R^{2m}$, \textit{i.e.} if $t\cdot\mathcal{C}=\mathcal{C}$, for all $t>0$. It is uniquely identified by its \textit{link} $\Sigma:=\mathcal{C}\cap \Sph^{2m-1}$.
\end{definition}

\begin{definition}\label{def:aclagsub}
Let $L^m$ be a smooth manifold, not necessarily connected. Assume given a Lagrangian immersion
$\iota:L\rightarrow \C^m$. 
We say that $L$ is an \textit{asymptotically conical Lagrangian submanifold}
with \textit{rate} $\boldsymbol{\lambda}$ if it satisfies the following
conditions.
\begin{enumerate}
\item We are given a compact subset $K\subset L$ such that $S:=L\setminus K$ has
a finite number of connected components $S_1,\dots,S_e$.
\item We are given Lagrangian cones $\mathcal{C}_i\subset \C^m$ with smooth
connected links $(\Sigma_i,g_i'):=\mathcal{C}_i\bigcap \Sph^{2m-1}$. Let
$\iota_i:\Sigma_i\times (0,\infty)\rightarrow \C^m$ denote the natural
immersions, parametrizing $\mathcal{C}_i$.
\item We are finally given an $e$-tuple of \textit{convergence rates}
$\boldsymbol{\lambda}=(\lambda_1,\dots,\lambda_e)$ with $\lambda_i<2$, \textit{centers} $p_i\in\C^m$ and
diffeomorphisms $\phi_i:\Sigma_i\times [R,\infty)\rightarrow \overline{S_i}$
for some $R>0$ such that, for $r\rightarrow\infty$ and all $k\geq 0$,
\begin{equation}\label{eq:aclagdecay}
|\tnabla^k(\iota\circ\phi_i-(\iota_i+p_i))|=O(r^{\lambda_i-1-k})
\end{equation}
with respect to the conical metric $\tg_i=dr^2+r^2g_i'$ on $\cone_i$.
\end{enumerate}
\end{definition}

\begin{definition}\label{def:cslagsub}
Let $\bar{L}^m$ be a manifold, not necessarily connected, smooth except for a finite number of possibly singular
points $\{x_1,\dots,x_e\}$. Assume given a continuous map
$\iota:\bar{L}\rightarrow \C^m$ which restricts to a smooth Lagrangian immersion of
$L:=\bar{L}\setminus\{x_1,\dots,x_e\}$. We say that $\bar{L}$ (or $L$) is a \textit{conically singular
Lagrangian submanifold} with \textit{rate} $\boldsymbol{\mu}$ if it satisfies
the following conditions.
\begin{enumerate}
\item We are given open connected neighbourhoods $S_i$ of $x_i$.
\item We are given Lagrangian cones $\mathcal{C}_i\subset \C^m$ with smooth
connected links $(\Sigma_i,g_i'):=\mathcal{C}_i\bigcap \Sph^{2m-1}$. Let
$\iota_i:\Sigma_i\times (0,\infty)\rightarrow \C^m$ denote the natural
immersions, parametrizing $\mathcal{C}_i$.
\item We are finally given an $e$-tuple of \textit{convergence rates}
$\boldsymbol{\mu}=(\mu_1,\dots,\mu_e)$ with $\mu_i>2$, \textit{centers} $p_i\in\C^m$ and diffeomorphisms
$\phi_i:\Sigma_i\times (0,\epsilon]\rightarrow \overline{S_i}\setminus\{x_i\}$ such that, for
$r\rightarrow 0$ and all $k\geq 0$,
\begin{equation}\label{eq:cslagdecay}
|\tnabla^k(\iota\circ\phi_i-(\iota_i+p_i))|=O(r^{\mu_i-1-k})
\end{equation}
with respect to the conical metric $\tg_i=dr^2+r^2g_i'$ on $\cone_i$. 
\end{enumerate}
\end{definition}

\begin{remark} Notice that in the case of a conically singular submanifold the centers are uniquely defined by the fact that $\iota(x_i)=p_i$. For asymptotically conical submanifolds the centers are uniquely defined only when $\lambda_i<1$. For other values of $\lambda_i$ the asymptotic cones (which by definition pass through the origin of $\C^m$) are unique but the convergence rate is so weak that the submanifold also converges to any translated copy $\mathcal{C}_i+p_i'$ of the cones (\textit{e.g.} if $\lambda_i=1$), or even slowly pulls away from the cones (if $\lambda_i>1$). In these cases we consider the centers $p_i$ as an additional piece of data.
 \end{remark}

\begin{definition} \label{def:accslagsub}
Let $\bar{L}^m$ be a manifold, not necessarily connected, smooth except for a finite number of possibly singular
points $\{x_1,\dots,x_s\}$. Assume
given a continuous map $\iota:\bar{L}\rightarrow \C^m$ which restricts to a smooth Lagrangian
immersion of $L:=\bar{L}\setminus\{x_1,\dots,x_s\}$. We say that $\bar{L}$ (or
$L$) is a \textit{CS/AC Lagrangian submanifold} with \textit{rate}
$(\boldsymbol{\mu},\boldsymbol{\lambda})$ if neighbourhoods $S_i$ of the points
$x_i$ satisfy Definition \ref{def:cslagsub} with rates $\mu_i$ and the complement $\bar{L}\setminus \cup S_i$ satisfies Definition \ref{def:aclagsub} with rates $\lambda_i$. We will often not distinguish between $\bar{L}$ and $L$.

We use the generic term \textit{Lagrangian conifold} to indicate any CS, AC or CS/AC Lagrangian
submanifold. We will denote by $g:=\iota^*\tilde{g}$ the induced metric on $L$.
\end{definition}

\begin{example} \label{ex:smooth_is_sing}
 Notice that any smooth point $p$ of a Lagrangian submanifold $L$ can be labelled as a CS singularity. Indeed, write $L$ as a graph over its tangent plane $T_pL$ and set $\mathcal{C}:=T_pL$. Since $\mathcal{C}$ is smooth through the origin and the graphing map is also smooth, a Taylor expansion shows that the graphing map vanishes to first order and the remainder is quadratic + higher order. By using polar coordinates on $T_pL$ one then obtains a parametrization of $L$ as in Definition \ref{def:cslagsub}, with $\mu=3$. 

In particular this is true for intersection and self-intersection points of the immersion $\iota$. Consider, for example, the case of two Lagrangian planes in $\C^m$ intersecting transversely in one point $p$. If we let $L$ denote the disjoint union of two copies of $\R^m$ we can parametrize our configuration of Lagrangian planes in $\C^m$ via an immersion of $L$ which maps the origins to $p$. This submanifold clearly has two AC ends. It would be natural to consider the origins as smooth points in $\R^m$ but we can also decide to label them as singularities. In this case the submanifold will also have two CS ends. This latter set-up will allow us, in Sections \ref{s:lagr_sum} and \ref{s:SL_sum}, to ``desingularize'' the point $p$ by gluing in a small Lagrangian ``neck'' which interpolates between the two planes. The initially disconnected manifold $L$ will then become connected.
\end{example}

Let $\iota:L\rightarrow\C^m$ be a Lagrangian conifold, with induced metric $g$. Choose a CS component $S_i$. Let
$\phi_i$ denote the diffeomorphism of Definition \ref{def:cslagsub} and set 
$\nu_i:=\mu_i-2>0$. One can then check that, as $r\rightarrow
0$ and for all $k\geq 0$,
\begin{equation}\label{eq:cs_metric}
|\tnabla^k(\phi_i^*g-\tg_i)|_{\tg_i}=O(r^{\nu_i-k}),
\end{equation}
where $\tnabla$ is the Levi-Civita connection on $\mathcal{C}_i$ defined by $\tg_i$. 

Analogously, choose an AC component $S_i$, let $\phi_i$ be the diffeomorphism of Definition \ref{def:aclagsub} and set $\nu_i:=\lambda_i-2<0$. Then, as $r\rightarrow
\infty$ and for all $k\geq 0$,
\begin{equation}\label{eq:ac_metric}
|\tnabla^k(\phi_i^*g-\tg_i)|_{\tg_i}=O(r^{\nu_i-k}),
\end{equation}
where $\tnabla$ is the Levi-Civita connection on $\mathcal{C}_i$ defined by $\tg_i$.

This shows that the Riemannian manifold $(L,g)$ is an \textit{abstract conifold} in the sense of \cite{pacini:ics}. We will call the components $S_i$ the \textit{ends} of $L$. On an abstract conifold each end defines a connected \textit{abstract link} $(\Sigma_i,g_i')$. The end is diffeomorphic to the \textit{abstract cone} $C_i=\Sigma_i\times (0,\infty)$ and the metric on the end is asymptotic, in the above sense, to the conical metric $\tilde{g}_i:=dr^2+r^2g_i'$.

\begin{remark}\label{rem:tensored_metrics}
Set $\sigma:=\phi_i^*g-\tg_i$. The object $\tnabla^k\sigma$ belongs to a bundle obtained via tensor products, so the metric used in Equations \ref{eq:cs_metric}, \ref{eq:ac_metric} to measure the norm of $\tnabla^k\sigma$ is obtained by tensoring the metric $\tg_i$ (applied to $\tnabla^k$) with the same metric $\tg_i$ (applied to $\sigma$). It is sometimes convenient to emphasize this fact by using the alternative notation  $|\tnabla^k(\phi_i^*g-\tg_i)|_{\tg_i\otimes\tg_i}$. One can then check that Equation \ref{eq:cs_metric} coincides with
\begin{equation}\label{eq:cs_metric_bis}
|\tnabla^k(\phi_i^*g-\tg_i)|_{r^{-2}\tg_i\otimes\tg_i}=O(r^{\nu_i}).
\end{equation}
Equations \ref{eq:aclagdecay}, \ref{eq:cslagdecay} and \ref{eq:ac_metric} can also be rewritten this way. 
\end{remark}

Analysis on abstract conifolds is a well-developed theory. The next two sections summarize the main definitions and results relevant to this paper, referring to \cite{pacini:ics} for details and further references.

\subsection{Analysis on abstract conifolds}\label{ss:analysis_conifolds}

Let $E$ be a vector bundle over $(L,g)$. Assume $E$ is endowed with a metric and
metric connection $\nabla$: we say that $(E,\nabla)$ is a \textit{metric pair}.
In this paper $E$ will usually be a bundle of differential forms $\Lambda^r$
on $L$, endowed with the metric and Levi-Civita connection induced from $g$. 

Regarding notation, given a vector
$\boldsymbol{\beta}=(\beta_1,\dots,\beta_e)\in \R^e$ and $j\in\N$ we set
$\boldsymbol{\beta}+j:=(\beta_1+j,\dots,\beta_e+j)$. We write
$\boldsymbol{\beta}\geq\boldsymbol{\beta}'$ if and only if $\beta_i\geq\beta_i'$.

\begin{definition}\label{def:csac_sectionspaces}
Let $(L,g)$ be a conifold with $e$ ends. We say that a smooth function
$\rho:L\rightarrow (0,\infty)$ is a \textit{radius function} if $\rho\circ\phi_i(x)\equiv
r$ on each end. Given any vector
$\boldsymbol{\beta}=(\beta_1,\dots,\beta_{e})\in\R^{e}$, choose a smooth function
$\boldsymbol{\beta}:L\rightarrow\R$ which, on each end $S_i$, restricts to the constant
$\beta_i$. Set $w(x):= \rho(x)^{-\boldsymbol{\beta}(x)}$.
We will refer to either $\boldsymbol{\beta}$ or $w$ as a \textit{weight} on $L$.

Given any metric pair $(E,\nabla)$, the \textit{weighted Sobolev spaces} are
defined by
\begin{equation}\label{eq:weighted_sob}
W^p_{k;\boldsymbol{\beta}}(E):=\mbox{Banach space completion of the space
}\{\sigma\in C^\infty(E):\|\sigma\|_{W^p_{k;\boldsymbol{\beta}}}<\infty\},
\end{equation}
where we use the norm (cf. Remark \ref{rem:tensored_metrics} for notation) 
\begin{equation*}
\|\sigma\|_{W^p_{k;\boldsymbol{\beta}}}:=\left(\Sigma_{j=0}^k\int_L|w\rho^j\nabla^j\sigma|_g^p\,\rho^{-m}\,\mbox{vol}_g\right)^{1/p}=\left(\Sigma_{j=0}^k\int_L|w\nabla^j\sigma|_{\rho^{-2}g\otimes g_E}^p\,\mbox{vol}_{\rho^{-2}g}\right)^{1/p}.
\end{equation*}
The \textit{weighted spaces of $C^k$ sections} are defined by
\begin{equation}\label{eq:weighted_C^k}
C^k_{\boldsymbol{\beta}}(E):=\{\sigma\in C^k(E):
\|\sigma\|_{C^k_{\boldsymbol{\beta}}}<\infty\},
\end{equation}
where we use the norm $\|\sigma\|_{C^k_{\boldsymbol{\beta}}}:=\sum_{j=0}^k
\mbox{sup}_{x\in L}|w\rho^j\nabla^j\sigma|_g$. Equivalently,
$C^k_{\boldsymbol{\beta}}(E)$ is the space of sections $\sigma\in C^k(E)$ such
that $|\nabla^j \sigma|=O(r^{\boldsymbol{\beta}-j})$ as $r\rightarrow 0$
(respectively, $r\rightarrow\infty$) along each CS (respectively, AC) end. These
are also Banach spaces.

To conclude, the  \textit{weighted space of smooth sections} is defined by
\begin{equation*}
C^\infty_{\boldsymbol{\beta}}(E):=\bigcap_{k\geq 0} C^k_{\boldsymbol{\beta}}(E).
\end{equation*}
Equivalently, this is the space of smooth sections such that $|\nabla^j
\sigma|=O(\rho^{\boldsymbol{\beta}-j})$ for all $j\geq 0$. This space has a natural
Fr\'echet structure. 

When $E$ is the trivial $\R$ bundle over $L$ we obtain weighted spaces of
functions on $L$. We usually denote these by $W^p_{k,\boldsymbol{\beta}}(L)$ and
$C^k_{\boldsymbol{\beta}}(L)$. 
In the case of a CS/AC manifold we will sometimes
separate the CS and AC weights, writing
$\boldsymbol{\beta}=(\boldsymbol{\mu},\boldsymbol{\lambda})$ for some
$\boldsymbol{\mu}\in \R^s$ and some $\boldsymbol{\lambda}\in \R^l$. We then
write $C^k_{(\boldsymbol{\mu},\boldsymbol{\lambda})}(E)$ and
$W^p_{k,(\boldsymbol{\mu},\boldsymbol{\lambda})}(E)$.
\end{definition}

\begin{remark}\label{rem:equivalentscaledmetrics}
Let $L$ be a manifold with ends equipped with two conifold metrics $g$, $\hat{g}$. We say that $g$, $\hat{g}$ are \textit{scaled-equivalent} if they
satisfy the following assumptions:
\begin{enumerate}
\item There exists $C_0>0$ such that 
\begin{equation*}
(1/C_0)g\leq \hat{g}\leq C_0 g.
\end{equation*}
\item For all $j\geq 1$ there exists $C_j>0$
such that
\begin{equation*}
|\nabla^j\hat{g}|_{\rho^{-2}g\otimes g}\leq C_j,
\end{equation*}
where $\nabla$ is the Levi-Civita connection defined by $g$ and we are using the notation introduced in Remark \ref{rem:tensored_metrics}.
\end{enumerate}
In this case one can prove that derivatives with respect to the corresponding Levi-Civita connections $\nabla$, $\hat{\nabla}$ coincide up to lower-order terms. The corresponding weighted Sobolev spaces also coincide, with equivalent norms. We refer to \cite{pacini:ics} for details. In particular, Definition \ref{eq:ac_metric} implies that, for $R$ large enough, the metrics $\phi_i^*g$, $\tg_i$ are scaled-equivalent. The analogue is true for CS ends.
\end{remark}

For these weighted spaces the following Sobolev Embedding Theorem holds.

\begin{theorem}\label{thm:embedding}
Let $(L,g)$ be a conifold and $(E,\nabla)$ be a metric pair over $L$.
Assume $k\geq 0$, $l\in\{1,2,\dots\}$ and $p\geq 1$. Set
$p^*_l:=\frac{mp}{m-lp}$. Then, for all
$\boldsymbol{\beta}$,
\begin{enumerate}
\item If $lp<m$ then there exists a continuous embedding
$W^p_{k+l,\boldsymbol{\beta}}(E)\hookrightarrow
W^{p^*_l}_{k,\boldsymbol{\beta}}(E)$. In other words, there exists a ``Sobolev embedding constant'' $C>0$ such that, for all $\sigma\in W^p_{k+l,\boldsymbol{\beta}}(E)$, 
\begin{equation*}
 \|\sigma\|_{W^{p^*_l}_{k,\boldsymbol{\beta}}}\leq C\|\sigma\|_{W^p_{k+l,\boldsymbol{\beta}}}.
\end{equation*}
\item If $lp=m$ then, for all $q\in [p,\infty)$, there exist continuous
embeddings $W^p_{k+l,\boldsymbol{\beta}}(E)\hookrightarrow
W^q_{k,\boldsymbol{\beta}}(E)$.
\item If $lp>m$ then there exists a continuous embedding
$W^p_{k+l,\boldsymbol{\beta}}(E)\hookrightarrow C^k_{\boldsymbol{\beta}}(E)$. 
\end{enumerate}
Furthermore, assume $lp>m$ and $k\geq 0$. Then the corresponding weighted
Sobolev spaces are closed under multiplication, in the following sense. For
any $\boldsymbol{\beta}_1$ and $\boldsymbol{\beta_2}$ there exists $C>0$ such
that, for all $u\in W^p_{k+l,\boldsymbol{\beta_1}}$ and $v\in
W^p_{k+l,\boldsymbol{\beta_2}}$,
\begin{equation*}
\|uv\|_{W^p_{k+l,\boldsymbol{\beta_1}+\boldsymbol{\beta_2}}}\leq
C\|u\|_{W^p_{k+l,\boldsymbol{\beta_1}}}\|v\|_{W^p_{k+l,\boldsymbol{\beta_2}}}.
\end{equation*}
\end{theorem}

\subsection{The Laplace operator on abstract conifolds}\label{ss:laplace_conifolds}

We now summarize some analytic results concerning the Laplace operator on
conifolds.

\begin{definition}\label{def:exceptionalweights}
Let $(\Sigma,g')$ be a compact Riemannian manifold, not necessarily connected. Consider the cone
$C:=\Sigma\times (0,\infty)$ endowed with the conical metric
$\tilde{g}:=dr^2+r^2g'$. Let $\Delta_{\tilde{g}}$ denote the corresponding
Laplace operator acting on functions.

For each component $(\Sigma_j,g_j')$ of $(\Sigma,g')$ and each $\gamma\in\R$,
consider the space of homogeneous harmonic functions 
\begin{equation}\label{eq:ac_harmonicter}
V^j_{\gamma}:=\{r^\gamma\sigma(\theta):
\Delta_{\tilde{g}}(r^{\gamma}\sigma)=0\}.
\end{equation}
Set $m^j(\gamma):=\mbox{dim}(V^j_\gamma)$. One can show that $m^j(\gamma)>0$
if and only if $\gamma$ satisfies the equation
\begin{equation}\label{eq:exceptionalforlaplacian}
\gamma=\frac{(2-m)\pm\sqrt{(2-m)^2+4e_n^j}}{2},
\end{equation}
for some eigenvalue $e_n^j$ of $\Delta_{g_j'}$ on $\Sigma_j$. Given any weight
$\boldsymbol{\gamma}\in \R^e$, we now set $m(\boldsymbol{\gamma}):=\sum_{j=1}^e
m^j(\gamma_j)$. Let $\mathcal{D}\subseteq\R^e$ denote the set of weights
$\boldsymbol{\gamma}$ for which $m(\boldsymbol{\gamma})>0$. We call these the
\textit{exceptional weights} of $\Delta_{\tg}$.
\end{definition}

Let $(L,g)$ be a conifold, asymptotic to a cone $(C,\tg)$ in
the sense of Equations \ref{eq:cs_metric}, \ref{eq:ac_metric}. Roughly speaking, the fact that $g$
is asymptotic to $\tilde{g}$
implies that the Laplace operator $\Delta_g$ is asymptotic to $\Delta_{\tg}$, cf. Remark \ref{rem:equivalentscaledmetrics}.
Applying Definition \ref{def:exceptionalweights} to $C$ defines weights
$\mathcal{D}\subseteq\R^e$: we call these the \textit{exceptional weights} of
$\Delta_g$. This terminology is due to the following result, which indicates that certain aspects of the behaviour of $\Delta_g$ depends only on its asymptotics. 

\begin{theorem}\label{thm:laplaceresults}
Let $(L,g)$ be a conifold with $e$ ends. Let $\mathcal{D}\subseteq\R^e$ denote
the exceptional weights of $\Delta_g$. Then $\mathcal{D}$ is discrete and the Laplace operator 
\begin{equation*}
\Delta_g:W^p_{k,\boldsymbol{\beta}}(L)\rightarrow
W^p_{k-2,\boldsymbol{\beta}-2}(L)
\end{equation*}
is Fredholm if and only if $\boldsymbol{\beta}\notin \mathcal{D}$.
\end{theorem}

The Fredholm index of $\Delta_g$ is constant under small perturbations of the weight. When the weight crosses an exceptional weight, however, the index changes according to a ``change of index formula'' which again depends only on the asymptotics. The following corollary is extracted, as a special case, from \cite{pacini:ics}.

\begin{corollary}\label{cor:laplaceresults}
Let $(L,g)$ be a conifold with non-exceptional weight $\boldsymbol{\beta}$. Let $\beta_i$ denote the value of $\boldsymbol{\beta}$ on the $i$-th end of $L$. Consider the map
\begin{equation*}
\Delta_g:W^p_{k,\boldsymbol{\beta}}(L)\rightarrow
W^p_{k-2,\boldsymbol{\beta}-2}(L).
\end{equation*}
We distinguish three cases:
\begin{enumerate}
 \item Assume $L$ is an AC manifold. If $\boldsymbol{\beta}>2-m$ then this map is surjective. If
$\boldsymbol{\beta}<0$ then this map is injective, so for
$\boldsymbol{\beta}\in (2-m,0)$ it is an isomorphism. 
\item Assume $L$ is a CS manifold with $e$ ends. 
If $\boldsymbol{\beta}>2-m$ then the kernel of this map contains at most the constant functions $\R$ so if, for some $i$, $\beta_i>0$ then this map is injective.
If $\boldsymbol{\beta}>0$ then this map is injective and 
\begin{equation*}
\mbox{dim(Coker($\Delta_g$))}=e+\sum_{0<\boldsymbol{\gamma}<\boldsymbol{\beta}}
m(\boldsymbol{\gamma}),
\end{equation*}
where $m(\boldsymbol{\gamma})$ is as in Definition \ref{def:exceptionalweights}.
\item Assume $L$ is a CS/AC manifold with $s$ CS ends and $l$ AC ends. 
If $\beta_i>2-m$ for all CS ends and $\beta_i<0$ for all AC ends then this map is injective. In particular, if $\boldsymbol{\beta}\in (2-m,0)$ then this map is an isomorphism. 
If instead $\beta_i>0$ for all CS ends and $\beta_i\in (2-m,0)$ for all AC ends then it is injective and 
\begin{equation*}
\mbox{dim(Coker($\Delta_g$))}=s+\sum_{0<\gamma_i<\beta_i}
m^i(\gamma_i),
\end{equation*}
where the sum is over all CS ends and $m^i(\gamma_i)$ is as in Definition \ref{def:exceptionalweights}.
\end{enumerate}
\end{corollary}

\subsection{Deformations of Lagrangian conifolds}\label{ss:conifold_defs}
We now need to review the deformation theory of Lagrangian conifolds, following \cite{joyce:I}, \cite{pacini:sldefsextended}. It is useful to do this in several steps.

\subsubsection*{First case: smooth compact Lagrangian submanifolds} Let $L$ be a compact manifold and $\iota:L\rightarrow \C^m$ a Lagrangian immersion. Let $T^*L$ be the cotangent bundle of $L$, endowed with its natural symplectic structure $\hat{\omega}$. We can identify $L$ with the zero-section in $T^*L$. It is well-known that one can build an open neighbourhood $\mathcal{U}\subset T^*L$ of $L$ and a symplectomorphism
\begin{equation}\label{eq:cptlagmap}
 \Phi_L:\mathcal{U}\rightarrow\C^m
\end{equation}
restricting to $\iota$ on $L$. Let $C^\infty(\mathcal{U})$ denote the space of sections of $T^*L$ whose image lies in $\mathcal{U}$. Up to reparametrization, the moduli space of Lagrangian immersions ``close'' to $\iota$ is then parametrized by the space of closed $1$-forms on $L$ whose image lies in $\mathcal{U}$.

\subsubsection*{Second case: graphs over Lagrangian cones} Assume $L=\Sigma\times(0,\infty)$ and the image of $\iota$ is a cone $\mathcal{C}$ in $\C^m$. We will identify $L$ with $\mathcal{C}$ and denote its generic point by $(\theta,r)$. The generic point in $T^*\mathcal{C}$ is then of the form $(\theta,r,\alpha_1+\alpha_2\,dr)$, where $\alpha_1\in T_\theta^*\Sigma$ and $\alpha_2\in\R$. There is a natural action of $\R^+$ on $T^*\mathcal{C}$ defined by
\begin{equation}\label{eq:action}
\R^+\times T^*\mathcal{C}\rightarrow T^*\mathcal{C},\ \
t\cdot(\theta,r,\alpha_1+\alpha_2\,dr):=(\theta,tr,t^2\alpha_1+t\alpha_2\,dr).
\end{equation}
One can then build an open neighbourhood $\mathcal{U}\subset T^*\mathcal{C}$ of $\mathcal{C}$ which is invariant under this action and a symplectomorphism 
\begin{equation}\label{eq:conelagmap}
 \Phi_{\mathcal{C}}:\mathcal{U}\rightarrow\C^m
\end{equation}
which is $\R^+$-equivariant and restricts to the identity on $\mathcal{C}$. The map $\Phi_{\mathcal{C}}$ is obtained as a perturbation of an explicitly defined map $\Psi_{\mathcal{C}}$ which is linear along the fibres of $T^*\mathcal{C}$; specifically,  
 $\Phi_{\mathcal{C}}=\Psi_{\mathcal{C}}+R$,
where $\Psi_{\mathcal{C}}$ and $R$ are both $\R^+$-invariant and $R$ satisfies 
\begin{equation}\label{eq:R_properties}
|R(\theta,1,\alpha_1,\alpha_2)|=O(|\alpha_1|^2_{g'}+|\alpha_2|^2),\ \ \mbox{as } |\alpha_1|_{g'}+|\alpha_2|\rightarrow 0.
\end{equation}
 These properties lead to the following facts:
\begin{itemize}
\item Consider a fibre $T^*_{(\theta,r)}\mathcal{C}$ of $T^*\mathcal{C}$, endowed with the metric induced by $g$. Set $\mathcal{U}_{(\theta,r)}:=\mathcal{U}\cap T^*_{(\theta,r)}\mathcal{C}$. One can then assume that $\mathcal{U}_{(\theta,r)}$ is an open ball in $T^*_{(\theta,r)}\mathcal{C}$ of radius $Cr$, for some $C>0$. 
\item Choose rates $(\mu,\lambda)$. Let $C^\infty_{(\mu-1,\lambda-1)}(\mathcal{U})$ denote the space of sections of $T^*\mathcal{C}$ whose image lies in $\mathcal{U}$. Up to reparametrization, the set of Lagrangian immersions ``close'' to $\iota$ and asymptotic to $\mathcal{C}$ with rate $(\mu,\lambda)$ is then parametrized by the space of closed $1$-forms in $C^\infty_{(\mu-1,\lambda-1)}(\mathcal{U})$.
\end{itemize}
More generally, assume a Lagrangian conifold in $\C^m$ with rate $(\mu,\lambda)$ is obtained via an immersion $\iota:=\Phi_{\mathcal{C}}\circ\alpha:\mathcal{C}\rightarrow \C^m$ for some closed $1$-form $\alpha$ in $C^\infty_{(\mu-1,\lambda-1)}(\mathcal{U})$. Using the symplectomorphism
\begin{equation}\label{eq:translation}
 \tau_{\alpha}:T^*\mathcal{C}\rightarrow T^*\mathcal{C}, \ \ \tau_{\alpha}(\theta,r,\eta):=(\theta,r,\eta+\alpha(\theta,r)),
\end{equation}
we can define a Lagrangian
neighbourhood for $\iota$ by setting
\begin{gather}\label{eq:Lgraphmap}
\nonumber\tau_{\alpha}^{-1}(\mathcal{U}):=\{(\theta,r,\eta)\in T^*\mathcal{C}:(\theta,r,\eta+\alpha)\in\mathcal{U}\},\\
 \Phi_L:=\Phi_\mathcal{C}\circ\tau_\alpha:\tau_{\alpha}^{-1}(\mathcal{U})\subset T^*\mathcal{C}\rightarrow\C^m.
\end{gather}
Notice that the zero-section is contained in $\tau_{\alpha}^{-1}(\mathcal{U})$ and that $\Phi_{L}$, restricted to the zero-section, coincides with $\iota$. The Lagrangian deformations of $\iota$, asymptotic to $\mathcal{C}$ with rate $(\mu,\lambda)$, are then parametrized by closed $1$-forms in $C^\infty_{(\mu-1,\lambda-1)}(\tau_{\alpha}^{-1}(\mathcal{U}))$.

\subsubsection*{Third case: Lagrangian conifolds} In general, let $\iota:L\rightarrow \C^m$ be a Lagrangian conifold. For simplicity, let us assume that it is an AC Lagrangian submanifold, \textit{i.e.} that it has only AC ends $S_1,\dots,S_e$ with centers $p_i$ and rate $\boldsymbol{\lambda}$. Set $K:=L\setminus\cup S_i$. We will also simplify the notation
by identifying $\Sigma_i\times [R,\infty)$ with $\overline{S_i}$ via the diffeomorphisms
$\phi_i$. We can thus write 
\begin{equation*}
L=K\cup \left(\Sigma_i\times [R,\infty)\right),
\end{equation*}
where we identify the boundary of $K$ with
$\cup\left(\Sigma_i\times\{R\}\right)$. The map $\Phi_{\mathcal{C}_i}+p_i$ identifies $\iota(S_i)$ with the graph
$\Gamma(\alpha_i)$, for some (locally defined) closed 1-form $\alpha_i$ in
$C^\infty_{\lambda_i-1}(\mathcal{U})$. We then set
\begin{equation*}
\Phi_{S_i}:=\Phi_{\mathcal{C}_i}\circ \tau_{\alpha_i}+p_i:\tau_{\alpha_i}^{-1}(\mathcal{U})\subset T^*(\Sigma_i\times [R,\infty))\rightarrow\C^m. 
\end{equation*}
It is possible to interpolate between this data, obtaining an open neighbourhood $\mathcal{U}\subset T^*L$ of $L$ and a symplectomorphism
\begin{equation*}\label{eq:lagmap}
\Phi_L:\mathcal{U}\subset T^*L\rightarrow \C^m
\end{equation*}
which restricts to $\iota$ along $L$. By construction we can assume that the ``radius'' of $\mathcal{U}$, in the sense defined following Equation \ref{eq:R_properties}, is  linear with respect to $r$ on each end $S_i$. A similar construction works for general Lagrangian conifolds. Up to reparametrization, the set of Lagrangian immersions ``close'' to $\iota$ and asympotic to the same cones, with the same centers and rate $(\boldsymbol{\mu},\boldsymbol{\lambda})$, is then parametrized by the space of closed $1$-forms in $C^\infty_{(\boldsymbol{\mu}-1,\boldsymbol{\lambda}-1)}(\mathcal{U})$.

\subsubsection*{Lagrangian conifolds with moving singularities} Notice that our restrictions on $\mu_i$ imply that the above deformations always fix the position of the singularities and the corresponding asymptotic cones in $\C^m$. It is useful to also take into account deformations which allow the singularities to translate in $\C^m$ and the cones to rotate. The correct set-up for doing
this is as follows. 

Assume $\iota:L\rightarrow\C^m$ is a CS/AC Lagrangian submanifold with singularities
$\{x_1,\dots,x_s\}$. Define
\begin{equation*}
P:=\{(p,\upsilon):p\in \C^m,\ \upsilon\in \unitary m\}.
\end{equation*}
$P$ is a $\unitary m$-principal fibre bundle over $\C^m$ with the action 
$$\unitary m\times P\rightarrow P,\ \ M\cdot (p,\upsilon):=(p,\upsilon\circ
M^{-1}).$$
As such, $P$ is a smooth manifold of dimension $m^2+2m$.

Our aim is to use one copy of $P$ to parametrize the location of each singular
point $p_i=\iota(x_i)\in \C^m$ and the direction of the corresponding cone in $\C^m$. The element $(p_i,Id)$ will correspond to the initial positions. 
As we are interested only in small deformations of $L$ we can restrict
our attention to a small open neighbourhood of the pair $(p_i,Id)\in P$. In general the
$\cone_i$ will have some symmetry group $G_i\subset \unitary m$, \textit{i.e.}
the action of this $G_i$ will leave the cone fixed. To ensure that we have no
redundant parameters we must therefore further restrict our attention to a
\textit{slice} of our open neighbourhood, \textit{i.e.} a smooth submanifold
transverse to the orbits of $G_i$. We denote this slice $\E_i$: it is a subset
of $P$ containing $(p_i,Id)$ and of dimension $m^2+2m-\mbox{dim}(G_i)$.
We then set $\E:=\E_1\times\dots\times\E_s$ and
$e:=((p_1,Id),\dots,(p_s,Id))\in \E$.

It is possible to choose a family of CS/AC Lagrangian immersions
$\iota_{\tilde{e}}:L\rightarrow\C^m$ parametrized by
$\tilde{e}=((\tilde{p}_1,\tilde{v}_1),\dots,(\tilde{p}_s,\tilde{v}_s))\in
\E$ with the following features:
\begin{itemize}
\item $\iota_e=\iota$.
\item For each $\tilde{e}$, $\iota_{\tilde{e}}(x_i)=\tilde{p}_i$ with asymptotic cone $\tilde{v}_i(\mathcal{C}_i)$. Furthermore, $\iota_{\tilde{e}}=\iota$ outside a neighbourhood of the singularities.
\item Choose $\mathcal{U}$ and $\Phi_L$ as in Equation \ref{eq:lagmap}. Then, for each $\tilde{e}$, there are symplectomorphisms $\Phi_{L}^{\tilde{e}}:\mathcal{U}\rightarrow \C^m$ which restrict to $\iota_{\tilde{e}}$ on $L$ and such that $\Phi_L^{e}=\Phi_L$. 
\end{itemize}

The final result is that, after such a choice and up to reparametrization, the
set of CS/AC Lagrangian immersions ``close'' to $\iota$ with rate
$(\boldsymbol{\mu},\boldsymbol{\lambda})$ and moving singularities can be parametrized in terms of
pairs $(\tilde{e}, \alpha)$ where $\tilde{e}\in\E$ and $\alpha$ is a closed 1-form
on $L$ belonging to the space $C^\infty_{(\boldsymbol{\mu}-1,\boldsymbol{\lambda}-1)}(\mathcal{U})$. 

\begin{remark}\label{rem:phitilde}
In calculations it is useful to extend $\tilde{e}$ to a compactly-supported symplectomorphism of $\C^m$ which coincides with the corresponding element of $\unitary m\times\C^m$ in a neighbourhood of each point $p_i$. We can then define $\Phi^{\tilde{e}}_L:=\tilde{e}\circ\Phi_L$. This point of view makes certain operations more explicit. For future reference, we give the following example.

Let $\alpha\in\Lambda^k(\C^m)$ be a differential form on $\C^m$. Assume we want to study the smoothness of the pull-back operation
\begin{equation}\label{eq:pullback}
 \tilde{\E}\rightarrow\Lambda^k(\mathcal{U}), \ \ \tilde{e}\mapsto (\Phi^{\tilde{e}}_L)^*\alpha.
\end{equation}
Let $\mathcal{G}$ denote the infinite-dimensional Lie group of compactly-supported diffeomorphisms of $\C^m$, endowed with its natural Fr\'{e}chet structure. By construction, the tangent space $T_{Id}\mathcal{G}$ at the identity is the vector space of smooth compactly-supported vector fields on $\C^m$.
At any other point $\phi\in\mathcal{G}$, we can then identify $T_\phi\mathcal{G}$ as follows:
\begin{equation*}
T_{\phi}\mathcal{G}=\{X\circ\phi:X\in T_{Id}\mathcal{G}\}.
\end{equation*}
Endow $\Lambda^k(\C^m)$ with its natural Fr\'{e}chet structure. With respect to these structures, the pull-back operation on $k$-forms,
\begin{equation*}
\mathcal{G}\rightarrow\Lambda^k(\C^m),\ \ \phi\mapsto\phi^*\alpha,                                        
\end{equation*}
is a smooth map. Its derivatives can be written in terms of the Lie derivatives of $\alpha$. We will think of $\tilde{\E}$ as a finite-dimensional submanifold of $\mathcal{G}$, so that the restricted map is also smooth. Composing with $\Phi_L^*$, we obtain the smoothness of the map in Equation \ref{eq:pullback}.
\end{remark}

\begin{remark}\label{rem:defs_same} Let $(L,\iota)$ be a Lagrangian submanifold in $\C^m$. Recall from Example \ref{ex:smooth_is_sing} that any smooth point can be labelled as a singularity. The deformation theory presented above is set up so that both points of view allow the same degree of flexibility: in either case the point in question can be translated and its tangent plane can be rotated. Thus, from the deformation theory point of view it makes no difference whether we label self-intersection points as smooth or singular. 
\end{remark}


\section{Rescaled Lagrangian conifolds in $\C^m$}\label{s:rescaled_lag_cons}

Let $(L,g)$ be an abstract Riemannian conifold. Rescaling the metric yields new manifolds $(L,t^2g)$. The initially given diffeomorphisms between ends and asymptotic cones can also be rescaled, so each $(L,t^2g)$ is again a conifold: we refer to \cite{pacini:ics} for details. Our goal here is to define an analogous rescaling procedure for the category of Lagrangian conifolds in $\C^m$.

Let $\R^+$ act on $\C^m$ by dilations: $t\cdot x:=tx$. The induced action on
forms is such that $t^*\tg=t^2\tg$, $t^*\tilde{\omega}=t^2\tilde{\omega}$. This implies that if
$\iota:L\rightarrow\C^m$ is a Lagrangian submanifold of $\C^m$ then the rescaled maps $t\iota:L\rightarrow\C^m$
are also Lagrangian.

\begin{lemma}\label{l:rescaledlag}
 Let $\iota:L\rightarrow \C^m$ be a Lagrangian conifold with cones $\mathcal{C}_i$,
rate $(\boldsymbol{\mu},\boldsymbol{\lambda})$ and centers $p_i$. Let $\phi_i$ be the diffeomorphism corresponding to $\iota$ and the AC end $S_i$, as in Definition \ref{def:aclagsub}. Then, for $r\rightarrow\infty$, the diffeomorphism
\begin{equation*}
 \phi_{t,i}:\Sigma_i\times [tR,\infty)\rightarrow \overline{S_i},\ \
\phi_{t,i}(\theta,r):=\phi_i(\theta,r/t)
\end{equation*}
has the property
\begin{equation*}
 |\tilde{\nabla}^k(t\iota\circ\phi_{t,i}-(\iota_i+tp_i))|_\tg(\theta,r)=t^{2-\lambda_i}O(r^{\lambda_i-1-k}).
\end{equation*}
This property allows us to select it as the preferred diffeomorphism corresponding to $t\iota$ and the AC end $S_i$, as in Definition \ref{def:aclagsub}.

Analogously, let $\phi_i$ be the diffeomorphism corresponding to $\iota$ and the CS end $S_i$. Then, for $r\rightarrow 0$, the diffeomorphism
\begin{equation*}
 \phi_{t,i}:\Sigma_i\times (0,t\epsilon]\rightarrow \overline{S_i}\setminus\{x_i\},\
\ \phi_{t,i}(\theta,r):=\phi_i(\theta,r/t)
\end{equation*}
has the property
\begin{equation*}
 |\tilde{\nabla}^k(t\iota\circ\phi_{t,i}-(\iota_i+tp_i))|_\tg(\theta,r) 
=t^{2-\mu_i}O(r^{\mu_i-1-k}).
\end{equation*}
We can thus use it to parametrize the CS ends of $t\iota$, as in Definition \ref{def:cslagsub}.

We conclude that $t\iota:L\rightarrow\C^m$ is a Lagrangian conifold with the same cones and rate as $(L,\iota)$, and with centers $tp_i$.
\end{lemma}
\begin{proof}
 On any AC end $S_i$,
\begin{align*}
 |\tilde{\nabla}^k(t\iota\circ\phi_{t,i}-(\iota_i+tp_i))|_\tg(\theta,r)
 &=|\tilde{\nabla}^k(t\iota\circ\phi_{t,i}-(\iota_i+tp_i))|_{t^2(t^{-1})^*\tg}(\theta,r)\\
 &=t^{-k}|\tilde{\nabla}^k(t\iota\circ\phi_{t,i}-(\iota_i+tp_i))|_{(t^{-1})^*\tg}(\theta,r)\\
 &=t^{-k}\left((t^{-1})^*(|t^*\tilde{\nabla}^k(t\iota\circ\phi_{t,i}-(\iota_i+tp_i))|_\tg)\right)(\theta,r)\\
&=t^{-k}(|\tilde{\nabla}^kt^*(t\iota\circ\phi_{t,i}-(\iota_i+tp_i))|_\tg)(\theta,r/t)\\
&=t^{-k}|\tilde{\nabla}^k(t\iota\circ\phi_i-(t\iota_i+tp_i))|_\tg(\theta,r/t)\\
&=t^{1-k}O((r/t)^{\lambda_i-1-k})\\
&=t^{2-\lambda_i}O(r^{\lambda_i-1-k}).
\end{align*}
The proof for CS ends is analogous.
\end{proof}

One can check that the above lemma implies for instance that, as $r\rightarrow\infty$,
\begin{equation*}
|\tnabla^k(\phi_{t,i}^*g-\tg)|=t^{2-\lambda_i}O(r^{\lambda_i-2-k}).
\end{equation*}
We now want to define Lagrangian neighbourhoods for the conifolds $t\iota$. We will do this in various stages, as
follows.

\subsubsection*{First case: smooth compact Lagrangian submanifolds}
\label{ss:rescaledcptlagdefs}

It is useful to start by considering how Lagrangian neighbourhoods can be built
for $t\iota$ when $L$ is smooth and compact.

Consider the action of $\R^+$ on the manifold $T^*L$ defined by
\begin{equation}\label{eq:fibreaction}
t\cdot(x,\alpha):=(x,t^2\alpha).
\end{equation}
It is simple to check that the induced action on forms is such that
$t^*\hat{\omega}=t^2\hat{\omega}$. Let $\mathcal{U}$, $\Phi_L$ be as in Equation \ref{eq:cptlagmap}.
Now define 
\begin{equation}\label{eq:tLcptlagmap}
 \Phi_{t,L}:t\mathcal{U}\rightarrow\C^m,\ \ \Phi_{t,L}:=t\Phi_Lt^{-1}.
\end{equation}
One can check that $\Phi_{t,L}$ is a symplectomorphism and that, restricted to
the zero-section $L$, it coincides with $t\iota$. Thus Equation \ref{eq:tLcptlagmap}
defines a Lagrangian neighbourhood for $t\iota$. 

Notice that if $|\alpha|_g=C$ then $|t^2\alpha|_{t^2g}=Ct$. In this sense, the ``radius'' of $t\mathcal{U}$ is linear in $t$.

\subsubsection*{Second case: graphs over Lagrangian
cones}\label{ss:rescaledconelagdefs}

Let $\mathcal{C}$ be a Lagrangian cone in $\C^m$ and
$\Phi_\mathcal{C}:\mathcal{U}\subset T^*\mathcal{C}\rightarrow \C^m$ be the
symplectomorphism defined in Equation \ref{eq:conelagmap}. Assume $L$ is a
Lagrangian conifold in $\C^m$ with rate
$(\mu,\lambda)$ obtained via an immersion 
$\iota:=\Phi_\mathcal{C}\circ\alpha:\mathcal{C}\rightarrow\C^m$, for some closed
1-form $\alpha=\alpha_1(\theta,r)+\alpha_2(\theta,r)\, dr$ in
$C^\infty_{(\mu-1,\lambda-1)}(\mathcal{U})$. Set
\begin{equation}\label{eq:alphat}
\alpha_t(\theta,r):=t^2\alpha_1(\theta,r/t)+t\alpha_2(\theta,r/t)\, dr.
\end{equation}
Then
\begin{align*}
 \Phi_\mathcal{C}\circ\alpha_t(\theta,r)
&=\Phi_\mathcal{C}(\theta,r,t^2\alpha_1(\theta,r/t)+t\alpha_2(\theta,r/t)\,dr)\\
&=\Phi_\mathcal{C}(t\cdot(\theta,r/t,\alpha_1(\theta,r/t)+\alpha_2(\theta,r/t)\,
dr))\\
&=t\Phi_\mathcal{C}(\theta,r/t,\alpha_1(\theta,r/t)+\alpha_2(\theta,r/t)\,
dr)=t\iota(\theta,r/t),
\end{align*}
where we use the equivariance of $\Phi_\mathcal{C}$ with respect to the $\R^+$-action defined in Equation \ref{eq:action}. This shows that, up to dilations, the rescaled
Lagrangian $t\iota$ can be written as $\Phi_\mathcal{C}\circ\alpha_t$. The results
of Section \ref{ss:conifold_defs} and Lemma \ref{l:rescaledlag} imply that
$\alpha_t$ is a closed 1-form in the space
$C^\infty_{(\mu-1,\lambda-1)}(\mathcal{U})$, though
this could also be checked by direct computation. We note, in passing, that
\begin{align*}
 |\alpha_t|^2_{\tg}(\theta,r)&=t^4|\alpha_1(\theta,r/t)|^2_{r^2g'}+t^2|\alpha_2(\theta,r/t)|^2\\
&=t^2\left(|\alpha_1(\theta,r/t)|^2_{(r/t)^2g'}+|\alpha_2(\theta,r/t)|^2\right)\\
&=t^2|\alpha|^2_\tg(\theta,r/t).
\end{align*}
Analogously to Section \ref{ss:conifold_defs}, we can define a Lagrangian neighbourhood for $t\iota$ by setting
\begin{equation*}\label{eq:tLgraphmap}
 \Phi_{t,L}:=\Phi_\mathcal{C}\circ\tau_{\alpha_t}:\tau_{\alpha_t}^{-1}(\mathcal{U})\subset T^*\mathcal{C}\rightarrow\C^m.
\end{equation*}
 
\subsubsection*{Third case: Lagrangian conifolds}\label{ss:rescaledaccslagdefs}

Let $\iota:L\rightarrow\C^m$ be a Lagrangian conifold. As in Section \ref{ss:conifold_defs}, for simplicity we assume that $L$
is an AC Lagrangian submanifold and we use the notation/identifications
\begin{equation*}\label{eq:Ldecomp}
L=K\cup \left(\Sigma_i\times [R,\infty)\right),
\end{equation*}
where we identify the boundary of $K$ with
$\cup\left(\Sigma_i\times\{R\}\right)$. Recall from Section \ref{ss:conifold_defs} that $\Phi_L$ restricts to $\Phi_{S_i}=\Phi_{\mathcal{C}_i}\circ \tau_{\alpha_i}+p_i$ on the bundle $\mathcal{U}$ over $\Sigma_i\times [R,\infty)$ and to some $\Phi_K$ on the bundle $\mathcal{U}$ over $K$. Notice that
\begin{equation}\label{eq:phicoincides}
\Phi_K=\Phi_{S_i} \ \  \mbox{on $T^*(\Sigma_i\times \{R\})$}.
\end{equation}
Let us now set
\begin{equation*}\label{eq:tLdecomp}
 tL:=K\cup\left(\Sigma_i\times [tR,\infty)\right),
\end{equation*}
again making an identification along the boundary: given that $\partial
K=\cup\left(\Sigma_i\times\{R\}\right)$, this requires identifying
$\cup\left(\Sigma_i\times\{R\}\right)\simeq\cup\left(\Sigma_i\times\{tR\}
\right)$ via rescaling. Define
\begin{equation*}\label{eq:tLdecompmap}
 \Phi_{t,K}:=t\Phi_Kt^{-1}: T^*K\rightarrow \C^m.
\end{equation*}
Define $\alpha_{t,i}$ from $\alpha_i$
as above and set
\begin{equation*}\label{eq:tLdecompmapbis}
\Phi_{t,S_i}:=\Phi_{\mathcal{C}_i}\circ\tau_{\alpha_{t,i}}+tp_i:\tau_{\alpha_{t,i}}^{-1}(\mathcal{U})\subset T^*(\Sigma_i\times
[tR,\infty))\rightarrow\C^m.
\end{equation*}
We now want to show that the maps
$\Phi_{t,K}$, $\Phi_{t,S_i}$ coincide on the common boundary. This is a simple
calculation, requiring only a bit of care to take into account the different
$\R^+$-actions used in Equations \ref{eq:action} and
\ref{eq:fibreaction}, as follows.

Choose $x\in\partial K$ and $\eta_x\in T^*_x(tL)$. On the one hand we are using
the identification $\partial K\simeq\cup\left(\Sigma_i\times\{R\}\right)$, so
for some $\theta\in\Sigma_i$ we get $x\simeq (\theta,R)$ and
$(x,\eta_x)\simeq(\theta,R,\eta_1(\theta,R)+\eta_2(\theta,R)\,dr)$. To simplify
the notation, set $\alpha:=\alpha_i$. Then, using Equation
\ref{eq:phicoincides}, 
\begin{align*}
\Phi_{t,K}(x,\eta_x)
&=t\Phi_K(\theta,R,t^{-2}\eta_1(\theta,R)+t^{-2}\eta_2(\theta,R)dr)\\
&=t\Phi_{\mathcal{C}_i}\circ\tau_\alpha(\theta,R,t^{-2}\eta_1(\theta,R)+t^{-2}
\eta_2(\theta,R)dr)+tp_i\\
&=\Phi_{\mathcal{C}_i}t(\theta,R,(\alpha_1+t^{-2}\eta_1)_{|(\theta,R)}
+(\alpha_2+t^{-2}\eta_2)_{|(\theta,R)}dr)+tp_i\\
&=\Phi_{\mathcal{C}_i}(\theta,tR,(t^2\alpha_1+\eta_1)_{|(\theta,R)}
+(t\alpha_2+t^{-1}\eta_2)_{|(\theta,R)}dr)+tp_i.
\end{align*}
On the other hand we are using the identification $\partial
K\simeq\cup\left(\Sigma_i\times\{tR\}\right)$. The two identifications are
related by a dilation so now
$(x,\eta_x)\simeq(\theta,tR,\eta_1(\theta,R)+t^{-1}\eta_2(\theta,R)\,dr)$. Set
$\alpha_t:=\alpha_{t,i}$. Then
\begin{align*}
 \Phi_{t,S_i}(x,\eta_x)
&=\Phi_{\mathcal{C}_i}\circ\tau_{\alpha_t}(\theta,tR,\eta_1(\theta,R)+t^{-1}
\eta_2(\theta,R)\,dr)+tp_i\\
&=\Phi_{\mathcal{C}_i}(\theta,tR,(t^2\alpha_1+\eta_1)_{|(\theta,R)}
+(t\alpha_2+t^{-1}\eta_2)_{|(\theta,R)}dr)+tp_i.
\end{align*}
This calculation proves that the two maps can be glued together along the common
boundary, obtaining a well-defined symplectomorphism
\begin{equation}\label{eq:tLmap}
\Phi_{t,L}:\mathcal{U}_t\subset T^*(tL)\rightarrow \C^m.
\end{equation}
This map restricts to $t\iota$ on the zero-section $tL$ because this is true
for both maps $\Phi_{t,K}$, $\Phi_{t,S_i}$. We have thus constructed a
Lagrangian neighbourhood for any rescaled AC Lagrangian submanifold $t\iota$. The
construction for rescaled Lagrangian conifolds is similar.


\section{Connect sums of Lagrangian conifolds}\label{s:lagr_sum}

Let $(L,g)$ be an abstract Riemannian conifold. Choose a singular point $x\in
\bar{L}$ with cone $C$. Assume we are given a second
conifold $(\hat{L},\hat{g})$ containing an AC end asymptotic to the same cone
$C$. For small $t>0$ one can
build a family of new conifolds $(L_t,g_t)$ by gluing the rescaled
manifold $(\hat{L},t^2\hat{g})$ into a neighbourhood of $x$. Notice that the construction is not symmetric in its
initial data $L$, $\hat{L}$: in particular only one of these is rescaled. 
If the gluing is done appropriately then one can show that the family $(L_t,g_t)$ satisfies certain analytic and elliptic estimates uniformly in $t$. An
analogous construction holds for any number of singularities $x_i\in\bar{L}$,
provided $\hat{L}$ contains the same number of appropriate AC ends. We refer to \cite{pacini:ics} for details.

We now want to define an analogous ``connect sum'' procedure for the category of Lagrangian conifolds.
The main additional assumption
required by this construction will be that the AC ends of $\hat{L}$ used in the gluing have rate
$\hat{\boldsymbol{\lambda}}<0$.

\begin{definition}\label{def:marking}
Let $\iota:L\rightarrow\C^m$ be a Lagrangian conifold. Let $S$ denote
the union of its ends. A subset $S^*$ of $S$ defines a \textit{marking} on $L$.
We can then write $S=S^*\amalg S^{**}$, where $S^{**}$ is the complement
of $S^*$. We say $S^*$ is a \textit{CS-marking} if all ends in $S^*$ are CS; it
is an \textit{AC-marking} if all ends in $S^*$ are AC. We will denote by $d$
the number of ends in $S^*$.
\end{definition}

\begin{definition}\label{def:compatible}
 Let $(L,\iota,S^*)$ be a CS-marked Lagrangian conifold in $\C^m$, with induced metric $g$. Let
$\Sigma^*$, $\mathcal{C}^*$ denote the
links and cones corresponding to $S^*$. Given any end $S_i\subseteq S^*$ let
$\phi_i:\Sigma_i\times (0,\epsilon]\rightarrow \overline{S_i}\setminus\{x_i\}$ be the corresponding
diffeomorphism.

Let $(\hat{L},\hat{\iota},\hat{S}^*)$ be an AC-marked Lagrangian 
conifold in $\C^m$, with induced metric $\hat{g}$. Let $\hat{\Sigma}^*$, $\hat{\mathcal{C}}^*$,
$\hat{\phi}_i:\hat{\Sigma}_i\times
[\hat{R},\infty)\rightarrow \overline{\hat{S}_i}$ denote the corresponding
links, cones and diffeomorphisms, as above. 

We say that $L$ and $\hat{L}$ are \textit{compatible} if they satisfy the
following assumptions:
\begin{enumerate}
\item $\mathcal{C}^*=\hat{\mathcal{C}}^*$. Up to relabelling the ends, we may assume that  $\mathcal{C}_i=\hat{\mathcal{C}}_i$.
\item $\hat{R}<\epsilon$. Up to a preliminary rescaling of $\hat{L}$, one can always assume this is true. We can then identify
appropriate
subsets of $S^*$ and $\hat{S}^*$ via the maps $\hat{\phi}_i\circ\phi_i^{-1}$. 
\item On each marked AC end,
the metrics $\hat{\phi}_i^*\hat{g}$ and $\tg_i$ are scaled-equivalent in the
sense of Definition \ref{rem:equivalentscaledmetrics}. Analogously, on each
marked CS end,
the metrics $\phi_i^*g$ and $\tg_i$ are scaled-equivalent. Again, up to a preliminary rescaling one can always assume this is true. 
\item If $\hat{S}_i,\hat{S}_j\in \hat{S}^*$ belong to the same connected component of $\hat{L}$ then the corresponding ends $S_i,S_j\in S^*$ satisfy $\iota(x_i)=\iota(x_j)$.
\item For each AC end $\hat{S}_i\in \hat{S}^*$, the corresponding center is $p_i=0$ and the corresponding rate $\hat{\lambda}_i$ satisfies $\hat{\lambda}_i<0$.
\end{enumerate}
If
$L$ is weighted via $\boldsymbol{\beta}$ and $\hat{L}$ is weighted via
$\hat{\boldsymbol{\beta}}$ we further require that, on $S_i\in S^{*}$ and $\hat{S}_i\in \hat{S}^{*}$, the
corresponding constants satisfy
$\beta_i=\hat{\beta}_i$ and that $\hat{\beta}_i=\hat{\beta}_j$ if $\hat{S}_i$ and $\hat{S}_j$ are marked ends in the same connected component of $\hat{L}$.
\end{definition}

Let $\iota:L\rightarrow\C^m$ be a Lagrangian conifold in $\C^m$ with metric $g$ and CS-marking $S^*$. Let $\hat{\iota}:\hat{L}\rightarrow\C^m$ be a second Lagrangian conifold with metric $\hat{g}$ and AC-marking $\hat{S}^*$. Let $(\rho,\boldsymbol{\beta})$, respectively
$(\hat{\rho},\hat{\boldsymbol{\beta}})$, be corresponding radius functions and
weights. Assume $L$, $\hat{L}$ are compatible. Let $d$ denote the number of marked ends. Choose parameters $\boldsymbol{t}=(t_1,\dots,t_d)>0$ sufficiently
small. We assume that $\boldsymbol{t}$ is compatible with the
decomposition of $\hat{L}$ into its connected components: specifically, that
$t_i=t_j$ if $\hat{S}_i$ and $\hat{S}_j$ belong to the same connected component
of $\hat{L}$. We then define and study the
\textit{parametric Lagrangian connect sum}
of $(L,\iota)$ and $(\hat{L},\hat{\iota})$ via the following steps. 
\subsubsection*{The abstract manifolds} We set
\begin{equation*}
L_{\boldsymbol{t}}:=(\hat{L}\setminus\hat{S}^*)\cup
(\cup_{\Sigma_i\subseteq\Sigma^*}\Sigma_i\times[t_i\hat{R},\epsilon])\cup
(L\setminus S^*), 
\end{equation*}
where the components of the boundary of $\hat{L}\setminus\hat{S}^*$ are
identified with
the $\Sigma_i\times\{t_i\hat{R}\}$ via maps $\hat{\phi}_{t_i,i}$ defined as in
Lemma
\ref{l:rescaledlag} and the components of the boundary of $L\setminus S^*$
are identified
with the $\Sigma_i\times\{\epsilon\}$ via the maps $\phi_i$. We call $\Sigma_i\times[t_i\hat{R},\epsilon]$ the
\textit{neck regions} of $L_{\boldsymbol{t}}$.

Notice that, using the notation of Definition \ref{def:cslagsub}, the corresponding manifold $\bar{L}_{\boldsymbol{t}}$  may contain a finite number of singularities inherited from either $S^{**}$ or $\hat{S}^{**}$.
\subsubsection*{The Lagrangian immersions} We now want to build a Lagrangian immersion
$\iota_{\boldsymbol{t}}:L_{\boldsymbol{t}}\rightarrow\C^m$. On $L\setminus S^*$ we simply let $\iota_{\boldsymbol{t}}$ coincide with the restriction of
$\iota$. On the $i$-th connected component of
$\hat{L}\setminus \hat{S}^*$ we let $\iota_{\boldsymbol{t}}$ coincide
with the restriction of $t_i\hat{\iota}+\iota(x_i)$. On $\cup_{\Sigma_i\subseteq\Sigma^*}\Sigma_i\times[t_i\hat{R},\epsilon]$  we need to interpolate between these maps. We do so using the assumption $\hat{\lambda}_i<0$, as follows.

As in Section \ref{ss:rescaledaccslagdefs}, we can identify $S_i$ with $\Sigma_i\times (0,\epsilon)$. Then $\iota(\theta,r)=\Phi_{\mathcal{C}_i}(\theta,r,\alpha_i(\theta,r))+\iota(x_i)$, for some 
closed 1-form $\alpha_i$ in $C^\infty_{\mu_i-1}(\mathcal{U})$. As usual, let us write $\alpha_i=\alpha_{i,1}+\alpha_{i,2}\,dr$.
Using the fact that $\mu_i>0$ and that we are restricting
our attention to $r\in (0,\epsilon)$, one can prove that $\alpha_i$ is exact: $\alpha_i=dA_i$ where
$A_i(\theta,r):=\int_0^r\alpha_{i,2}(\theta,\rho)\,d\rho\in C^\infty_{\mu_i}(\mathcal{C}_i)$, cf.  \cite{pacini:sldefs} for details.
Likewise, identifying $\hat{S}_i$ with $\Sigma_i\times (\hat{R},\infty)$, we obtain $\hat{\iota}(\theta,r)=\Phi_{\mathcal{C}_i}(\theta,r,\hat{\alpha}_i(\theta,r))$, for some closed 1-form
$\hat{\alpha}_i=\hat{\alpha}_{i,1}+\hat{\alpha}_{i,2}\,dr\in C^\infty_{\hat{\lambda}_i-1}(\mathcal{U})$. As above, our assumption $\hat{\lambda}_i<0$ can be used to prove that $\hat{\alpha}_i=d\hat{A}_i$, where
$\hat{A}_i(\theta,r):=-\int_r^\infty\hat{\alpha}_{i,2}(\theta,\rho)\,d\rho\in
C^\infty_{\hat{\lambda}_i}(\mathcal{C}_i)$.

Define $\hat{\alpha}_{t_i}$ from $\hat{\alpha}_i$ as in Equation
\ref{eq:alphat}. More explicitly, since $\hat{\alpha}_{i,1}=d_{\Sigma_i}\hat{A}_i$ and $\hat{\alpha}_{i,2}=\frac{d}{dr}\hat{A}_i$,
\begin{equation*}
 \hat{\alpha}_{t_i}(\theta,r)=t_i^2d_{\Sigma_i}\hat{A}_{i|(\theta,r/t_i)}+t_i\frac{d}{dr}\hat{A}_{i|(\theta,r/t_i)}dr=d(t_i^2\hat{A}_i(\theta,r/t_i)).
\end{equation*}
Write $\hat{\alpha}_{t_i}=\hat{\alpha}_{t_i,1}+\hat{\alpha}_{t_i,2}\,dr$. If we define
\begin{align}
\label{eq:A_t=t^2A}
\hat{A}_{t_i}(\theta,r)
&:=-\int_r^\infty\hat{\alpha}_{t_i,2}(\theta,\rho)\,d\rho\\
&=-\int_r^\infty t_i\hat{\alpha}_2(\theta, \rho/t_i)\,d\rho\nonumber\\
&=-\int_{r/t_i}^\infty
t_i^2\hat{\alpha}_2(\theta,\rho)\,d\rho=t_i^2\hat{A}_i(\theta,
r/t_i)\nonumber,
\end{align}
we conclude
$\hat{\alpha}_{t_i}(\theta,r)=d\hat{A}_{t_i}(\theta,r)$.

Choose
$\tau\in (0,1)$. If the $t_i$ are sufficiently small, we find
$t_i\hat{R}<t_i^\tau<2t_i^\tau<\epsilon$. Fix a monotone increasing function $G:(0,\infty)\rightarrow [0,1]$ such that
$G(r)\equiv 0$ on $(0,1]$ and $G(r)\equiv 1$ on $[2,\infty)$. Define a function $A_{t_i}:\Sigma_i\times[t_i\hat{R},\epsilon]\rightarrow\R$ by setting
\begin{equation}
\label{eq:A_t}
 A_{t_i}(\theta,r):=G(t_i^{-\tau}r)A_i(\theta,r)+(1-G(t_i^{-\tau}r))\hat{A}_{
t_i}(\theta,r). 
\end{equation}
Then $A_{t_i}$ interpolates between $\hat{A}_{t_i}$ (when $r\simeq
t_i\hat{R}$) and $A_i$ (when $r\simeq \epsilon$) so $dA_{t_i}$
interpolates between $\hat{\alpha}_{t_i}$ and $\alpha_i$. 

We now set
$\iota_{\boldsymbol{t}}(\theta,r):=\Phi_{\mathcal{C}_i}(\theta,r,dA_{t_i}(\theta,r))+\iota(x_i)$. Then  $\iota_{\boldsymbol{t}}\equiv\iota$ on $\Sigma_i\times
[2t_i^\tau,\epsilon]$ and
$\iota_{\boldsymbol{t}}(\theta,r)=t_i\hat{\iota}(\theta,r/t_i)+\iota(x_i)$ on
$\Sigma_i\times [t_i\hat{R},t_i^\tau]$.

\subsubsection*{The induced metrics} Let $g_{\boldsymbol{t}}:=\iota_{\boldsymbol{t}}^*\tilde{g}$ be the induced metric on $L_{\boldsymbol{t}}$. 
\begin{lemma}\label{l:induced_metric} Assume $\tau$ is sufficiently close to $1$: specifically, for each neck, $\frac{2-\hat{\lambda}_i}{\mu_i-\hat{\lambda}_i}<\tau<1$. Then the metric $g_{\boldsymbol{t}}$ satisfies:
\begin{equation*}
 g_{\boldsymbol{t}}=\left\{
\begin{array}{llll}
 t_i^2\hat{g} & \mbox{ on the corresponding component of }
\hat{L}\setminus\hat{S}^*\\
\hat{\phi}_{t_i,i}^*(t_i^2\hat{g}) & \mbox{ on
}\Sigma_i\times[t_i\hat{R},t_i^\tau]\\
\phi_i^*g & \mbox{ on } \Sigma_i\times[2t_i^\tau,\epsilon]\\
g & \mbox{ on } L\setminus S^*
\end{array}\right.
\end{equation*}
and, for all $j\geq 0$ and as
$\boldsymbol{t}\rightarrow 0$,
\begin{equation}
 \label{eq:g_estimate}
\sup_{\Sigma_i\times[t_i^\tau,2t_i^\tau]}
|\tnabla^j(g_{\boldsymbol{t}}-\tg_i)|_{
r^ { -2 } \tg_i\otimes\tg_i} \rightarrow 0.
\end{equation}
\end{lemma}
\begin{proof}
The first part of the claim is a direct consequence of the definitions so we only need to check Equation \ref{eq:g_estimate}. We will use the same methods already used in Section \ref{s:con_review} and in \cite{pacini:sldefs} Section 4.2 except that, instead of concentrating on the behaviour as $r\rightarrow 0$, we concentrate on what happens in the subset $\Sigma_i\times[t_i^\tau,2t_i^\tau]$.

 Using Equations \ref{eq:A_t} and \ref{eq:A_t=t^2A} we see that, on $\Sigma_i\times[t_i^\tau,2t_i^\tau]$,
\begin{align*}
 dA_{t_i}(\theta,r)&=d_\Sigma A_{t_i}+A_{t_i}'dr\\
&=G(t_i^{-\tau}r)d_\Sigma A_{i|(\theta,r)}+(1-G(t_i^{-\tau}r))t_i^2d_\Sigma \hat{A}_{i|(\theta,r/t_i)}\\
&\ \ \ +\left(t_i^{-\tau}G'A_i+GA_i'-t_i^{-\tau}G't_i^2\hat{A}_i(\theta,r/t_i)+(1-G)t_i\hat{A}_i'(\theta,r/t_i)\right)dr,
\end{align*}
where $'$ denotes differentiation. Notice that $G$, $G'$ are bounded. Thus, using the norm $\tg_{i|(\theta,r)}$,
\begin{align*}
|dA_{t_i}|(\theta,r)\leq C\Big\{&|d_\Sigma A_i|_{|(\theta,r)}+t_i^{-\tau}|A_i|_{|(\theta,r)}+|A_i'|_{|(\theta,r)}\\
&+t_i^2|d_\Sigma\hat{A}_i|_{|(\theta,r/t_i)}+t_i^{2-\tau}|\hat{A}_i|_{|(\theta,r/t_i)}+t_i|\hat{A}_i'|_{|(\theta,r/t_i)}\Big\}.
\end{align*}
Since $A_i\in C^\infty_{\mu_i}$ we find
\begin{equation*}
 |d_\Sigma A_i|_{|(\theta,r)}\leq Cr^{\mu_i-1},\ \ |A_i|_{|(\theta,r)}\leq Cr^{\mu_i},\ \ |A_i'|_{|(\theta,r)}\leq Cr^{\mu_i-1}.
\end{equation*}
Since $\hat{A}_i\in C^\infty_{\hat{\lambda}_i}$ we find $|\hat{A}_i|_{|(\theta,r/t_i)}\leq C(r/t_i)^{\hat{\lambda}_i}$,   
$|\hat{A}_i'|_{|(\theta,r/t_i)}\leq C(r/t_i)^{\hat{\lambda}_i-1}$.
Furthermore, since we are using the norm $\tg_{i|(\theta,r)}$,
\begin{align*}
|d_\Sigma\hat{A}_i|_{|(\theta,r/t_i)}&=|d_\Sigma\hat{A}_i|_{r^2g_i'}(\theta,r/t_i)=|d_\Sigma\hat{A}_i|_{t_i^2(r/t_i)^2g_i'}(\theta,r/t_i)\\
&=t_i^{-1}|d_\Sigma\hat{A}_i|_{(r/t_i)^2g_i'}(\theta,r/t_i)\leq Ct_i^{-\hat{\lambda}_i}r^{\hat{\lambda}_i-1}.
\end{align*}
The end result is that
\begin{equation*}
 \sup_{\Sigma_i\times[t_i^\tau,2t_i^\tau]}|dA_{t_i}|_{\tg_i}\leq C\Big\{t_i^{\tau(\mu_i-1)}+t_i^{2-\hat{\lambda}_i+\tau(\hat{\lambda}_i-1)}\Big\}\leq C\,t_i^{2-\hat{\lambda}_i+\tau(\hat{\lambda}_i-1)},
\end{equation*}
where we use our assumption on $\tau$. More generally, one can check that
\begin{equation*}
 \sup_{\Sigma_i\times[t_i^\tau,2t_i^\tau]}|\tnabla^k A_{t_i}|_{\tg_i}\leq C\,t_i^{2-\hat{\lambda}_i+\tau(\hat{\lambda}_i-k)}.
\end{equation*}
 Recall from Section \ref{ss:conifold_defs} that we can write $\Phi_{\mathcal{C}_i}=\Psi_{\mathcal{C}_i}+R_i$. Using the equivariance properties and estimates given for $R_i$ and writing $dA_{t_i}=\alpha_{t_i,1}+\alpha_{t_i,2}\,dr$, we find that, on $\Sigma_i\times[t_i^\tau,2t_i^\tau]$,
\begin{align*}
 (R_i\circ dA_{t_i})(\theta,r)&=R_i(\theta,r,\alpha_{t_i,1},\alpha_{t_i,2})\\
&=R_i(\theta, r\cdot 1,r^2 r^{-2}\alpha_{t_i,1},rr^{-1}\alpha_{t_i,2})\\
&=rR_i(\theta,1,r^{-2}\alpha_{t_i,1},r^{-1}\alpha_{t_i,2})\\
&\leq Cr(r^{-4}|\alpha_{t_i,1}|^2_{g_i'}+r^{-2}|\alpha_{t_i,2}|^2)\\
&=Cr^{-1}(|\alpha_{t_i,1}|^2_{r^2g_i'}+|\alpha_{t_i,2}|^2)\\
&\leq Ct_i^{-\tau}|dA_{t_i}|_{\tg_i}^2(\theta,r)=Ct_i^{-\tau+2(2-\hat{\lambda}_i+\tau(\hat{\lambda}_i-1))}\\
&< Ct_i^{2-\hat{\lambda}_i+\tau(\hat{\lambda}_i-1)},
\end{align*}
where we use $\tau<1$. The definition of $\Psi_{\mathcal{C}_i}$, cf. \cite{pacini:sldefs} Section 4.2, yields an isometry $\Psi_{\mathcal{C}_i}\circ dA_{t_i}-\iota_i\simeq dA_{t_i}$. We thus conclude that, on $\Sigma_i\times[t_i^\tau,2t_i^\tau]$,
\begin{align*}
|\Phi_{\mathcal{C}_i}\circ dA_{t_i}-\iota_i|(\theta,r)&=|(\Psi_{\mathcal{C}_i}\circ dA_{t_i}-\iota_i)+R_i\circ dA_{t_i}|(\theta,r)\\
&\leq Ct_i^{2-\hat{\lambda}_i+\tau(\hat{\lambda}_i-1)}.
\end{align*}
The same methods prove that, more generally,
\begin{equation*}
 \sup_{\Sigma_i\times[t_i^\tau,2t_i^\tau]}|\tnabla^k(\Phi_{\mathcal{C}_i}\circ dA_{t_i}-\iota_i)|_{\tg_i}
\leq Ct_i^{2-\hat{\lambda}_i+\tau(\hat{\lambda}_i-1-k)}.
\end{equation*}
Analogously to Equations \ref{eq:cs_metric}, \ref{eq:ac_metric}, one can conclude that
\begin{equation*}
 \sup_{\Sigma_i\times[t_i^\tau,2t_i^\tau]}|\tnabla^k(g_{\bt}-\tg_i)|_{\tg_i}
\leq Ct_i^{2-\hat{\lambda}_i+\tau(\hat{\lambda}_i-2-k)}=Ct_i^{(1-\tau)(2-\hat{\lambda}_i)-\tau k},
\end{equation*}
which rescales to 
\begin{equation*}
 \sup_{\Sigma_i\times[t_i^\tau,2t_i^\tau]}|\tnabla^k(g_{\bt}-\tg_i)|_{r^{-2}\tg_i\otimes\tg_i}
\leq Ct_i^{(1-\tau)(2-\hat{\lambda}_i)}.
\end{equation*}
\end{proof}

\subsubsection*{The radius functions}
 We endow $L_{\boldsymbol{t}}$ with the radius function
\begin{equation*}
 \rho_{\boldsymbol{t}}:=\left\{
\begin{array}{ll}
 t_i\hat{\rho} & \mbox{ on the corresponding component of }
\hat{L}\setminus\hat{S}^*\\
r & \mbox{ on } \Sigma_i\times[t_i\hat{R},\epsilon]\\
\rho & \mbox{ on }L\setminus S^*.
\end{array}\right.
\end{equation*}

\subsubsection*{The weights and Sobolev spaces}
We endow $L_{\boldsymbol{t}}$ with the weight
\begin{equation*}
 \boldsymbol{\beta}_{\boldsymbol{t}}:=\left\{
\begin{array}{ll}
 \hat{\boldsymbol{\beta}} & \mbox{ on } \hat{L}\setminus\hat{S}^*\\
\beta_i & \mbox{ on } \Sigma_i\times[t_i\hat{R},\epsilon]\\
\boldsymbol{\beta} & \mbox{ on }L\setminus S^*.
\end{array}\right.
\end{equation*}
We now need to define the function $w_{\boldsymbol{t}}$, used to define weighted Sobolev spaces as in Definition \ref{def:csac_sectionspaces}. The simplest
case is when $\hat{\boldsymbol{\beta}}$ is constant on each connected
component of $\hat{L}$. We then define
\begin{equation*}
 w_{\boldsymbol{t}}:=\rho_{\boldsymbol{t}}^{-{\boldsymbol{\beta}_{\boldsymbol{t}}}}=\left\{
\begin{array}{ll}
 (t_i\hat{\rho})^{-\hat{\beta}_i} & \mbox{ on the corresponding
component of }
\hat{L}\setminus\hat{S}^*\\
r^{-\beta_i} & \mbox{ on } \Sigma_i\times[t_i\hat{R},\epsilon]\\
\rho^{-\boldsymbol{\beta}} & \mbox{ on }L\setminus S^*.
\end{array}\right.
\end{equation*}
For general weights $\hat{\boldsymbol{\beta}}$ we need to modify the function $w_{\boldsymbol{t}}$ as
follows:
\begin{equation*}
 w_{\boldsymbol{t}}:=\left\{
\begin{array}{ll}
 (t_i^\frac{\hat{\beta}_i-\hat{\boldsymbol{\beta}}}{\hat{\boldsymbol{\beta}}}t_i\hat{\rho})^{-\hat{
\boldsymbol {
\beta
}} }=t_i^{-\hat{\beta}_i}\hat{\rho}^{-\hat{\boldsymbol{\beta}}}  & \mbox{ on the
corresponding component of }
\hat{L}\setminus\hat{S}^*\\
r^{-\beta_i} & \mbox{ on } \Sigma_i\times[t_i\hat{R},\epsilon]\\
\rho^{-\boldsymbol{\beta}} & \mbox{ on }L\setminus S^*.
\end{array}\right.
\end{equation*}
Using this data we now define weighted Sobolev spaces
$W^p_{k,\boldsymbol{\beta}_{\boldsymbol{t}}}$ on $L_{\boldsymbol{t}}$ as in
Definition \ref{def:csac_sectionspaces}.

\subsubsection*{The Lagrangian neighbourhoods}
We define Lagrangian neighbourhoods for $(L_{\boldsymbol{t}}, i_{\boldsymbol{t}})$ using the same ideas as in Section 
\ref{s:rescaled_lag_cons}. Specifically, let 
\begin{equation*}
 \Phi_{\hat{L}}:\hat{\mathcal{U}}\subset T^*\hat{L}\rightarrow\C^m,\ \ \Phi_L:\mathcal{U}\subset T^*L\rightarrow\C^m
\end{equation*}
be Lagrangian nieghbourhoods and maps for $(L,\iota)$, $(\hat{L},\hat{\iota})$. We then define neighbourhoods and maps  
\begin{equation*}
\Phi_{L_{\boldsymbol{t}}}:\mathcal{U}_{\boldsymbol{t}}\subset T^*L_{\boldsymbol{t}}\rightarrow\C^m
\end{equation*}
as follows:
\begin{align*}
 \mathcal{U}_{\boldsymbol{t}} :=& \left\{
\begin{array}{ll}
 t_i\hat{\mathcal{U}}, & \mbox{ a subset of the corresponding component of }
T^*(\hat{L}\setminus\hat{S}^*)\\
\tau_{dA_{t_i}}^{-1}(\mathcal{U}), & \mbox{ a subset of } T^*(\Sigma_i\times[t_i\hat{R},\epsilon])\\
\mathcal{U}, & \mbox{ a subset of }T^*(L\setminus S^*),
\end{array}\right.\\
 \Phi_{L_{\boldsymbol{t}}} :=& \left\{
\begin{array}{ll}
 t_i\Phi_{\hat{L}}t_i^{-1}+\iota(x_i) & \mbox{ on } t_i\hat{\mathcal{U}}\\
\Phi_{\mathcal{C}_i}\circ\tau_{dA_{t_i}} +\iota(x_i)& \mbox{ on } \tau_{dA_{t_i}}^{-1}(\mathcal{U})\\
\Phi_L & \mbox{ on } \mathcal{U}.
\end{array}\right.
\end{align*}
By construction $\Phi_{L_{\boldsymbol{t}}}$, restricted to the zero-section, coincides with $\iota_{\boldsymbol{t}}$.

\subsubsection*{Deformations of Lagrangian connect sums}

Assume for simplicity that all ends in $S^{**}$, respectively $\hat{S}^{**}$, are AC with convergence rate $\boldsymbol{\lambda}$, respectively $\hat{\boldsymbol{\lambda}}$. Exactly as in Section \ref{ss:conifold_defs}, for each fixed $\bt$ we can parametrize the set of Lagrangian immersions which are close to $\iota_{\bt}$ and are asymptotic to the same cones with rate $(\boldsymbol{\lambda},\hat{\boldsymbol{\lambda}})$ using the space of closed $1$-forms in $C^\infty_{(\boldsymbol{\lambda}-1,\hat{\boldsymbol{\lambda}}-1)}(\mathcal{U}_{\bt})$. 
Those same methods also allow us to parametrize deformations of general Lagrangian connect sum conifolds, with either fixed or moving singularities.

\begin{example}\label{ex:lagr_planes_desing}
Assume given a finite configuration of $l$ Lagrangian planes in $\C^m$. We assume there is only a finite number $s$  of (necessarily transverse) intersection points and that any intersection involves only two planes at a time. Let $(L,\iota)$ denote a parametrization of this configuration. Specifically, $L$ is the disjoint union of $l$ copies of $\R^m$ and we choose an even number of points $x_1,\dots,x_{2s}\in L$ such that each $\iota(x_{2i})=\iota(x_{2i+1})\in\C^m$ is one of the intersection points. Then $(L, \iota)$ is a Lagrangian conifold in 
$\C^m$ with $l$ AC ends and $2s$ CS ends. Notice that each asymptotic cone $\mathcal{C}_i$ is parallel to one of the original planes in the configuration. Letting $S^*$ be the union of the CS ends we now have a CS-marked Lagrangian conifold in $\C^m$.

Now assume that for each pair of indices $\{2i,2i+1\}$ there exists an embedded Lagrangian submanifold $\hat{L}_i$ interpolating between $\mathcal{C}_{2i}$ and $\mathcal{C}_{2i+1}$ and satisfying the necessary condition on convergence rates: $\hat{\lambda}_{2i}<0$, $\hat{\lambda}_{2i+1}<0$. Let $(\hat{L}, \hat{\iota})$ denote the union of these manifolds together with the corresponding embeddings. Then $(\hat{L}, \hat{\iota})$ is an AC Lagrangian conifold with $2s$ ends. Letting $\hat{S}^*$ be the union of the AC ends we now have an AC-marked Lagrangian conifold.

By construction $(L,\iota)$, $(\hat{L},\hat{\iota})$ are compatible so we can glue them together as above. The result is a connected AC Lagrangian submanifold $(L_{\bt},\iota_{\bt})$ embedded in $\C^m$. This submanifold has $s$ necks and $l$ AC ends. Away from the neck regions it coincides with the original configuration of planes.
\end{example}

\begin{example}\label{ex:sub_and_planes}
 Let $(L',\iota')$ be a Lagrangian conifold in $\C^m$ with $l$ AC ends and $s$ CS ends. Choose a smooth point $x\in L'$ and a Lagrangian plane $\Pi+\iota(x)$ transverse to $L'$ in $\iota(x)$. Let $L$ denote the disjoint union of $L'$ and $\Pi$ and let $\iota$ denote the natural immersion. The point $x$ in $L'$ and the origin in the plane can be labelled as singularities so $(L,\iota)$ is a Lagrangian conifold in $\C^m$ with $l+1$ AC ends and $s+2$ CS ends. Assume there exists an embedded Lagrangian submanifold $\hat{L}$ interpolating between the tangent plane $i_*'(T_xL')$ and $\Pi$. As in Example \ref{ex:lagr_planes_desing} we can then build a family $(L_t,\iota_t)$ of Lagrangian conifolds which desingularizes $(L,\iota)$. Notice that $L_t$ has $l+1$ AC ends, $s$ CS ends and one neck.  

Of course we can repeat this construction any number of times, adding as many new AC ends as we want to our original conifold $L'$.   
\end{example}

As already seen, a weighted Sobolev Embedding Theorem holds for any conifold. Given that the connect sum of two conifolds is another (possibly smooth compact) conifold, it also holds for this connect sum. It is an important fact that the corresponding ``embedding constants'' are independent of the parameter $\boldsymbol{t}$, cf. \cite{pacini:ics} Theorem 7.7.

\begin{theorem}\label{thm:normstequivalent}
 Let $(L,\boldsymbol{\beta})$, $(\hat{L},\hat{\boldsymbol{\beta}})$ be
compatible weighted marked conifolds. Let $L_{\boldsymbol{t}}$,
$\rho_{\boldsymbol{t}}$, 
$\boldsymbol{\beta}_{\boldsymbol{t}}$ and $w_{\bt}$ 
denote the corresponding connect sums. Then all
forms of the weighted Sobolev Embedding Theorem, as in Theorem \ref{thm:embedding}, hold uniformly in
${\boldsymbol{t}}$,
\textit{i.e.} the corresponding Sobolev embedding constants are independent of
${\boldsymbol{t}}$.
\end{theorem}

The above result relies on the fact that, on the two sides of each inequality, the relevant quantities are expressed using the same weight $\boldsymbol{\beta}_{\bt}$. If instead we allow the weight to change on different sides of the inequality, factors of $\bt$ are inevitable. We will need the following case.

\begin{lemma}\label{l:changing_weight_estimates}
Let $(L,\boldsymbol{\beta})$, $(\hat{L},\hat{\boldsymbol{\beta}})$ be compatible weighted marked conifolds. Choose a second pair of compatible weights $\boldsymbol{\beta}'$, $\hat{\boldsymbol{\beta}'}$.  Let $L_{\boldsymbol{t}}$,
$\rho_{\boldsymbol{t}}$, $\boldsymbol{\beta}_{\bt}$ and $w_{\bt}$, respectively $\boldsymbol{\beta}_{\bt}'$ and $w_{\bt}'$, denote the corresponding connect sums. Assume
\begin{equation*}
\left\{
\begin{array}{l}
\boldsymbol{\beta}_{\bt}\leq\boldsymbol{\beta}_{\bt}'\mbox{ on all AC ends of }L_{\bt}\\
\boldsymbol{\beta}_{\bt}\geq\boldsymbol{\beta}_{\bt}'\mbox{ on all CS ends of }L_{\bt}\\
\beta_i\leq\beta_i'\mbox{ on all necks of }L_{\bt}.
\end{array}
\right.
\end{equation*}
Then the natural immersions $C^k_{\boldsymbol{\beta}_{\bt}}(L_{\bt})\hookrightarrow C^k_{\boldsymbol{\beta}_{\bt}'}(L_{\bt})$ are not uniform in $\bt$. However, there exists $C>0$ such that, for all $\bt$ and for all $f\in C^k_{\boldsymbol{\beta}_{\bt}}(L_{\bt})$,
\begin{equation*}
 \|f\|_{C^k_{\boldsymbol{\beta}_{\bt}'}}\leq C \max \{t_i^{\beta_i-\beta_i'}\} \|f\|_{C^k_{\boldsymbol{\beta}_{\bt}}},
\end{equation*}
where $\max$ is calculated over all necks in $L_{\bt}$.
\end{lemma}
\begin{proof}
 Consider the case $k=0$. On the $i$-th component of $\hat{L}\setminus\hat{S}^*$,
\begin{align*}
 |w_{\bt}'f|=&t_i^{-\hat{\beta}_i'}\hat{\rho}^{-\hat{\boldsymbol{\beta}}'}|f|\\
=&t_i^{\hat{\beta}_i-\hat{\beta}_i'}\hat{\rho}^{\hat{\boldsymbol{\beta}}-\hat{\boldsymbol{\beta}}'}t_i^{-\hat{\beta}_i}\hat{\rho}^{-\hat{\boldsymbol{\beta}}}|f|\leq C t_i^{\hat{\beta}_i-\hat{\beta}_i'} |w_{\bt}f|,
\end{align*}
because our assumptions imply that $\hat{\rho}^{\hat{\boldsymbol{\beta}}-\hat{\boldsymbol{\beta}}'}$ is bounded. On $L\setminus S^*$,
\begin{align*}
 |w_{\bt}'f|=\rho^{-\boldsymbol{\beta}'}|f|=\rho^{\boldsymbol{\beta}-\boldsymbol{\beta}'}\rho^{-\boldsymbol{\beta}}|f|\leq C |w_{\bt}f|.
\end{align*}
On any neck $\Sigma_i\times [t_i\hat{R},\epsilon]$ of $L_{\bt}$,
\begin{align*}
 |w_{\bt}'f|=r^{-\beta_i'}|f|=r^{\beta_i-\beta_i'}r^{-\beta_i}|f|\leq C t_i^{\beta_i-\beta_i'} |w_{\bt}f|.
\end{align*}
The case $k\geq 1$ is similar.
\end{proof}

Under appropriate conditions one can also show that the Laplace operator is uniformly invertible in the following sense. 

\begin{theorem}\label{thm:sum_injective}
Let $(\hat{L},\hat{\boldsymbol{\beta}})$ be a
weighted AC-marked conifold. Assume $\hat{\boldsymbol{\beta}}$ satisfies the
conditions
\begin{equation*}
\left\{
\begin{array}{l}
\hat{\beta}_i<0\mbox{ for all AC ends
}\hat{S}_i\in \hat{S}\\
\hat{\beta}_i>2-m\mbox{
for all CS ends }\hat{S}_i\in \hat{S}
\end{array}
\right.
\end{equation*}
so that $\Delta_{\hat{g}}$ is injective.

Let $(L,\boldsymbol{\beta})$ be a weighted CS-marked conifold. Assume
$\boldsymbol{\beta}$ satisfies the conditions
\begin{equation*}
\left\{
\begin{array}{l}
\beta_i<0\mbox{ for all AC
ends }S_i\in S\\
\beta_i>2-m\mbox{ for all CS ends
}S_i\in S.
\end{array}
\right.
\end{equation*}
This is not yet sufficient to conclude that $\Delta_g$ is
injective because the set of AC ends might be empty. To obtain injectivity we
must furthermore assume that each component of $L$ has at least one end $S'$ satisfying the condition
\begin{equation*}
\left\{
\begin{array}{l}
\beta'<0 \mbox{ if $S'$ is AC}\\
\beta'>0 \mbox{ if $S'$ is CS}.
\end{array}
\right.
\end{equation*}

Now assume that $L$, $\hat{L}$ are compatible and $\tau$ is sufficiently close to 1, as in Lemma \ref{l:induced_metric}. Then, for all ends $S_i\in S^*$,
$2-m<\beta_i<0$. This implies that, for each component, $S'\in S^{**}$ so
$L_{\boldsymbol{t}}$ has at least one end. Furthermore,
$\boldsymbol{\beta}_{\boldsymbol{t}}$ satisfies the conditions
\begin{equation*}
\left\{
\begin{array}{l}
\boldsymbol{\beta}_{{\boldsymbol{t}}|S_i}<0\mbox{ for all AC ends
}S_i\in \hat{S}^{**}\cup S^{**}\\
\boldsymbol{\beta}_{{\boldsymbol{t}}|S_i}>2-m\mbox{
for all CS ends }S_i\in \hat{S}^{**}\cup S^{**}.
\end{array}
\right.
\end{equation*}
Together with the condition on $S'$, this implies that $\Delta_{g_{\boldsymbol{t}}}$ is injective. 

If furthermore $\boldsymbol{\beta}$, $\hat{\boldsymbol{\beta}}$ are
non-exceptional for $\Delta_{g}$,
$\Delta_{\hat{g}}$ then 
$$\Delta_{g_{\boldsymbol{t}}}:W^p_{k,\boldsymbol{\beta}_{\boldsymbol{t}}}(L_{\boldsymbol{t}})
\rightarrow W^p_{k-2,\boldsymbol{\beta}_{\boldsymbol{t}}-2}(L_{\boldsymbol{t}}) $$
is a topological isomorphism onto its image and there exists $C>0$ such that, for all $\boldsymbol{t}$ and
all $f\in W^p_{k,\boldsymbol{\beta}_{\boldsymbol{t}}}(L_{\boldsymbol{t}})$,
$$\|f\|_{W^p_{k,\boldsymbol{\beta}_{\boldsymbol{t}}}}\leq C\|\Delta_{g_{\boldsymbol{t}}}
f\|_{W^p_{k-2,\boldsymbol{\beta}_{\boldsymbol{t}}-2}}.$$
\end{theorem}
\begin{proof}
 For later reference we give a sketch of the proof, referring to \cite{pacini:ics} for details. The basic idea is to choose constants $a$, $b$ satisying $0<b<a<\tau$ and a cut-off function $\eta_t$ supported in the interval $[t^a,t^b]\subset (2t^\tau, \epsilon)$ so that, for any $f\in W^p_{k,\boldsymbol{\beta}_{\bt}}$, $\eta_tf$ has support in $L$ and $(1-\eta_t)f$ has support in $\hat{L}$. One then argues that 
\begin{equation}
\begin{split}
 \|\eta_tf\|_{W^p_{k,\boldsymbol{\beta}_{\bt}}(g_{\bt})}&=\|\eta_tf\|_{W^p_{k,\boldsymbol{\beta}}(g)}\leq C\|\Delta_g(\eta_tf)\|_{W^p_{k-2,\boldsymbol{\beta}-2}(g)}\\
&=C\|\Delta_{g_{\bt}}(\eta_tf)\|_{W^p_{k-2,\boldsymbol{\beta}_{\bt}-2}(g_{\bt})}.\\
\end{split}
\end{equation}
Likewise, on each component of $\hat{L}$,
\begin{equation}\label{eq:1-eta}
\begin{split}
 \|(1-\eta_t)f\|_{W^p_{k,\boldsymbol{\beta}_{\bt}}(g_{\bt})}&\simeq\|(1-\eta_t)f\|_{W^p_{k,\hat{\boldsymbol{\beta}}}(t_i^2\hat{g})}=t_i^{-\beta_i}\|(1-\eta_t)f\|_{W^p_{k,\hat{\boldsymbol{\beta}}}(\hat{g})}\\
&\leq Ct_i^{-\beta_i}\|\Delta_{\hat{g}}((1-\eta_t)f)\|_{W^p_{k-2,\hat{\boldsymbol{\beta}}-2}(\hat{g})}\\
&=Ct_i^{2-\beta_i}\|\Delta_{t_i^2\hat{g}}((1-\eta_t)f)\|_{W^p_{k-2,\hat{\boldsymbol{\beta}}-2}(\hat{g})}\\
&\simeq C\|\Delta_{g_{\bt}}((1-\eta_t)f)\|_{W^p_{k-2,\boldsymbol{\beta}_{\bt}-2}(g_{\bt})}.\\
\end{split}
\end{equation}
If $\eta_t$ is chosen carefully, one can estimate $\|\Delta_{g_{\bt}}(\eta_tf)\|$ and $\|\Delta_{g_{\bt}}((1-\eta_t)f)\|$ in terms of $\|\Delta_{g_{\bt}} f\|$, generating only a very small error term of the form $(1/|\log t|)\|f\|_{W^p_{k,\boldsymbol{\beta}_{\bt}}}$. Combining these estimates gives the desired result.
\end{proof}


\section{Special Lagrangian conifolds in $\C^m$}\label{s:SL_defs}

\begin{definition}\label{def:cy} A \textit{Calabi-Yau} (CY) manifold is the data
of a K\"ahler manifold ($M^{2m}$,$g$,$J$,$\omega$) and a non-zero ($m,0$)-form
$\Omega$ satisfying $\nabla\Omega\equiv 0$ and normalized by the condition
$\omega^m/m!=(-1)^{m(m-1)/2}(i/2)^m\Omega\wedge\bar{\Omega}$. In particular $\Omega$ is holomorphic and the holonomy of $(M,g)$ is contained
in $\sunitary m$. We will refer to $\Omega$ as the \textit{holomorphic volume
form} on $M$.

Let $M^{2m}$ be a CY manifold and $L^m\rightarrow M$ be an immersed or embedded
Lagrangian submanifold. We can restrict $\Omega$ to $L$, obtaining a
non-vanishing complex-valued $m$-form $\Omega_{|L}$ on $L$. We say that $L$ is
\textit{special Lagrangian} (SL) if and only if this form is real, \textit{i.e.}
$\Imag\,\Omega_{|L}\equiv 0$. In this case $\Real\,\Omega_{|L}$ defines a volume
form on $L$, thus a natural orientation.

$\Real\,\Omega_{|L}$ is actually a ``calibration'' on $M$ in the sense of \cite{harveylawson}. It follows that any SL is volume-minimizing in its homology class.
\end{definition}

\begin{example}\label{e:C^m_examples}
 The simplest example of a CY manifold is $\C^m$ with its standard structures
$\tilde{g}$, $\tilde{J}$, $\tilde{\omega}$ and
$\tilde{\Omega}:=dz^1\wedge\dots\wedge dz^m$. We are interested in SL conifolds in $\C^m$. Because of the possibility of dilations, it is clear that $\C^m$ does not admit compact volume-minimizing submanifolds, even with singularities. It follows that each connected component of a SL conifold in $\C^m$ must contain at least one AC end. This is particularly relevant for Theorem \ref{thm:sum_injective}: using appropriate weights on these AC ends  allows us to assume that the Laplace operator is injective.

Cones with one isolated singularity at the origin are the most basic class of SL conifolds in $\C^m$, cf. \cite{harveylawson},
\cite{haskins}, \cite{haskinskapouleas}, \cite{haskinspacini}, \cite{joyce:symmetries} for examples. AC SLs are a  second important class. We will restrict our attention to the following two examples, of interest to us because their convergence rates satisfy the assumption $\boldsymbol{\lambda}<0$.
\begin{enumerate}
 \item Let $\Pi$ be any SL plane in $\C^m$, \textit{e.g.} the standard $\R^m$ generated by the vectors $\partial x_i$. Choose $M\in\unitary m$ of the form 
$M=\mbox{Diag}(e^{i\theta_1},\dots,e^{i\theta_m})$, where $\sum\theta_i=\pi$. The union $\Pi\cup-(M\cdot\Pi)$ is a SL cone with an isolated singularity in the origin. Lawlor \cite{lawlor} constructed an AC SL with 2 ends, diffeomorphic to $\Sph^{m-1}\times\R$, asymptotic to $\Pi\cup-(M\cdot\Pi)$ and with rate $\boldsymbol{\lambda}=2-m$ as $r\rightarrow \infty$. This submanifold is known as the ``Lawlor neck''. We refer to \cite{joyce:V} Example 6.11 for details.
\item Let $\mathcal{C}$ be a SL cone in $\C^m$ with an isolated singularity in the origin. Let $\Sigma$ denote its link. Given any $c>0$ define the curve $\gamma_c\simeq\R$ via the equations
\begin{equation*}
\gamma_c:=\{\lambda\in\C:\Imag{(\lambda^m)}=c^m,\ \ \arg{\lambda}\in(0,\pi/m)\}.
\end{equation*}
Then the immersion
\begin{equation*}
 \iota_c:\Sigma\times\gamma_c\rightarrow \C^m,\ \ (\theta,\lambda)\mapsto \lambda\theta
\end{equation*}
defines an AC SL with two ends, diffeomorphic to $\Sigma\times\R$, asymptotic to $\mathcal{C}\cup -(e^{i\pi/m}\mathcal{C})$ and with rate $\boldsymbol{\lambda}=2-m$ as $r\rightarrow\infty$. This construction was found independently by Castro-Urbano, Haskins and Joyce, cf. \cite{joyce:symmetries} Theorem 6.4 for details.
\end{enumerate}
\end{example}

Let $(L_{\bt},\iota_{\bt})$ be a Lagrangian connect sum in $\C^m$. We want to find a way to detect whether any small Lagrangian deformations of $(L_{\bt},\iota_{\bt})$ are SL. 
The results we obtain here will be applied in Section \ref{s:SL_sum}. For those purposes, compared to previous sections, we can restrict our attention to ``exact'' Lagrangian deformations, defined via functions. As usual, this class of deformations is defined by a choice of weight $\boldsymbol{\beta}$ which we will always assume satisfies the following bounds: $\beta_i<2$ on each AC end, $\beta_i>2$ on each CS end. In this section, which aims only to set the problem up correctly, we will see that various choices of such $\boldsymbol{\beta}$ are possible as long as they are coupled with appropriate additional conditions on the ends of $(L,\iota)$. In particular, Lemma \ref{lemma:extra} presents the two most natural cases. All results in this section will hold in both cases.
On the other hand, actually finding a SL deformation, as in Section \ref{s:SL_sum}, will be possible only if we choose the stronger of the two sets of assumptions.

As usual, it is useful to build up to the problem slowly, dividing it into several steps. We thus start by examining the analogous problem for Lagrangian conifolds $(L,\iota)$ which have no $\bt$-dependence. In what follows we will repeatedly use the following abstract result concerning Taylor expansions.

\begin{lemma}\label{l:taylor}
 Let $(V, |\cdot|)$ be a normed vector space. Let $F:U\rightarrow \R$ be a smooth function defined on a convex subset $U\subset V$. Write
\begin{equation*}
 F(x)=F(0)+dF_{|0}\,x+Q(x),
\end{equation*}
 for some smooth $Q:U\rightarrow \R$. Assume $C:=\sup_{x\in U}|d^2F_{|x}|<\infty$, where 
\begin{equation*}
|d^2F_{|x}|:=\sup_{|v|\leq 1}\sup_{|w|\leq 1}|d^2F_{|x}(v,w)|.
\end{equation*} 
Then:
\begin{enumerate}
 \item $|Q(x)|\leq (C/2)|x|^2$,
\item $|dQ_{|x}|\leq C|x|$,
\item $|Q(x)-Q(y)|\leq (C/2)|x-y|\left(|x|+|y|\right)$.
\end{enumerate}
In particular, if we substitute $|\cdot|$ with $|\cdot|_t:=t|\cdot|$ then the corresponding constant satisfies $C_t=t^{-2}C$ so the above estimates do not change. In this sense these estimates are scale-invariant.
\end{lemma}
\begin{proof}
 Let $\gamma(t):=tx$ be the geodesic segment connecting $0$ to $x$, so that $|\gamma'|\equiv |x|$. Then
\begin{equation*}
 \begin{split}
  Q(x)&=F(x)-F(0)-dF_{|0}\,x=\int_0^1(F\circ\gamma)'(t)\,dt-\int_0^1(F\circ\gamma)'(0)\,dt\\
&=\int_0^1\int_0^t (F\circ\gamma)''(s)\,ds\,dt=\int_0^1\int_0^td^2F_{|\gamma(s)}\cdot(\gamma'(s))^2\,ds\,dt\\
&\leq \int_0^1\int_0^t|d^2F_{|\gamma(s)}|\cdot|\gamma'(s)|^2\,ds\,dt\leq (C/2)|x|^2. 
\end{split}
\end{equation*}
This proves (1). Regarding (2), 
\begin{equation*}
 \begin{split}
  dQ_{|x}=dF_{|x}-dF_{|0}=\int_0^1\frac{d}{dt}dF_{|\gamma(t)}\,dt=\int_0^1d^2F_{|\gamma(t)}\cdot\gamma'(t)\,dt\leq \int_0^1C|\gamma'|\,dt=C|x|.\\
 \end{split}
\end{equation*}
This proves (2). Regarding (3), let $\gamma(t):=tx+(1-t)y$ be the geodesic segment connecting $y$ to $x$ so that $|\gamma'|\equiv |x-y|$. Then:
\begin{equation*}
 \begin{split}
  Q(x)-Q(y)&=\int_0^1\frac{d}{dt}Q(\gamma(t))\,dt=\int_0^1dQ_{|\gamma(t)}\gamma'(t)\,dt\\
&\leq\int_0^1C|\gamma(t)|\cdot|\gamma'(t)|\,dt\leq C|x-y|\int_0^1|\gamma(t)|\,dt\\
&\leq C|x-y|\int_0^1(t|x|+(1-t)|y|)\,dt.\\
 \end{split}
\end{equation*}
Calculating the integral proves (3).
\end{proof}
\begin{remark}\label{rem:taylor}
As a special case of Lemma \ref{l:taylor}, assume given normed spaces $(Y,|\cdot|)$ and $(Z,|\cdot|)$. Set $V:=Y\oplus Z$, endowed with the sum of the norms. Let $\partial_1$, $\partial_2$ denote the corresponding directional derivatives. Define $C_{ij}:=\sup_{x\in U}|\partial_i\partial_jF_{|x}|$. Then the inequalities in Lemma \ref{l:taylor} can be slightly refined, \textit{e.g.}:
\begin{equation*}
 \begin{split}
  |Q(y,z)|&\leq (C_{11}/2)|y|^2+C_{12}|y||z|+(C_{22}/2)|z|^2,\\
 |\partial_1Q_{|(y,z)}|&\leq C_{11}|y|+C_{21}|z|,\\
|\partial_2Q_{|(y,z)}|&\leq C_{12}|y|+C_{22}|z|.\\
 \end{split}
\end{equation*}
In particular, using the same notation and methods as in the proof of Lemma \ref{l:taylor} (3),
\begin{equation*}
 \begin{split}
  Q(y,z)-Q(y',z')&=\int_0^1\partial_1Q_{|\gamma(t)}\cdot\gamma_1'(t)\,dt+\int_0^1\partial_2Q_{|\gamma(t)}\cdot\gamma'_2(t)\,dt\\
&\leq |y-y'|\int_0^1|\partial_1Q_{|\gamma(t)}|\,dt+|z-z'|\int_0^1|\partial_2Q_{|\gamma(t)}|\,dt\\
&\leq |y-y'|\Big((C_{11}/2)(|y|+|y'|)+(C_{21}/2)(|z|+|z'|)\Big)\\
&\ \ \ +|z-z'|\Big((C_{12}/2)(|y|+|y'|)+(C_{22}/2)(|z|+|z'|)\Big).\\
 \end{split}
\end{equation*}
\end{remark}

\subsubsection*{Lagrangian conifolds}
Let $(L,\iota)$ be a AC or CS/AC Lagrangian conifold in $\C^m$; as seen above, there is no point in searching for SL deformations of a CS Lagrangian submanifold in $\C^m$. Fix a weight $\boldsymbol{\beta}$ and let $B_r$ denote the ball centered in $0$ and of radius $r$ in $W^p_{3,\boldsymbol{\beta}}(L)$. Consider the map 
\begin{equation}\label{eq:Fsobolev}
F:B_{r}\subset W^p_{3,\boldsymbol{\beta}}(L)\rightarrow W^p_{1,\boldsymbol{\beta}-2}(L), \ \
f\mapsto \star((\Phi_{L}\circ df)^*\Imag\,\tilde{\Omega}),
\end{equation}
where $\star$ denotes the Hodge star operator of $L$. According to Definition \ref{def:cy}, the deformed immersion $(L,\Phi_L\circ df)$ is SL if and only if $F(f)=0$.

We need to find conditions on $\boldsymbol{\beta}$ and on $(L,\iota)$ ensuring that this map is well-defined. It is instructive to include an analogous discussion for $F$ defined between spaces of smooth functions because it is somewhat more transparent and contains all the main ideas.

\begin{lemma} \label{lemma:extra}
Let $(L,\iota)$ be a Lagrangian conifold in $\C^m$ with rates $(\boldsymbol{\mu},\boldsymbol{\lambda})$. Choose a weight $\boldsymbol{\beta}$ and consider the map 
\begin{equation*}
 F:B_{r}\subset C^\infty_{\boldsymbol{\beta}}(L)\rightarrow C^{\infty}_{\boldsymbol{\beta}-2}(L), \ \
f\mapsto \star((\Phi_{L}\circ df)^*\Imag\,\tilde{\Omega}).
\end{equation*}
Then, under either of the following conditions, this map is well-defined (for $r$ sufficiently small):
\begin{enumerate}
 \item Assume the asymptotic cones of $(L,\iota)$ are SL and $\boldsymbol{\beta}$ satisfies $\mu_i\geq\beta_i>2$ on each CS end, $\lambda_i\leq\beta_i<2$ on each AC end.
\item Assume the ends of $(L,\iota)$ are SL and $\boldsymbol{\beta}$ satisfies $\beta_i>\mu_i>2$ on each CS end, $\beta_i<\lambda_i<2$ on each AC end.
\end{enumerate}
Now assume $p>m$. Then, in both cases, the map $F$ introduced in Equation \ref{eq:Fsobolev} is well-defined.
 \end{lemma}

\begin{proof}
Recall that the radius of $\mathcal{U}$, in the sense introduced following Equation \ref{eq:R_properties}, is proportional to $\rho$. For $r$ sufficiently small, the Sobolev Embedding Theorem ensures that the $C^2_{\boldsymbol{\beta}}$ norm of $f$ is small. Our bounds on $\boldsymbol{\beta}$, relative to $2$, yield a further embedding $C^2_{\boldsymbol{\beta}}\hookrightarrow C^2_2$ implying that the $C^2_2$ norm of $f$ is small so $df$ has small $C^1_1$ norm. In particular $\rho^{-1}df$ is small so indeed the graph of $df$ lies in $\mathcal{U}$: this proves that the composition $\Phi_L\circ df$ makes sense.

We now need to show that $F$ takes values in the space $C^{\infty}_{\boldsymbol{\beta}-2}(L)$.
The proof requires a slight generalization of the results of Section \ref{ss:conifold_defs}, allowing for weights $\boldsymbol{\beta}\neq (\boldsymbol{\mu},\boldsymbol{\lambda})$. The details are discussed in \cite{pacini:sldefsextended} Section 4.5. The result is that, given weights as in (1) and $f\in B_{r}$, the immersion $\iota':=\Phi_L\circ df$ defines a Lagrangian conifold with convergence rate $\boldsymbol{\beta}$. Choose an AC end $S_i$. Then, as in Definition \ref{def:aclagsub}, up to diffeomorphisms $\phi_i$ and translations by $p_i$, $\iota'$ has the property that, for $r\rightarrow\infty$ and all $k\geq 0$,
\begin{equation*}
|\tnabla^k(\iota'-\iota_i)|=O(r^{\beta_i-1-k}).
\end{equation*}
Set $R:=\iota'-\iota_i$ and choose a $g$-orthonormal basis $e_i$. Then
\begin{equation*}
\begin{split}
F(f)&=\star((\Phi_{L}\circ df)^*\Imag\,\tilde{\Omega})=\Imag\,\tilde{\Omega}(\iota'_*(e_1),\dots,\iota'_*(e_m))\\
&=\Imag\,\tilde{\Omega}\big((\iota_{i*}+R_*)(e_1),\dots,(\iota_{i*}+R_*)(e_m)\big)\\
&=\Imag\,\tilde{\Omega}(\iota_{i*}(e_1),\dots,\iota_{i*}(e_m))+O(r^{\beta_i-2})=O(r^{\beta_i-2}),\\
\end{split}
\end{equation*}
where we use the fact that the asymptotic cone, thus $\iota_i$, is SL. A similar proof holds for the CS ends and for higher derivatives, showing that $F$ takes values in $C^{\infty}_{\boldsymbol{\beta}-2}(L)$.

Notice that the assumption that the asymptotic cones of $(L,\iota)$ are SL would not suffice to obtain the same result for the weights of case (2): for example, $f:=0\in C^\infty_{\boldsymbol{\beta}}(L)$ but $F(0)\in C^\infty_{(\boldsymbol{\mu}-2,\boldsymbol{\lambda}-2)}(L)$. For this reason we add the stronger condition that the ends of $(L,\iota)$ be SL. Choose $f\in B_{r}$. The immersion $\iota':=\Phi_L\circ df$ then defines a Lagrangian conifold which is asymptotic to $\iota$ in a sense analogous to Definitions \ref{def:aclagsub}, \ref{def:cslagsub}: for example, on each AC end and for $r\rightarrow\infty$,
\begin{equation*}
|\tnabla^k(\iota'-\iota)|=O(r^{\beta_i-1-k}).
\end{equation*} 
The same proof then shows that $F$ is well-defined in this case also. 

The proof that the map $F$ of Equation \ref{eq:Fsobolev} is well-defined is similar: the only extra ingredient is the choice $p>m$ which ensures that our Sobolev spaces are Banach algebras, cf. Theorem \ref{thm:embedding}.

\end{proof}
We can think of $F$ as being obtained from an underlying smooth function 
\begin{equation}\label{eq:F'}
F'=F'(x,y,z):\mathcal{U}\oplus(T^*L\otimes T^*L)\rightarrow \R
\end{equation}
defined as follows. Choose a point $(x,y)\in
\mathcal{U}$. Let $e_1,\dots,e_m$ be an orthonormal positive basis of $T_xL$.
Now choose any $z\in T_x^*L\otimes T_x^*L$. We use the notation $i_{e_i}z:=z(e_i,\cdot)$. The Levi-Civita connection gives a decomposition  $T_{(x,y)}\mathcal{U}\simeq T_x^*L\oplus
T_xL$. Thus the vectors $(i_{e_i}z,e_i)$ span an $m$-plane in
$T_{(x,y)}\mathcal{U}$. We then set 
\begin{equation*}\label{eq:hatF}
F'(x,y,z):=\Phi_L^*\Imag\,\tilde{\Omega}_{|(x,y)}((i_{e_1}z,e_1),\dots,(i_{e_m}z,
e_m)).
\end{equation*}
If $f$ is a function on $L$ then, for each $x\in L$, the vectors $(i_{e_i}\nabla df,e_i)$ span the tangent plane to the graph of $df$. This yields the relationship
\begin{equation*}
 F(f)_{|x}=F'(x,df_{|x},\nabla^2 f_{|x}).
\end{equation*}
For any fixed $x\in L$, $y$ and $z$ vary in the linear space
$T^*_xL\oplus(T^*_xL\otimes T^*_xL)$ so Taylor's theorem shows
\begin{equation*}\label{eq:Ftaylor}
F'(x,y,z)=F'(x,0,0)+\partial_1 F'(x,0,0)\,y+\partial_2 F'(x,0,0)\,z+Q'(x,y,z),
\end{equation*}
where $\partial_1$ denotes differentiation with respect to the variable $y$, $\partial_2$ denotes differentiation with respect to the variable $z$ and $Q'=Q'(x,y,z)$ is a smooth function.

\begin{prop}\label{prop:Fproperties}
The map $F$ has the following properties:
 \begin{enumerate}
  \item The linearization of $F$ is the map
\begin{equation*}
P:=dF[0]: W^p_{3,\boldsymbol{\beta}}(L)\rightarrow W^p_{1,\boldsymbol{\beta}-2}(L), \ \
f\mapsto d^\star\left((\star\iota^*(\Real\,\tilde{\Omega}))df\right).
\end{equation*}
We can thus write
\begin{equation*}
F(f)=F(0)+P(f)+Q(f),
\end{equation*}
where $Q(f)_{|x}:=Q'(x,df_{|x},\nabla^2f_{|x})$.
\item There exists $C>0$ such that
\begin{equation*}
\begin{split}
\|Q(f)-Q(g)\|_{W^p_{1,\boldsymbol{\beta}-2}}
&\leq C\Bigl\{\left(\|f\|_{C^2_2}+\|g\|_{C^2_2}\right)\cdot\|f-g\|_{W^p_{3,\boldsymbol{\beta}}}\\
&\ \ \ \ \ \ \ +\left(\|f\|_{W^p_{3,\boldsymbol{\beta}}}+\|g\|_{W^p_{3,\boldsymbol{\beta}}}\right)\cdot\|f-g\|_{C^2_2}\Bigr\}\\
&\leq Cr\cdot\|f-g\|_{W^p_{3,\boldsymbol{\beta}}}.
\end{split}
\end{equation*}
 \end{enumerate}
\end{prop}
\begin{proof}
 We can compute the linearization of $F$ as in \cite{joyce:III} Proposition 5.6. This gives (1).

Our proof of (2) again follows \cite{joyce:III}. In his Proposition 5.8 Joyce provides $C^1$ estimates for analogous quantities $|Q(\alpha)-Q(\beta)|$. In his set-up the manifolds are compact and there are no weights. The result depends upon certain $C^2$ estimates for his map $F'$, cf. \cite{joyce:III} Equation 24, obtained via a particular connection. In our setting we want to set up an analogous proof, this time keeping track of behaviour with respect to $\rho$. The first step is to introduce Joyce's connection.

In general, let $E\rightarrow L$ be a vector bundle over $L$. Let $\mathcal{E}$ denote the total space of $E$, \textit{i.e.} the underlying differentiable manifold. Assume we want to build a connection on $\mathcal{E}$, allowing us to differentiate vector fields. Choosing a connection $\nabla^E$ on $E$ gives a splitting of the tangent bundle $T\mathcal{E}=H\oplus V$ into ``horizontal'' and ``vertical'' subbundles. It is then sufficient to indicate how to differentiate horizontal or vertical vector fields in horizontal or vertical directions, at the generic point $(x,e)\in\mathcal{E}$. Recall the canonical isomorphisms $H_{(x,e)}\simeq T_xL$, $V_{(x,e)}\simeq E_x$. We then see that some combinations do not require a connection. For example, the above identifications allow us to reduce the problem of differentiating vertical fields in vertical directions to the problem of differentiating a map $T_xE\rightarrow T_xE$: this is a map between vector spaces so the ordinary notion of differentiation suffices. Likewise for horizontal vector fields and vertical directions. Again using the identifications, the problem of differentiating a vertical field in a horizontal direction can instead be solved using $\nabla^E$, while horizontal vector fields can be differentiated in horizontal directions by choosing a connection on $L$, \textit{e.g.} the Levi-Civita connection induced by some metric $g$ on $L$. The result of this process is a connection on $\mathcal{E}$ which is not torsion-free: one can check that its torsion depends on the Riemannian curvature of $(L,g)$, cf. \cite{joyce:III} p.15.

Consider the case $L:=\mathcal{C}\subset\C^m$ endowed with the induced metric $\tg$ as in Section \ref{ss:conifold_defs}, and $E:=T^*\mathcal{C}$. Let $\tnabla$ denote both the Levi-Civita connection on $\mathcal{C}$  and the induced connection on $T^*\mathcal{C}$. In particular, $\tnabla$ is $t$-invariant with respect to rescaling, \textit{i.e.} $t^*\tnabla=\tnabla$. This implies, for example, that for vector fields at $(\theta,r)\in \mathcal{C}$, $t_*(\tnabla_vX_{|(\theta,r)})=\tnabla_{t_*v}t_*X_{|(\theta,tr)}$. Since the connection on $T^*\mathcal{C}$ is also $t$-invariant, the connection $\nabla$ on the total space of $T^*\mathcal{C}$, built as above, is $t$-invariant. 

We can use this connection to differentiate differential forms defined on the total space of $T^*\mathcal{C}$. In particular we can differentiate the form $\beta:=\Phi_{\mathcal{C}}^*(\Imag \tilde{\Omega})$. Then
\begin{align*}
 |\nabla^k\beta|_{\tg}(t(\theta,r,\alpha))&=(t^*(|\nabla^k\beta|_{\tg}))(\theta,r,\alpha)=|t^*\nabla^k\beta|_{t^*\tg}(\theta,r,\alpha)\\
&=|\nabla^kt^*\beta|_{t^2\tg}(\theta,r,\alpha)=t^{-m-k}t^m|\nabla^k\beta|_{\tg}(\theta,r,\alpha),
\end{align*}
where we use the fact that $\Phi_{\mathcal{C}}$ is $t$-equivariant and $t^*\tilde{\Omega}=t^m\tilde{\Omega}$, thus $t^*\beta=t^m\beta$. In particular, this proves that 
\begin{equation*}
 |\nabla^k\beta|_{\tg}(\theta,r,r^2\alpha_1+r\alpha_2dr)=r^{-k}|\nabla^k\beta|_{\tg}(\theta,1,\alpha_1+\alpha_2dr),
\end{equation*}
showing that $|\nabla^k\beta|_{\tg}=O(r^{-k})$, both for $r\rightarrow 0$ and for $r\rightarrow\infty$. 
It follows that, for our Lagrangian conifold $L$ and for all $k\in\N$, there exists $C_k>0$ such that, on $\mathcal{U}$,
\begin{equation}\label{eq:pullback_estimates}
|\nabla^k(\Phi_L^*\Imag\tilde{\Omega})|_g\leq C_k(\rho^{-k}).
\end{equation} 
As in \cite{joyce:III} Proposition 5.8 (Equation 24), these estimates show that, on $\mathcal{U}$, 
\begin{equation*}
 |\partial_1^2F'|\leq C\rho^{-2},\ \ |\partial_1\partial_2F'|\leq C\rho^{-1},\ \ |\partial_2^2F'|\leq C,
\end{equation*}
where all norms are calculated with respect to $g$. The constant $C$ depends only on $C_0$, $C_1$ and $C_2$ above. 
We can now prove (2) exactly as in \cite{joyce:III}. Specifically, fix $x$ and choose $y,y'\in T^*_xL$ and $z,z'\in T^*_xL\otimes T^*_xL$. Then, as in Remark \ref{rem:taylor} and using the above estimates, we find
that there exists $C>0$ such that
\begin{equation*}
\begin{split}
|Q'(x,y,z)-Q'(x,y',z')|\leq 
&C(\rho^{-1}|y-y'|+|z-z'|)\\
&\cdot(\rho^{-1}|y|+\rho^{-1}|y'|+|z|+|z'|).
\end{split}
\end{equation*}
Substituting $y=df$, $y'=dg$, $z=\nabla^2f$, $z'=\nabla^2g$, we obtain that, for each $x\in L$,
\begin{equation}\label{eq:Q}
\begin{split}
|Q(f)-Q(g)|
&\leq C(\|f\|_{C^2_2}+\|g\|_{C^2_2})(\rho^{-1}|df-dg|+|\nabla^2f-\nabla^2g|).\\
\end{split}
\end{equation}
Likewise, again as in \cite{joyce:III}, there exists $C>0$ such that, for each $x\in L$,
\begin{multline*}
|\rho d\left(Q(f)-Q(g)\right)|\\
\begin{split}
\leq C\Bigl\{&\rho^{-1}|df-dg|\cdot\rho^{-1}(|df|+|dg|)+\rho^{-1}|df-dg|\cdot(|\nabla^2f|+|\nabla^2g|)\\
&+\rho^{-1}|df-dg|\cdot\rho(|\nabla^3f|+|\nabla^3g|)+|\nabla^2f-\nabla^2g|\cdot\rho^{-1}(|df|+|dg|)\\
&+|\nabla^2f-\nabla^2g|\cdot(|\nabla^2f|+|\nabla^2g|)+|\nabla^2f-\nabla^2g|\cdot\rho(|\nabla^3 f|+|\nabla^3 g|)\\
&+\rho(|\nabla^3f-\nabla^3g|)\cdot\rho^{-1}(|df|+|dg|)+\rho(|\nabla^3 f-\nabla^3 g|)
\cdot(|\nabla^2 f|+|\nabla^2 g|)\Bigr\},\\
\end{split}
\end{multline*}
where all norms are calculated with respect to $g$. On the right hand side, consider the third and sixth terms: we can bound their first factors with $C^2_2$ norms. Now consider the remaining terms: we can bound their second factors with $C^2_2$ norms. We thus obtain
\begin{equation}\label{eq:dQ}
\begin{split}
|\rho d\left(Q(f)-Q(g)\right)|
\leq C\Bigl\{&\left(\rho^{-1}|df-dg|+|\nabla^2f-\nabla^2g|+\rho|\nabla^3f-\nabla^3g|\right)\\
&\cdot(\|f\|_{C^2_2}+\|g\|_{C^2_2})+\|f-g\|_{C^2_2}\cdot\rho(|\nabla^3f|+|\nabla^3g|)\Bigr\}.\\
\end{split}
\end{equation}
As in Definition \ref{def:csac_sectionspaces}, set $w:=\rho^{-\boldsymbol{\beta}}$. Multiplying both sides of Equations \ref{eq:Q}, \ref{eq:dQ} by $w\rho^{2-m/p}$, raising to the power $p$ and integrating, we obtain
\begin{equation*}
 \begin{split}
\|Q(f)-Q(g)\|_{L^p_{\boldsymbol{\beta}-2}}
&\leq C(\|f\|_{C^2_2}+\|g\|_{C^2_2})\cdot\|f-g\|_{W^p_{2,\boldsymbol{\beta}}},\\
\|d\left(Q(f)-Q(g)\right)\|_{L^p_{\boldsymbol{\beta}-3}}
&\leq C\Bigl\{(\|f\|_{C^2_2}+\|g\|_{C^2_2})\cdot\|f-g\|_{W^p_{3,\boldsymbol{\beta}}}\\
&\ \ \ \ \ \ \ +(\|f\|_{W^p_{3,\boldsymbol{\beta}}}
+\|g\|_{W^p_{3,\boldsymbol{\beta}}})\cdot\|f-g\|_{C^2_2}\Bigr\}.\\
\end{split}
\end{equation*}
This proves the first inequality in (2). Choosing $f$, $g$ in $B_r$ and using the Sobolev Embedding Theorem and our bounds on $\boldsymbol{\beta}$, relative to $2$, we prove the second inequality.
\end{proof}

\subsubsection*{Lagrangian conifolds with moving singularities} Let $(L,\iota)$ be a CS/AC Lagrangian conifold in $\C^m$ with $s$ CS ends. Let $p_i:=\iota(x_i)$ denote the corresponding singular points in $\C^m$. Assume either case of Lemma \ref{lemma:extra} holds. In Section \ref{ss:conifold_defs} we defined a space $\E$ which allows the singularities to translate and rotate in $\C^m$. Since both cases of Lemma \ref{lemma:extra} imply that the cone $\mathcal{C}_i$ corresponding to each singularity $x_i$ is SL, it is useful to modify the previous definitions, as follows. Set
\begin{equation}
\tilde{P}:=\{(p,\upsilon):p\in \C^m,\ \upsilon\in \sunitary m\},
\end{equation}
so that $\tilde{P}$ is a $\sunitary m$-principal fibre bundle over $\C^m$ of
dimension $m^2+2m-1$. For each CS end, the SL cone $\cone_i$ will have symmetry
group $G_i\subset \sunitary m$. As in Section \ref{ss:conifold_defs}, let
$\tilde{\E}_i$ denote a smooth submanifold of $\tilde{P}$ transverse to the
orbits of $G_i$ and containing $(p_i,Id)$. It has dimension $m^2+2m-1-\mbox{dim}(G_i)$. Set
$\tilde{\E}:=\tilde{\E}_1\times\dots\times\tilde{\E}_s$. We will restrict our attention to $\tilde{e}\in\tilde{\mathcal{E}}$ obtained as small perturbations of $e:=((p_1,Id),\cdots (p_s,Id))$. We then define Lagrangian conifolds $(L,\iota^{\tilde{e}})$ and embeddings $\Phi_L^{\tilde{e}}$ with
the same properties as in Section \ref{ss:conifold_defs}. As in Remark \ref{rem:phitilde}, we can think of $\tilde{e}$ as a compactly-supported symplectomorphism of $\C^m$ extending the corresponding element of $\sunitary m\times \C^m$ near each $p_i$. In this case we set $\Phi_L^{\tilde{e}}:=\tilde{e}\circ\Phi_L$. In order to use the formalism of Lemma \ref{l:taylor} we fix a local coordinate system on $\tilde{\E}$ and a norm on $T_e\tilde{\E}$. This will allow us to locally identify $\tilde{\E}$ with an open subset of a normed vector space. In particular, we identify $e$ with the origin. Locally, this process induces a distance $\mbox{d}(\cdot,\cdot)$ on $\tilde{\E}$. 

The variable $\tilde{e}$ thus defines a new class of Lagrangian deformations of $(L,\iota)$. For future reference we also want to allow a second finite-dimensional space of Lagrangian deformations. For any $i=1,\dots,s$ choose a smooth function $v_i$ on $L$ such that $v_i\equiv 1$ on the CS end $S_i$ and $v_i\equiv 0$ on the other ends, so that the $1$-forms $dv_i$ have support contained in the compact subset $K\subset L$. Let $E_0$ denote the $s$-dimensional vector space generated by these functions. Notice that, for any $v\in E_0$, the 1-form $dv$ has support in $K$ and is exact, thus closed. However it is not exact with respect to functions contained in the space $W^p_{3,\boldsymbol{\beta}}(L)$. Notice also that any non-zero $v\in E_0$ must interpolate between the value $1$ on at least one CS end and the value $0$ on the AC ends of $L$. In particular it cannot be constant, so $dv\neq 0$.
Given $v\in E_0$, set $\|v\|:=\|dv\|_{W^p_{2,\boldsymbol{\beta}-1}}$. This defines a norm on $E_0$. 

Using the ball $\tilde{B}_r\subset \tilde{\E}\times E_0\times W^p_{3,\boldsymbol{\beta}}(L)$, consider the map
\begin{equation}\label{eq:csFsobolev}
\tilde{F}:\tilde{B}_r\rightarrow W^p_{1,\boldsymbol{\beta}-2}(L), \ \
(\tilde{e},v,f)\mapsto \star((\Phi^{\tilde{e}}_L\circ (dv+df))^*\Imag\,\tilde{\Omega}),
\end{equation}
where $\star$ denotes the Hodge star operator of $L$. Since, near each singularity, $\tilde{e}$ preserves all structures on $\C^m$ and $dv=0$, $\tilde{F}(\tilde{e},v,f)=\tilde{F}(e,0,f)=F(f)$ near the singularities; this implies that $\tilde{F}$ is well-defined simply because $F$ is well-defined. 

According to Definition \ref{def:cy}, the deformed immersion  $(L,\Phi^{\tilde{e}}_L\circ (dv+df))$ is SL if and only if $\tilde{F}(\tilde{e},v,f)=0$. Setting
\begin{equation*}\label{eq:tildeF'}
\tilde{F}':\tilde{\E}\times\left(\mathcal{U}\oplus(T^*L\otimes T^*L)\right)\rightarrow \R,
\ \ (\tilde{e},x,y,z)\mapsto(\Phi^{\tilde{e}}_L)^*\Imag\,\tilde{\Omega}_{|(x,y)}((i_{e_1}z,e_1),\dots,(i_{e_m}z,
e_m)),
\end{equation*}
we obtain $\tilde{F}(\tilde{e},v,f)_{|x}=\tilde{F}'(\tilde{e},x,dv_{|x}+df_{|x},\nabla^2v_{|x}+\nabla^2f_{|x})$.
Smoothness of $\tilde{F}'$ with respect to the variable $\tilde{e}$ can be discussed as in Remark \ref{rem:phitilde}. 

As in Lemma \ref{l:taylor} and using our local coordinate system on $\tilde{\E}$, for any fixed $x\in L$ we can write
\begin{equation*}\label{eq:tildeFtaylor}
\begin{split}
\tilde{F}'(\tilde{e},x,y,z)=&\tilde{F}'(e,x,0,0)+\partial_0 \tilde{F}'(e,x,0,0)\,\tilde{e}+\partial_1 \tilde{F}'(e,x,0,0)\,y+\partial_2 \tilde{F}'(e,x,0,0)\,z\\
&+\tilde{Q}'(\tilde{e},x,y,z),
\end{split}
\end{equation*}
where $\partial_0$ denotes differentiation with respect to the variable $\tilde{e}$ and  $\tilde{Q}'=\tilde{Q}'(\tilde{e},x,y,z)$ is a smooth function.
\begin{prop}\label{prop:tildeFproperties}
The map $\tilde{F}$ has the following properties:
 \begin{enumerate}
  \item There exists an injective linear map $\chi:T_e\tilde{\E}\rightarrow
C^\infty_{\boldsymbol{0}}(L)$ such that (i) $\chi(\tilde{e})\equiv 0$ away from the
singularities and (ii) the linearized map $\tilde{P}:=d\tilde{F}[e,0,0]$ is of the form
\begin{equation*}
\tilde{P}: T_e\tilde{\E}\oplus E_0\oplus W^p_{3,\boldsymbol{\beta}}(L)\rightarrow W^p_{1,\boldsymbol{\beta}-2}(L), \ \
(\tilde{e},v,f)\mapsto d^\star\left((\star\iota^*(\Real\,\tilde{\Omega}))(d\,\chi(\tilde{e})+dv+df)\right).
\end{equation*}
We can thus write
\begin{equation*}
\tilde{F}(\tilde{e},v,f)=\tilde{F}(e,0,0)+\tilde{P}(\tilde{e},v,f)+\tilde{Q}(\tilde{e},v,f),
\end{equation*}
where $\tilde{Q}(\tilde{e},v,f)_{|x}:=\tilde{Q}'(\tilde{e},x,dv_{|x}+df_{|x},\nabla^2v_{|x}+\nabla^2f_{|x})$.
\item There exists $C>0$ such that
\begin{equation*}
\begin{split}
\|\tilde{Q}(\tilde{e},v,f)-\tilde{Q}(\tilde{d},u,g)\|_{W^p_{1,\boldsymbol{\beta}-2}}
&\leq Cr\cdot\left(\emph{d}(\tilde{e},\tilde{d})+\|v-u\|+\|f-g\|_{W^p_{3,\boldsymbol{\beta}}}\right).
\end{split}
\end{equation*}
 \end{enumerate}
\end{prop}
\begin{proof}
Concerning (1) we refer to \cite{pacini:sldefsextended}, Proposition 6.5. The idea is that, near each singularity, $\chi(\tilde{e})$ is a Hamiltonian function for $\tilde{e}\in T_{e}\tilde{\E}$ (thought of as a vector field on $\C^m$). One can write this function down explicitly to show that $\chi(\tilde{e})\in C^\infty_{\boldsymbol{0}}(L)$. On the other hand, as a symplectomorphism of $\C^m$, we defined $\tilde{e}$ so that it acts trivially away from the singularities. This implies that $\chi(\tilde{e})=0$ there. 
  
Fix $(\tilde{e},v,0)\in\tilde{B}_r$. This data defines a new Lagrangian conifold to which we can apply Proposition \ref{prop:Fproperties}, finding
\begin{equation}\label{eq:tildeQ_contraction}
\begin{split}
\|\tilde{Q}(\tilde{e},v,f)-\tilde{Q}(\tilde{e},v,g)\|_{W^p_{1,\boldsymbol{\beta}-2}}
&\leq Cr\cdot\|f-g\|_{W^p_{3,\boldsymbol{\beta}}}.
\end{split}
\end{equation}
By continuity and compactness, we can assume that $C$ is independent of $(\tilde{e},v)$. 

The same methods used to prove Proposition \ref{prop:Fproperties} allow us to prove, for fixed $(\tilde{e},0,f)$, that
\begin{equation*}
\begin{split}
|\tilde{Q}(\tilde{e},v,f)-\tilde{Q}(\tilde{e},u,f)|
&\leq C(\rho^{-1}|dv|+\rho^{-1}|du|+|\nabla^2v|+|\nabla^2u|)(\rho^{-1}|dv-du|+|\nabla^2v-\nabla^2u|)\\
&\leq C(\|dv\|_{C^1_1}+\|du\|_{C^1_1})\cdot (\rho^{-1}|dv-du|+|\nabla^2v-\nabla^2u|)\\
&\leq C(\|dv\|_{W^p_{2,\boldsymbol{\beta}-1}}+\|du\|_{W^p_{2,\boldsymbol{\beta}-1}})\cdot (\rho^{-1}|dv-du|+|\nabla^2v-\nabla^2u|),
\end{split}
\end{equation*}
where the last inequality relies on the Sobolev Embedding Theorem and the bounds on $\boldsymbol{\beta}$, relative to $2$.
Recall that in Proposition \ref{prop:Fproperties} the constant $C$ is defined after restricting to  $\mathcal{U}\subset T^*L$. Likewise, $C$ here is defined after restricting to a neighbourhood of the graph of $df$ in $T^*L$. The bounds on $f$ allow us to assume that this neighbourhood is contained in $\mathcal{U}$, so $C$ is independent of $f$. Again, continuity and compactness imply that it is also independent of $\tilde{e}$. Analogous results hold for $\rho dQ$. We can multiply both sides of these equations by $w\rho^{2-m/p}$, raise to the power $p$ and integrate, obtaining 
\begin{equation}\label{eq:tildeQ_contractionbis}
\|\tilde{Q}(\tilde{e},v,f)-\tilde{Q}(\tilde{e},u,f)\|_{W^p_{1,\boldsymbol{\beta}-2}}
\leq Cr\cdot\|v-u\|.\\
\end{equation}
Finally, fix $(e,v,f)\in \tilde{B}_r$ and $(x,y,z)$. Applying Lemma \ref{l:taylor} (3) with respect to the $\tilde{e}$ variable, there exists $C(x,y,z)>0$ such that
\begin{equation*}
 |\tilde{Q}'(\tilde{e},x,y,z)-\tilde{Q}'(\tilde{d},x,y,z)|\leq C(x,y,z)\cdot\mbox{d}(\tilde{e},\tilde{d})\cdot\left(\mbox{d}(\tilde{e},e)+\mbox{d}(\tilde{d},e)\right).
\end{equation*}
Recall that the action of $\tilde{\E}$ is trivial away from a neighbourhood of the singularities of $L$, \textit{i.e.} $\tilde{Q}'$ is independent of $\tilde{e}$ there. Furthermore, if we restrict $y$ and $z$ within a bounded set of their domains we can assume that $C$ is independent of $y$ and $z$. We can thus write $C=C(x)$, obtaining a function which is compactly supported in a neighbourhood of the singularities.

In order to substitute $y:=df(x)$ and $z:=\nabla^2f(x)$ we need to ensure that these values are appropriately bounded, as required above. This follows from the fact that the $C^2_2$ norm of $f$ is small, as already seen in the proof of Lemma \ref{lemma:extra}. We conclude that
\begin{equation*}
 |\tilde{Q}(\tilde{e},v,f)-\tilde{Q}(\tilde{d},v,f)|\leq C(x)\cdot\mbox{d}(\tilde{e},\tilde{d})\cdot\left(\mbox{d}(\tilde{e},e)+\mbox{d}(\tilde{d},e)\right).
\end{equation*}
First derivatives (with respect to the variable $x$) can be studied analogously, starting from the Taylor expansion (with respect to $\tilde{e}$) of $d\tilde{F}(\tilde{e},v,f)$. As usual, introducing weight factors and integrating we obtain 
\begin{equation}\label{eq:tildeQ_contractionter}
\|\tilde{Q}(\tilde{e},v,f)-\tilde{Q}(\tilde{d},v,f)\|_{W^p_{1,\boldsymbol{\beta}-2}}
\leq \|C\|_{W^p_{1,\boldsymbol{\beta}-2}}\left(\mbox{d}(\tilde{e},e)+\mbox{d}(\tilde{d},e)\right)\cdot \mbox{d}(\tilde{e},\tilde{d})\leq C r\cdot\mbox{d}(\tilde{e},\tilde{d}).\\
\end{equation}
Combining Equations \ref{eq:tildeQ_contraction}, \ref{eq:tildeQ_contractionbis}, \ref{eq:tildeQ_contractionter} via the triangle inequality gives (2).
\end{proof}
\begin{remark}
 As already mentioned following Equation \ref{eq:csFsobolev}, $\tilde{F}$ is independent of $\tilde{e}$ also near the singularities. We could use this fact to simplify parts of the proof of Proposition \ref{prop:tildeFproperties}. It also implies that $\tilde{P}(\tilde{e})$ has compact support away from the singularities. This argument however depends strongly on our choice of working in $\C^m$. In a general CY manifold $M$, the discrepancy between the CY structures on $M$ and on $\C^m$ introduces error terms which must be compensated for by appropriate restrictions on the rates of $(L,\iota)$, see \cite{pacini:sldefsextended} for details.
\end{remark}

\subsubsection*{Lagrangian connect sums} We now come to our main problem: how to find SL deformations of Lagrangian conifolds $(L_{\bt},\iota_{\bt})$ obtained as connect sums. Specifically, let $(L,\iota)$, $(\hat{L},\hat{\iota})$ be marked compatible Lagrangian conifolds with weights $\boldsymbol{\beta}$, $\hat{\boldsymbol{\beta}}$. 
We will study the existence of SL deformations of $(L_{\bt},\iota_{\bt})$ by considering maps $F_{\bt}$ and extensions $\tilde{F}_{\bt}$, defined as follows.

Let $\beta_i$ denote the value of $\boldsymbol{\beta}_{\bt}$ on the $i$-th neck of $L_{\bt}$ and choose $\alpha>\max{\{2-\beta_i\}}$, where $\max$ is taken over all necks. For each $\bt$, set $t:=\max\{t_i:t_i\in\bt\}$. Fix $p>m$. Consider the map 
\begin{equation}\label{eq:tFsobolev}
F_{\bt}:B_{t^\alpha}\subset W^p_{3,\boldsymbol{\beta_{\bt}}}(L_{\bt})\rightarrow W^p_{1,\boldsymbol{\beta}_{\bt}-2}(L_{\bt}), \ \
f\mapsto \star_{\bt}((\Phi_{L_{\bt}}\circ df)^*\Imag\,\tilde{\Omega}),
\end{equation}
where $\star_{\bt}$ denotes the Hodge star operator of $L_{\bt}$. We can think of $F_{\bt}$ as being obtained from an underlying smooth function
\begin{equation*}\label{eq:tF'}
F_{\bt}'=F_{\bt}'(x,y,z):\mathcal{U}_{\bt}\oplus(T^*L_{\bt}\otimes T^*L_{\bt})\rightarrow \R
\end{equation*}
defined as in Equation \ref{eq:F'}, via the relationship
\begin{equation*}
 F_{\bt}(f)_{|x}=F_{\bt}'(x,df_{|x},\nabla^2 f_{|x}).
\end{equation*}
As long as $F_{\bt}$ is well-defined, we can define $Q'_{\bt}$, $Q_{\bt}$ and we can linearize $F_{\bt}$ as above.

We can extend the domain of this map to include moving singularities, as follows. Let $s$ denote the number of CS ends in $S^{**}$ and $\hat{s}$ denote the number of CS ends in $\hat{S}^{**}$. 
Given any singularity $x_i\in S^{**}$, set $p_i:=\iota(x_i)=\iota_{\bt}(x_i)$ and define $\tilde{\E}_i$ as above, containing $e_i:=(p_i, Id)$. Using local coordinates on $\tilde{\E}_i$ and a norm $\|\cdot\|$ on $T_{e_i}\tilde{\E}_i$ we locally identify $\tilde{\E}_i$ with a normed vector space, thus obtaining a  distance $\mbox{d}_i$ on $\tilde{\E}_i$. Given any singularity $\hat{x}_i\in \hat{S}^{**}$, set $\hat{p}_i:=\hat{\iota}(\hat{x}_i)$ and again define $\tilde{\E}_i$ as above, containing $\hat{e}_i:=(\hat{p}_i, Id)$. We again use local coordinates and a norm $\|\cdot\|$ on $T_{\hat{e}_i}\tilde{\E}_i$ to endow it with a distance $\hat{\mbox{d}}_i$. 

Notice that the point $p_{t_i}:=\iota_{\bt}(\hat{x}_i)$ moves in $\C^m$ via dilation. We will thus find ourselves working with the ``rescaled'' objects $t_i\cdot\tilde{e}\cdot t_i^{-1}$,
where as in Remark \ref{rem:phitilde} we think of $\tilde{e}$ as a compactly-supported symplectomorphism of $\C^m$ so that it makes sense to compose it with dilations of $\C^m$. We will also need to rescale the norms on $T_{\hat{e}_i}\tilde{\E}_i$ and the corresponding distances, setting 
\begin{equation}\label{eq:tnorms}
\|\cdot\|_{t_i}:=t_i^{2-\beta_i}\|\cdot\|,\ \ \hat{\mbox{d}}_{t_i}:=t_i^{2-\beta_i}\hat{\mbox{d}}_i,
\end{equation}
 where $\beta_i$ is the weight on the necks corresponding to the $i$-th connected component of $\hat{L}$. 
Set
$\tilde{\E}_{\bt}:=\tilde{\E}_1\times\dots\times\tilde{\E}_{\hat{s}}\times\tilde{\E}_1\times\dots\times\tilde{\E}_s$, endowed with the distance $\mbox{d}_{\bt}:=\hat{\mbox{d}}_{t_1}+\cdots+\hat{\mbox{d}}_{t_{\hat{s}}}+\mbox{d}_1+\cdots+\mbox{d}_s$. Choose an element $\tilde{e}:=(\tilde{e}_1,\dots,\tilde{e}_{\hat{s}},\tilde{e}_1,\dots\tilde{e}_s)\in\tilde{\E}_{\bt}$. We can identify $\tilde{e}$ with the symplectomorphism obtained by composition of the individual $\tilde{e}_i$. Notice that the supports of the $\tilde{e}_i$ are all disjoint. We denote by $\bt\cdot\tilde{e}\cdot\bt^{-1}$ the rescaled object $(t_1\cdot\tilde{e}_1\cdot t_1^{-1},\dots,t_{\hat{s}}\cdot\tilde{e}_{\hat{s}}\cdot t_{\hat{s}}^{-1},\tilde{e}_1,\dots\tilde{e}_s)$.

For any $i=1,\dots,s$ choose a smooth function $v_i$ on $L$ such that $v_i\equiv 1$ on the CS end $S_i\in S^{**}$ and $v_i\equiv 0$ on the other ends, so that $dv_i$ has support in the compact subset $K\subset L$.
For any $i=1,\dots,\hat{s}$ choose a smooth function $v_i$ on $\hat{L}$ such that $v_i\equiv 1$ on the CS end $\hat{S}_i\in \hat{S}^{**}$ and $v_i\equiv 0$ on the other ends, so that $dv_i$ has support in the compact subset $\hat{K}\subset \hat{L}$. By extending these functions to zero, we can think of them as $\bt$-independent functions on $L_{\bt}$. 
Let $E_0$ denote the $(s+\hat{s})$-dimensional vector space generated by these functions. 
Given $v\in E_0$, set $\|v\|_{\bt}:=\|dv\|_{W^p_{2,\boldsymbol{\beta}_{\bt}-1}}$. This defines a norm on $E_0$. 

Using the ball $\tilde{B}_{t^\alpha}\subset \tilde{\E}_{\bt}\times E_0\times W^p_{3,\boldsymbol{\beta}_{\bt}}(L_{\bt})$, consider the map
\begin{equation}\label{eq:cstFsobolev}
\tilde{F}_{\bt}:\tilde{B}_{t^\alpha}\rightarrow W^p_{1,\boldsymbol{\beta}_{\bt}-2}(L_{\bt}), \ \
(\tilde{e},v,f)\mapsto \star_{\bt}((\Phi^{\bt\cdot\tilde{e}\cdot \bt^{-1}}_{L_{\bt}}\circ (dv+df))^*\Imag\,\tilde{\Omega}),
\end{equation}
where $\star_{\bt}$ denotes the Hodge star operator of $(L_{\bt},\iota_{\bt})$. We define functions $\tilde{F}'_{\bt}$ and $\tilde{Q}'_{\bt}$ as before. Notice that $\tilde{F}_{\bt}$ extends $F_{\bt}$ and that the deformed immersion  $(L_{\bt},\Phi^{\bt\cdot\tilde{e}\cdot\bt^{-1}}_{L_{\bt}}\circ (dv+df))$ is SL if and only if  $\tilde{F}_{\bt}(\tilde{e},v,f)=0$. 

\begin{prop}\label{prop:tFproperties}
Assume $\boldsymbol{\beta}_{\bt}$ and $(L_{\bt}, \iota_{\bt})$ satisfy either case of Lemma \ref{lemma:extra} and $\alpha$ is chosen as above. Then:
\begin{enumerate}
\item For all $M>1$ there exists $\epsilon>0$ such that, for all $\bt>0$ in the circular sector defined in $\bt$-space by the conditions
\begin{equation*}
 |\bt|<\epsilon,\ \ 1\leq\max\{t_i\}/\min\{t_i\}<M,
\end{equation*}
the maps $F_{\bt}$ and $\tilde{F}_{\bt}$ are well-defined.
\item There exists $C>0$ such that, for all $\bt$ as in (1), all $(\tilde{e},v,f)\in \tilde{B}_{t^\alpha}$ and all $x\in L_{\bt}$, 
\begin{equation*}
|\tilde{F}_{\bt}(\tilde{e},v,f)|\leq C.
\end{equation*}
There exists a function $C=C(x)>0$, compactly supported in a neighbourhood of the singularities, such that, for all $\bt$ as in (1), all $(\tilde{e},v,f)\in \tilde{B}_{t^\alpha}$ and all $x$ in the $i$-th component of $\hat{L}$, 
\begin{equation*}
\|\partial_0\tilde{F}_{\bt|\tilde{e},v,f}\|_{t_i}\leq C(x)t_i^{\beta_i-2}, \ \ 
\|\partial_0\partial_0\tilde{F}_{\bt|\tilde{e},v,f}\|_{t_i}\leq C(x)t_i^{2\beta_i-4},
\end{equation*}
where $\partial_0$ denotes differentiation with respect to the variable $\tilde{e}$, $\beta_i$ is the weight on the necks corresponding to the $i$-th component of $\hat{L}$, and the norms are defined as in Lemma \ref{l:taylor}.
\item For each $\bt$ as in (1), the linearized map $\tilde{P}_{\bt}:=d\tilde{F}_{\bt}[e,0,0]$ is of the form
\begin{equation*}
\tilde{P}_{\bt}: T_e\tilde{\E}_{\bt}\oplus E_0\oplus W^p_{3,\boldsymbol{\beta}_{\bt}}(L_{\bt})\rightarrow W^p_{1,\boldsymbol{\beta}_{\bt}-2}(L_{\bt}), \ \
(\tilde{e},v,f)\mapsto d^{\star_{\bt}}\left((\star_{\bt}\iota_{\bt}^*(\Real\,\tilde{\Omega}))(d\,\chi_{\bt}(\tilde{e})+dv+df)\right),
\end{equation*}
where $\chi_{\bt}=t_i^2\chi$ near the singularities of the $i$-th component of $\hat{L}$, $\chi_{\bt}=\chi$ near the singularities of $L$ and $\chi_{\bt}= 0$ elsewhere.
We can thus write
\begin{equation*}
\tilde{F}_{\bt}(\tilde{e},v,f)=\tilde{F}_{\bt}(e,0,0)+\tilde{P}_{\bt}(\tilde{e},v,f)+\tilde{Q}_{\bt}(\tilde{e},v,f),
\end{equation*}
where $\tilde{Q}_{\bt}(\tilde{e},v,f)_{|x}:=\tilde{Q}_{\bt}'(\tilde{e},x,dv_{|x}+df_{|x},\nabla^2v_{|x}+\nabla^2f_{|x})$.
\item There exists $C>0$ such that, for each $\bt$ as in (1),
\begin{equation*}\label{eq:cstQ_contraction}
\|\tilde{Q}_{\bt}(\tilde{e},v,f)-\tilde{Q}_{\bt}(\tilde{d},u,g)\|_{W^p_{1,\boldsymbol{\beta}_{\bt}-2}}
\leq C t^{\alpha+\min\{\beta_i-2\}}\left(\emph{d}_{\bt}(\tilde{e},\tilde{d})+\|v-u\|_{\bt}+\|f-g\|_{W^p_{3,\boldsymbol{\beta}_{\bt}}}\right).
\end{equation*}
\end{enumerate}
\end{prop}
\begin{proof}
As in the proof of Lemma \ref{lemma:extra}, to show that $F_{\bt}$ and $\tilde{F}_{\bt}$ are well-defined the main issue is to prove that the $C^2_2$ norm of any $f\in B_{t^{\alpha}}$ is sufficiently small. Theorem \ref{thm:embedding} and Lemma \ref{l:changing_weight_estimates} give embeddings
\begin{equation}\label{eq:composed_embeddings}
 W^p_{3,\boldsymbol{\beta}_{\bt}}\hookrightarrow C^2_{\boldsymbol{\beta}_{\bt}}\hookrightarrow C^2_2
\end{equation}
such that, for any function $f\in W^p_{3,\boldsymbol{\beta}_{\bt}}$,
\begin{equation*}
 \|f\|_{C^2_2}\leq C\max\{t_i^{\beta_i-2}\}\|f\|_{W^p_{3,\boldsymbol{\beta}_{\bt}}}\leq C(\min\{t_i\})^{\min\{\beta_i-2\}}\|f\|_{W^p_{3,\boldsymbol{\beta}_{\bt}}}.
\end{equation*}
For $f\in B_{t^\alpha}$ this implies
\begin{equation*}
\begin{split}
 \|f\|_{C^2_2} &\leq C\min{\{t_i\}}^{\min{\{\beta_i-2\}}}t^\alpha=C\left(\frac{\min\{t_i\}}{t}\right)^{\min\{\beta_i-2\}}t^{\alpha+\min\{\beta_i-2\}}\\
&\leq C(1/M)^{\min\{\beta_i-2\}}t^{\alpha+\min\{\beta_i-2\}},
\end{split}
\end{equation*}
 which is small when $t$ is small, by definition of $\alpha$. This proves (1).

Concerning (2), write $\tilde{F}_{\bt}(\tilde{e},v,f)=\star_{\bt}(dv+df)^*\Phi_{L_{\bt}}^*(\bt\cdot\tilde{e}\cdot \bt^{-1})^*\Imag\, \tilde{\Omega}$. As seen above, the condition $f\in \tilde{B}_{t^\alpha}$ implies that the second derivatives of $f$ are uniformly bounded. To bound $\tilde{F}_{\bt}$ it is thus sufficient to bound the $m$-form $\Phi_{L_{\bt}}^*(\bt\cdot\tilde{e}\cdot \bt^{-1})^*\Imag\, \tilde{\Omega}$ on $\mathcal{U}_{\bt}$. Recall also that on $\mathcal{U}_{\bt}$ we use the pull-back metric. To prove (2) it is thus sufficient to bound the $m$-form $(\bt\cdot\tilde{e}\cdot \bt^{-1})^*\Imag\, \tilde{\Omega}$ on $\C^m$. Notice that $t_i^*\Imag\,\tilde{\Omega}=t_i^m\Imag\,\tilde{\Omega}$ and that there exists $C>0$ such that
\begin{equation*}|\tilde{e}^*\Imag\, \tilde{\Omega}|\leq C,\ \ |(t_i^{-1})^*\tilde{e}^*\Imag\, \tilde{\Omega}|\leq Ct_i^{-m},
 \end{equation*}
because the symplectomorphism $\tilde{e}$ has compact support. By continuity, we may assume $C$ is independent of $\tilde{e}\in \tilde{B}_{t^\alpha}$. This shows that $(\bt\cdot\tilde{e}\cdot \bt^{-1})^*\Imag\, \tilde{\Omega}$ is bounded independently of $\bt$ and $\tilde{e}$, thus yielding the desired bound on $\tilde{F}_{\bt}$. It follows from Remark \ref{rem:phitilde} that derivatives with respect to $\tilde{e}$ can be written in terms of Lie derivatives, \textit{e.g.} given $X\in T_{\tilde{e}}\tilde{\E}_{\bt}$,
\begin{equation*}
\partial_0\tilde{F}_{\bt|\tilde{e},v,f}(X) =\star_{\bt}((\Phi_{L_{\bt}}^{\bt\cdot\tilde{e}\cdot\bt^{-1}}\circ(dv+df))^*\mathcal{L}_X\Imag\,\tilde{\Omega}).
\end{equation*}
The same arguments then prove the desired bounds for such derivatives, but we now allow $C$ to depend on $x$ to emphasize the fact that, away from the singularities, $\tilde{e}$ acts trivially so $\partial_0\tilde{F}_{\bt}=0$. The factors of $t_i$ are due to the definition of $\|\cdot\|_{t_i}$, cf. Equation \ref{eq:tnorms}.

To prove (3) we start by noticing that Proposition \ref{prop:tildeFproperties} holds for any Lagrangian conifold, thus for any $L_{\bt}$ for fixed $\bt$. This shows that $\tilde{P}_{\bt}$ is of the claimed form for some $\chi_{\bt}$ which vanishes away from the singularities. It is clear that, near the singularities of $L$, 
\begin{equation}\label{eq:Ft_independent}
 \tilde{F}_{\bt}(\tilde{e},0,0)=\star((\tilde{e}\circ\iota)^*\Imag\,\tilde{\Omega})=\tilde{F}(\tilde{e},0,0).
\end{equation}
We now want to show that this is true also near the singularities of $\hat{L}$. 
To this end, recall that $\Phi_{L_{\bt}}\circ 0=\iota_{\bt}$, which (near the singularities of $\hat{L}$) is simply a rescaling of $\iota$, so for $(v,f)=(0,0)$ we obtain 
\begin{equation*}
\begin{split}
\tilde{F}_{\bt}(\tilde{e},0,0)&=\star_{\bt}((\bt\cdot\tilde{e}\cdot\bt^{-1}\circ\iota_{\bt})^*\Imag\,\tilde{\Omega})\\
&=\star_{\bt}((\bt\cdot\tilde{e}\circ\iota)^*\Imag\,\tilde{\Omega})=\star_{\bt}((\tilde{e}\circ\iota)^*\bt^*\Imag\,\tilde{\Omega}).
\end{split}
\end{equation*}
Now recall how the Hodge star operator behaves under rescaling of the metric: on $k$-forms, $\star_{t_i}=t_i^{m-2k}\star$. This implies that the $\bt$ factors above cancel, proving Equation \ref{eq:Ft_independent}. This equation shows that, near each singularity, $\partial_0\tilde{F}_{\bt|e,0,0}=\partial_0\tilde{F}_{|e,0,0}=d^\star((\star\iota^*(\Real\,\tilde{\Omega}))d\,\chi)$. Again using the rescaling properties of $\star$, we find $\star\iota^*(\Real\,\tilde{\Omega})=\star_{\bt}\iota_{\bt}^*(\Real\,\tilde{\Omega})$ and $d^\star=t_i^2d^{\star_{t_i}}$. This completes the proof of (3).

The proof of (4) requires several steps. To begin with, let us prove the equivalent statement for the restricted map $Q_{\bt}$. Substituting scale-invariant norms into Equation \ref{eq:pullback_estimates} shows that, on $\mathcal{U}_{\bt}$, 
\begin{equation*}
 |\nabla^k(\Phi_{L_{\bt}}^*\Imag\tilde{\Omega})|_{g_{\bt}}\leq C_k(\rho_{\bt}^{-k}),
\end{equation*}
where $C_k$ are independent of $\bt$. This leads to estimates
\begin{equation*}
 |\partial_1^2F'_{\bt}|\leq C\rho_{\bt}^{-2},\ \ |\partial_1\partial_2F'_{\bt}|\leq C\rho_{\bt}^{-1},\ \ |\partial_2^2F'_{\bt}|\leq C,
\end{equation*}
where all norms are calculated with respect to $g_{\bt}$ and $C$ depends on $C_0$, $C_1$ and $C_2$ above.
It follows that there exists $C>0$ such that, for each $\bt$ and each $x\in L_{\bt}$,
\begin{equation*}
\begin{split}
|Q_{\bt}(f)-Q_{\bt}(g)|
&\leq C(\|f\|_{C^2_2}+\|g\|_{C^2_2})(\rho_{\bt}^{-1}|df-dg|+|\nabla^2f-\nabla^2g|).\\
\end{split}
\end{equation*}
\begin{equation*}
\begin{split}
|\rho_{\bt}d\left(Q_{\bt}(f)-Q_{\bt}(g)\right)|
\leq C\Bigl\{&\left(\rho_{\bt}^{-1}|df-dg|+|\nabla^2f-\nabla^2g|+\rho_{\bt}|\nabla^3f-\nabla^3g|\right)\\
&\cdot(\|f\|_{C^2_2}+\|g\|_{C^2_2})+\|f-g\|_{C^2_2}\cdot\rho_{\bt}(|\nabla^3f|+|\nabla^3g|)\Bigr\}.\\
\end{split}
\end{equation*}
Multiplying both sides of these equations by $w_{\bt}\rho_{\bt}^{2-m/p}$, raising to the power $p$ and integrating, we obtain 
\begin{equation*}
\begin{split}
\|Q_{\bt}(f)-Q_{\bt}(g)\|_{W^p_{1,\boldsymbol{\beta}_{\bt}-2}}
&\leq C\Bigl\{\left(\|f\|_{C^2_2}+\|g\|_{C^2_2}\right)\cdot\|f-g\|_{W^p_{3,\boldsymbol{\beta}_{\bt}}}\\
&\ \ \ \ \ \ \ +\left(\|f\|_{W^p_{3,\boldsymbol{\beta}_{\bt}}}+\|g\|_{W^p_{3,\boldsymbol{\beta}_{\bt}}}\right)\cdot\|f-g\|_{C^2_2}\Bigr\}.
\end{split}
\end{equation*}
As in Proposition \ref{prop:tildeFproperties}, we can apply this estimate to the Lagrangian conifold defined by any $(\tilde{e},v,0)\in \tilde{B}_{t^\alpha}$. Choosing $f, g\in B_{t^{\alpha}}$ and using the embeddings of Equation \ref{eq:composed_embeddings} then gives
\begin{equation*}
\|\tilde{Q}_{\bt}(\tilde{e},v,f)-\tilde{Q}_{\bt}(\tilde{e},v,g)\|_{W^p_{1,\boldsymbol{\beta}_{\bt}-2}}
\leq C t^{\alpha+\min\{\beta_i-2\}}\|f-g\|_{W^p_{3,\boldsymbol{\beta}_{\bt}}}.
\end{equation*}
 As in Proposition \ref{prop:tildeFproperties}, we can prove that $C$ is independent of $(\tilde{e},v)$. 

The same methods, but this time using the embeddings $W^p_{2,\boldsymbol{\beta}_{\bt}-1}\hookrightarrow C^1_{\boldsymbol{\beta}_{\bt}-1}\hookrightarrow C^1_1$ for spaces of $1$-forms, show that 
\begin{equation*}
\|\tilde{Q}_{\bt}(\tilde{e},v,f)-\tilde{Q}_{\bt}(\tilde{e},u,f)\|_{W^p_{1,\boldsymbol{\beta}_{\bt}-2}}
\leq C t^{\alpha+\min\{\beta_i-2\}}\|v-u\|_{\bt}.
\end{equation*}
Finally, fix $(e,v,f)\in \tilde{B}_{t^\alpha}$. It follows from the estimates in (2) and from Lemma \ref{l:taylor} that, on the $i$-th component of $\hat{L}$,
\begin{equation*}
 |\tilde{Q}_{\bt}(\tilde{d},v,f)-\tilde{Q}_{\bt}(\tilde{e},v,f)|\leq C(x)t_i^{2\beta_i-4}\cdot\mbox{d}_{\bt}(\tilde{d},\tilde{e})\cdot\left(\mbox{d}_{\bt}(\tilde{d},e)+\mbox{d}_{\bt}(\tilde{e},e)\right).
\end{equation*}
First derivatives (with respect to $x$) can be studied analogously. Multiplying by the appropriate weight factors and integrating, on the $i$-th component we obtain
\begin{equation*}
\begin{split}
 \|\tilde{Q}_{\bt}(\tilde{d},v,f)-\tilde{Q}_{\bt}(\tilde{e},v,f)\|_{W^p_{1,\boldsymbol{\beta}_{\bt}-2}}&\leq \|C\|_{W^p_{1,\boldsymbol{\beta}_{\bt}-2}}t_i^{2\beta_i-4}\cdot\mbox{d}_{\bt}(\tilde{d},\tilde{e})\cdot\left(\mbox{d}_{\bt}(\tilde{d},e)+\mbox{d}_{\bt}(\tilde{e},e)\right)\\
&\leq Ct_i^{\beta_i-2} \mbox{d}_{\bt}(\tilde{d},\tilde{e})\cdot\left(\mbox{d}_{\bt}(\tilde{d},e)+\mbox{d}_{\bt}(\tilde{e},e)\right)\\
&\leq C\max\{t_i^{\beta_i-2}\}t^{\alpha}\mbox{d}_{\bt}(\tilde{d},\tilde{e})\\
&\leq Ct^{\alpha+\min\{\beta_i-2\}}\mbox{d}_{\bt}(\tilde{d},\tilde{e}),\\
\end{split}
\end{equation*}
where we use the fact that, on the support of $C(x)$, the metric is simply rescaled so $\|C\|_{W^p_{1,\boldsymbol{\beta}_{\bt}-2}}\leq Ct_i^{2-\beta_i}$. Estimates on the components of $L$ are similar but slightly stronger, because here the data does not depend on $\bt$: the upper bound is thus of the form $Ct^\alpha\mbox{d}_{\bt}(\tilde{d},\tilde{e})$.

Combining the above estimates via the triangle inequality proves (4).
\end{proof}


\section{Connect sums of SL conifolds}\label{s:SL_sum}

Let $(L,\iota)$, $(\hat{L},\hat{\iota})$ be marked compatible SL conifolds in $\C^m$. Their connect sum $(L_{\boldsymbol{t}},\iota_{\boldsymbol{t}})$ is a Lagrangian conifold. By construction it is SL except in the compact subset where gluing occurred. The goal of this paper is to prove that there exists a small perturbation $\iota'_{\boldsymbol{t}}$ of $\iota_{\boldsymbol{t}}$ such that $(L_{\bt},\iota'_{\bt})$ is SL. The proof will depend on careful choices of weights $\boldsymbol{\beta}$, $\hat{\boldsymbol{\beta}}$ and of constants $p$, $\tau$, $\alpha$ used in the gluing and perturbation. We will distinguish two cases and some choices will vary accordingly. The following assumptions are however common to both cases:
\begin{description}
\item[A1] Let $\mu_i$, respectively $\hat{\lambda}_i$, denote the convergence rates of $L$, respectively $\hat{L}$, on the marked ends. We choose the parameter $\tau\in (0,1)$, used in the construction of $(L_{\boldsymbol{t}},\iota_{\boldsymbol{t}})$, sufficiently close to $1$ as in Lemma \ref{l:induced_metric}: specifically, such that
\begin{equation*}
 \label{eq:tau}
\tau>\frac{2-\hat{\lambda}_i}{\mu_i-\hat{\lambda}_i}.
\end{equation*}
\item[A2] We assume that the weights $\boldsymbol{\beta},\hat{\boldsymbol{\beta}}$ are compatible as in Definition \ref{def:compatible}. We can then define $\boldsymbol{\beta}_{\bt}$ as in Section \ref{s:lagr_sum}. In particular,  on the $i$-th neck $\boldsymbol{\beta}_{\bt}$ takes the constant value $\beta_i=\hat{\beta}_i$. 

We further assume that 
$\hat{\beta}_i\in (2-m,0)$ for all ends $\hat{S}_i\in\hat{S}^{*}$
and that $\hat{\beta}_i>\hat{\lambda}_i$: this is possible because our connect sum construction requires $\hat{\lambda}_i<0$.

We finally assume that, for all necks, the difference $|\beta_i-\beta_j|$ is sufficiently small: Lemma \ref{l:alpha} can be used to quantify this statement precisely. Notice however that we could simply choose weights such that $\beta_i=\beta_j$ for all necks in $L_{\bt}$. 
\item[A3] We choose $p>m$ and $\alpha$ such that
\begin{equation*}
  \max{\{2-\beta_i\}}<\alpha<\min{\{(1-\tau)(2-\hat{\lambda}_i)+\tau(2-\beta_i)\}},
 \end{equation*}
where $\max$ and $\min$ are taken over all necks. In particular, $\alpha+\min\{\beta_i-2\}>0$.
\end{description}
Assumption A3 actually depends on the following fact.
\begin{lemma} \label{l:alpha}
Assumption A2 implies that, for an appropriate choice of $\tau$ satisfying A1, there exists $\alpha>2$ satisfying Assumption A3.
\end{lemma}
\begin{proof}
 Let us prove that, for any pair of necks indexed by $i$, $j$,
\begin{equation*}
 2-\beta_i<(1-\tau)(2-\hat{\lambda}_j)+\tau(2-\beta_j).
\end{equation*}
Rearranging, this is equivalent to
\begin{equation*}
 \hat{\lambda}_j-\beta_i<\tau(\hat{\lambda}_j-\beta_j).
\end{equation*}
The fact $\hat{\beta}_j>\hat{\lambda}_j$ implies 
\begin{equation*}
 \tau<\frac{\hat{\lambda}_j-\beta_i}{\hat{\lambda}_j-\beta_j}=1-\frac{\beta_j-\beta_i}{\beta_j-\hat{\lambda}_j}.
\end{equation*}
Think of this as a condition on $\tau$. Thanks to Assumption A2 this condition is compatible with Assumption A1. We can thus choose $\alpha$ as desired.
\end{proof}
We now distinguish two geometrically distinct cases. Each case requires a slightly different additional assumption, leading to its own gluing theorem.
\subsubsection*{AC SL connect sums}
Let us assume that $(L_{\bt},\iota_{\boldsymbol{t}})$ is an AC conifold. In other words, we assume that $\hat{L}$ is an AC conifold and that $S^*$ is the set of all CS ends of $L$, so that the connect sum construction removes all singularities at once. Notice however that, because we work in the immersed category, we have the flexibility of treating transverse intersection points either as singularities to be removed or as smooth points to be ignored, cf. also Example \ref{ex:SL_desing}.

We add another assumption to those listed above:
\begin{description}
\item[A4] The weights $\boldsymbol{\beta}$ on $L$, $\hat{\boldsymbol{\beta}}$ on $\hat{L}$, must satisfy a stronger version of the assumptions of Theorem \ref{thm:sum_injective}, ensuring surjectivity of the Laplacian. Specifically, we need to assume that the weights satisfy 
\begin{equation*}
\left\{
\begin{array}{l}
\hat{\beta}_i\in (2-m,0)\mbox{ for all ends $\hat{S}_i$ in $\hat{L}$}\\
\beta_i\in (2-m,0)\mbox{ for all ends $S_i$ in $L$}.\\
\end{array}
\right.
\end{equation*}
\end{description}
Define $F_{\bt}$ as in Equation \ref{eq:tFsobolev}. The above assumptions, together with the ideas of Lemma \ref{lemma:extra} case (2), guarantee the validity of Proposition \ref{prop:tFproperties} but we can now prove more, as follows. 
\begin{prop}\label{prop:Q_estimates}
The map $F_{\bt}$ has the following properties.
\begin{enumerate}
\item The linearized operator 
\begin{equation*}
 P_{\bt}:=d^{\star_{\bt}}\left((\star_{\bt}\iota_{\bt}^*(\Real\,\tilde{\Omega}))df\right): W^p_{3,\boldsymbol{\beta}_{\bt}}(L_{\bt})\rightarrow W^p_{1,\boldsymbol{\beta}_{\bt}-2}(L_{\bt})
\end{equation*}
is a linear self-adjoint elliptic operator and is a compactly-supported perturbation of $\Delta_{g_{\bt}}$. It thus has the same exceptional weights and change of index formula as $\Delta_{g_{\bt}}$, in the sense of Section \ref{ss:laplace_conifolds}. In particular, Corollary \ref{cor:laplaceresults} and Theorem \ref{thm:sum_injective} apply \textit{verbatim} to $P_{\bt}$.
\item There exists $C>0$ such that, for each $\bt$ sufficiently small,
\begin{equation*}
\|F_{\bt}(0)\|_{W^p_{1,\boldsymbol{\beta}_{\bt}-2}}\leq 
C\max{\{t^{(1-\tau)(2-\hat{\lambda}_i)+\tau(2-\beta_i)}\}}<t^\alpha/2,
\end{equation*}
where $\max$ is taken over all necks. 
\end{enumerate}
\end{prop}
\begin{proof}
The fact that $(L_{\bt}, \iota_{\bt})$ is SL away from the necks implies that $\star_{\bt}\iota_{\bt}^*(\Real\,\tilde{\Omega})\equiv 1$ away from the necks, so $P_{\bt}=\Delta_{g_{\bt}}$ there. The proofs of Corollary \ref{cor:laplaceresults} and Theorem \ref{thm:sum_injective} depend mostly on the properties of the operator listed in (1). The only additional ingredient is the rescaling property $\Delta_{t^2g}=t^{-2}\Delta_g$. It is simple to check that $P_{\bt}$ rescales similarly. This proves (1).  

To prove (2) we rely on the estimates given in Joyce \cite{joyce:III}. His computations concern connect sum SLs in general ``almost Calabi-Yau'' ambient spaces. In our case the ambient space is $\C^m$ so certain things simplify.
 
Joyce's Proposition 6.4 computes point-wise estimates for $F_{\bt}(0)$. In our case $L_{\bt}$ is SL away from the regions $\Sigma_i\times [t_i^\tau,2t_i^\tau]$ so Joyce's result simplifies, giving:
\begin{align*}
|F_{\bt}(0)|(\theta,r) &\leq
\left\{
\begin{array}{ll}
C t_i^{\tau(\mu_i-2)}+Ct_i^{(1-\tau)(2-\hat{\lambda}_i)} & \mbox{ on } \Sigma_i\times[t_i^\tau,2t_i^\tau]\\
0 & \mbox{ elsewhere}
\end{array}
\right.\\
|d(F_{\bt}(0))|(\theta,r) &\leq
\left\{
\begin{array}{ll}
C t_i^{\tau(\mu_i-3)}+Ct_i^{(1-\tau)(2-\hat{\lambda}_i)-\tau} & \mbox{ on } \Sigma_i\times[t_i^\tau,2t_i^\tau]\\
0 & \mbox{ elsewhere}.
\end{array}
\right.
\end{align*}
Multiplying the first equation by $w_{\bt}\rho_{\bt}^{2-\frac{m}{p}}$, raising to the power $p$ and integrating we obtain, for each neck,
\begin{equation*}
\begin{split}
 \|F_{\bt}(0)\|_{L^p_{\boldsymbol{\beta}_{\bt}-2}}^p &=\int_{\Sigma_i\times[t_i^\tau,2t_i^\tau]}|r^{2-\beta_i}F_{\bt}(0)|^pr^{-m}\mbox{vol}_{g_{\bt}}\leq C\int_{\Sigma_i\times[t_i^\tau,2t_i^\tau]}|r^{2-\beta_i}F_{\bt}(0)|^pr^{-m}\mbox{vol}_{\tg}\\
&\leq C t_i^{\tau(2-\beta_i)p}(t_i^{\tau(\mu_i-2)}+t_i^{(1-\tau)(2-\hat{\lambda}_i)})^p\\
&\leq C t^{\tau(2-\beta_i)p}(t^{\tau(\mu_i-2)}+t^{(1-\tau)(2-\hat{\lambda}_i)})^p.
\end{split}
\end{equation*}
Analogously, using $w_{\bt}\rho_{\bt}^{1-\frac{m}{p}}$, one obtains $L^p_{\boldsymbol{\beta}_{\bt}-1}$ estimates for $d(F_{\bt}(0))$. Combining, we find
\begin{equation*}
\|F_{\bt}(0)\|_{W^p_{1,\boldsymbol{\beta}_{\bt}-2}}\leq 
C \max{\{t^{\tau(2-\beta_i)}(t^{\tau(\mu_i-2)}+t^{(1-\tau)(2-\hat{\lambda}_i)})\}}.
\end{equation*}
Our Assumption A1 implies that $(1-\tau)(2-\hat{\lambda}_i)< \tau(\mu_i-2)$, allowing us to simplify this estimate. Our choice of $\alpha$ now implies (2).
\end{proof}
We can now state and prove our first result concerning SL connect sums.
\begin{theorem}\label{thm:ACSLgluing}
Let $(L,\iota)$, $(\hat{L},\hat{\iota})$ be compatible marked SL conifolds. Assume that $\hat{L}$ is an AC conifold and that $S^*$ contains all CS ends of $L$, so that the connect sum $(L_{\boldsymbol{t}},\iota_{\bt})$ is an AC Lagrangian submanifold. Let $\mathcal{C}_i$ denote the corresponding cones. Choose constants $p$, $\tau$, $\alpha$ and weights $\boldsymbol{\beta}$, $\hat{\boldsymbol{\beta}}$ satisfying Assumptions A1-A4 above. 

Then for all $\bt>0$ in the circular sectors defined in Proposition \ref{prop:tFproperties} there exists $f_{\bt}\in B_{t^\alpha}\cap C^\infty_{\boldsymbol{\beta}_{\bt}}(L_{\bt})$ such that $\iota_{\bt}':=\Phi_{L_{\bt}}\circ df_{\bt}:L_{\bt}\rightarrow \C^m$ is an AC SL submanifold asymptotic to the same cones $\mathcal{C}_i$.
\end{theorem}
\begin{proof}
 Let $P_{\bt}$ denote the linearized operator. According to Proposition \ref{prop:Q_estimates} (1), our assumptions on $\boldsymbol{\beta}$, $\hat{\boldsymbol{\beta}}$ imply that 
\begin{equation*}
P_{\bt}:W^p_{3,\boldsymbol{\beta}_{\bt}}(L_{\bt})\rightarrow W^p_{1,\boldsymbol{\beta}_{\bt}-2}(L_{\bt})
\end{equation*}
is a surjective isomorphism and that there exists $C>0$ such that, for all $\boldsymbol{t}$ and
all $f\in W^p_{3,\boldsymbol{\beta}_{\boldsymbol{t}}}(L_{\boldsymbol{t}})$,
\begin{equation}\label{eq:Pt_invertible}
\|f\|_{W^p_{3,\boldsymbol{\beta}_{\boldsymbol{t}}}}\leq C\|P_{\boldsymbol{t}}
(f)\|_{W^p_{1,\boldsymbol{\beta}_{\boldsymbol{t}}-2}}.
\end{equation}
Notice that
\begin{equation*}
F_{\bt}(f)=0\Leftrightarrow -F_{\bt}(0)-Q_{\bt}(f)=P_{\bt}(f) \Leftrightarrow G_{\bt}(f)=f,
\end{equation*}
where $G_{\bt}$ is the map
\begin{equation*}
G_{\bt}:B_{t^{\alpha}}\subset W^p_{3,\boldsymbol{\beta}_{\bt}}(L_{\bt})\rightarrow  W^p_{3,\boldsymbol{\beta}_{\bt}}(L_{\bt}), \ \ f\mapsto P_{\bt}^{-1}(-F_{\bt}(0)-Q_{\bt}(f)).
\end{equation*}
In other words, functions $f$ such that $\Phi_{L_{\bt}}\circ df$ is a SL immersion of $L_{\bt}$ correspond exactly to fixed points of $G_{\bt}$. 

Let us prove that $G_{\bt}$ maps $B_{t^{\alpha}}$ into itself. Choose $f\in B_{t^{\alpha}}$. Then, using the fact that $Q_{\bt}(0)=0$ and Propositions \ref{prop:tFproperties}, \ref{prop:Q_estimates},
\begin{equation}\label{eq:selfmap}
\begin{split}
 \|G_{\bt}(f)\|_{W^p_{3,\boldsymbol{\beta}_{\bt}}}&\leq C (\|F_{\bt}(0)\|_{W^p_{1,\boldsymbol{\beta}_{\bt}-2}}+\|Q_{\bt}(f)\|_{W^p_{1,\boldsymbol{\beta}_{\bt}-2}})\\
&\leq C(\|F_{\bt}(0)\|_{W^p_{1,\boldsymbol{\beta}_{\bt}-2}}+\|Q_{\bt}(f)-Q_{\bt}(0)\|_{W^p_{1,\boldsymbol{\beta}_{\bt}-2}})\\
&< t^{\alpha}/2+Ct^{\alpha+\min\{\beta_i-2\}}\|f-0\|_{W^p_{3,\boldsymbol{\beta}_{\bt}}},\\
\end{split}
\end{equation}
where $C$ incorporates the constants of Proposition \ref{prop:tFproperties} and of Equation \ref{eq:Pt_invertible}.
When $\bt$ is small enough, $Ct^{\alpha+\min\{\beta_i-2\}}\leq 1/2$.
It is thus clear that $G_{\bt}$ maps $B_{t^{\alpha}}$ into itself. Likewise,
\begin{align*}
 \|G_{\bt}(f)-G_{\bt}(g)\|_{W^p_{3,\boldsymbol{\beta}_{\bt}}}&\leq C \|-Q_{\bt}(f)+Q_{\bt}(g) \|_{W^p_{1,\boldsymbol{\beta}_{\bt}-2}}\\
&\leq Ct^{\alpha+\min\{\beta_i-2\}}\|f-g\|_{W^p_{3,\boldsymbol{\beta}_{\bt}}},
\end{align*}
showing that $G_{\bt}$ is a contraction. We can now apply the standard Fixed Point Theorem for contractions of Banach spaces to prove that $G_{\bt}$ admits a unique fixed point $f_{\bt}\in B_{t^\alpha}$.

The fact that $f_{\bt}\in C^\infty_{\boldsymbol{\beta}_{\bt}}(L_{\bt})$ follows from Joyce's regularity results \cite{joyce:I}.
\end{proof}

\begin{remark}\label{rem:convergence_rate}
Joyce's regularity results \cite{joyce:I} Theorem 7.11 show that if the ends of $(L_{\bt},\iota_{\bt})$ converge towards the corresponding asymptotic cones with rates in the interval $(2-m,0)$ then they converge with rates $2-m+\epsilon$, for any small $\epsilon>0$. The same is true for $(L_{\bt},\iota'_{\bt})$. 

If the convergence rate of $(L_{\bt},\iota_{\bt})$ is greater than or equal to zero, then our assumption $\boldsymbol{\beta}_{\bt}<0$ implies that the perturbation $df_{\bt}$ is of lower order than the original convergence rate, so $(L_{\bt},\iota'_{\bt})$ will have the same convergence rate as $(L_{\bt},\iota_{\bt})$. 
\end{remark}

\begin{example}\label{ex:SL_planes_desing} Consider the setting of Example \ref{ex:lagr_planes_desing}. Assume the initial planes are SL, and that at each intersection point they satisfy the conditions yielding the existence of a Lawlor neck interpolating between the corresponding pair of planes, as in Example \ref{e:C^m_examples}. Lawlor necks satisfy the necessary condition on convergence rates so we obtain an AC Lagrangian desingularization $L_{\bt}$. Theorem \ref{thm:ACSLgluing} now proves that $L_{\bt}$ can be perturbed to an embedded AC SL which is asymptotic to the initial configuration of planes. According to Remark \ref{rem:convergence_rate}, the convergence rate is $2-m+\epsilon$, for any small $\epsilon>0$.
 \end{example}

\begin{example}\label{ex:SL_desing}
More generally, Theorem \ref{thm:ACSLgluing} allows us to desingularize all transverse intersections for which there exists a Lawlor neck. The fact that we are working in the immersed category allows us to not worry about other intersections: we simply label them as smooth, immersed, points. The theorem also proves that there is no upper bound on the number of AC ends of AC SL conifolds in $\C^m$: given any AC SL, we can intersect it with an arbitrary number of transverse SL planes satisfying Lawlor's angle condition and then resolve the intersections using Lawlor necks. 
\end{example}

\begin{example}\label{ex:SL_doubling}
Let $L$ be a SL conifold in $\C^m$ with one isolated conical singularity asymptotic to a SL cone $\mathcal{C}$. Recall from Example \ref{e:C^m_examples} that there exists a 2-ended AC SL $\hat{L}$ asymptotic to $\mathcal{C}\cup-(e^{i\pi/m}\mathcal{C})$. The gluing construction defined in Theorem \ref{thm:ACSLgluing} basically allows us to replace the conical singularity with a new AC end asymptotic to the rotated cone. The final result is a smooth AC SL with the same number of ends as the initial conifold. This same procedure can be applied to any number of singularities. Applying it to a smooth point of $L$ we obtain the same result as if we intersected it with the
 rotated SL plane and then glued in a Lawlor neck.
\end{example}

\subsubsection*{Connect sums with moving singularities}

The second case of interest is when $(L_{\bt},\iota_{\bt})$ has CS singularities. To obtain a gluing result in this case we need to replace Assumption A4 above with the following:
\begin{description}
\item[A4] We assume that the weights satisfy 
\begin{equation*}
\left\{
\begin{array}{l}
\hat{\beta}_i\in (2-m,0)\mbox{ for all AC ends $\hat{S}_i$ in $\hat{L}$ and $S_i$ in $L$}\\
\beta_i\in (2,2+\epsilon)\mbox{ for all CS ends $\hat{S}_i\in\hat{S}^{**}$ and $S_i\in S^{**}$},
\end{array}
\right.
\end{equation*}
where $\epsilon$ is chosen small enough so that $(2,2+\epsilon)$ does not contain any exceptional weights for the Laplace operator on the corresponding end. 
\end{description}

We also need to impose the following ``stability'' assumption on the singularities in $\hat{S}^{**}\cup S^{**}$. We refer to \cite{joyce:II}, \cite{haskins:complexity}, \cite{ohnita} for further information concerning this condition and for examples. Very few examples of stable cones are known, but below we explain why it is an important and natural assumption; for similar reasons it also appears in \cite{joyce:II}, \cite{pacini:sldefs}, \cite{behrndt}. 

\begin{definition}\label{def:stability}
Let $\mathcal{C}$ be a SL cone in $\C^m$. Let $(\Sigma,g')$ denote the link of
$\mathcal{C}$ with the induced metric. Assume $\mathcal{C}$ has a unique
singularity at the origin; equivalently, assume that $\Sigma$ is smooth and that
it is not a sphere $\Sph^{m-1}\subset \Sph^{2m-1}$. It is known, cf. \cite{pacini:sldefs}, that the standard action of $\sunitary
m\ltimes\C^m$ on $\C^m$ admits a ``moment map'' $\mu$ and that the components of
$\mu$ restrict to harmonic functions on $\mathcal{C}$. Let $G$ denote the
subgroup of $\sunitary m$ which preserves $\mathcal{C}$. Then $\mu$ defines on
$\mathcal{C}$ $2m$ linearly independent harmonic functions of linear growth; in
the notation of Definition \ref{def:exceptionalweights} these functions are
contained in the space $V_\gamma$ with $\gamma=1$. The moment map also defines
on $\mathcal{C}$ $m^2-1-\mbox{dim}(G)$ linearly independent harmonic functions
of quadratic growth: these belong to the space $V_\gamma$ with $\gamma=2$.
Constant functions define a third space of homogeneous harmonic functions on
$\mathcal{C}$, \textit{i.e.} elements in $V_\gamma$ with $\gamma=0$. In
particular, these three values of $\gamma$ are always exceptional values for the
operator $\Delta_{\tilde{g}}$ on any SL cone, in the sense of Definition
\ref{def:exceptionalweights}.

We say that $\mathcal{C}$ is \textit{stable} if these are the only functions in
$V_\gamma$ for $\gamma=0,1,2$ and if there are no other exceptional values
$\gamma$ in the interval $[0,2]$. More generally, let $L$ be a CS or CS/AC SL
submanifold. We say that a singularity $x_i$ of $L$ is \textit{stable} if the
corresponding cone $\mathcal{C}_i$ is stable.
\end{definition}

Consider the map $F_{\bt}$ defined as in Equation \ref{eq:tFsobolev}. Notice that, depending on the convergence rates of $(L_{\bt},\iota_{\bt})$, it is possible that the CS ends satisfy the assumptions of case (1) of Lemma \ref{lemma:extra} while the AC ends satisfy case (2). This discrepancy is not a problem because Lemma \ref{lemma:extra} is essentially local: it examines each end separately. As in Proposition \ref{prop:Q_estimates} (1), Assumptions A1-A4 would suffice to obtain injectivity of the linearization $P_{\bt}$. As in Corollary \ref{cor:laplaceresults}, however, this choice of weights does not lead to surjectivity. To get surjectivity it is necessary to add extra variables into the domain, thus enlarging the image. This motivates us to consider the map $\tilde{F}_{\bt}$ defined as in Equation \ref{eq:cstFsobolev} and its linearization $\tilde{P}_{\bt}$. The stability assumption will ensure that this introduces exactly the right number of new variables to ensure surjectivity of the linearized operator but still maintain injectivity. 
\begin{lemma}\label{l:norms_equivalent}
 Let $(E,\|\cdot\|)$ be a Banach space. Assume given a finite number of closed subspaces $E_1,\dots,E_n$ such that $E=\oplus E_i$. Let $\|\cdot\|_i$ denote the norm on $E_i$ induced by restricting $\|\cdot\|$. Then the norm $\|\cdot\|$ on $E$ is equivalent to the norm $\sum_{i=1}^n\|\cdot\|_i$ on $\oplus E_i$.
\end{lemma}
\begin{proof}
 The triangle inequality shows that the identity map $Id:\oplus E_i\rightarrow E$ is a continuous isomorphism with respect to the given norms. Applying the Open Mapping Theorem, it follows that it is a topological isomorphism.
\end{proof}

We can now state and prove our second result concerning SL connect sums.

\begin{theorem}\label{thm:conifold_gluing}
Let $(L,\iota)$, $(\hat{L},\hat{\iota})$ be compatible marked SL conifolds. Assume that all CS singularities of the connect sum $(L_{\boldsymbol{t}},\iota_{\bt})$ are stable. Let $\mathcal{C}_i$ denote the cones corresponding to the AC ends. Choose constants $p$, $\tau$, $\alpha$ and weights $\boldsymbol{\beta}$, $\hat{\boldsymbol{\beta}}$ satisfying Assumptions A1-A4 above. 

Then for all $\bt>0$ in the circular sectors defined in Proposition \ref{prop:tFproperties}
there exists $(\tilde{e}_{\bt},v_{\bt},f_{\bt})\in \tilde{B}_{t^{\alpha}}\cap \left(\tilde{\E}_{\bt}\times E_0\times C^\infty_{\boldsymbol{\beta}_{\bt}}(L_{\bt})\right)$ such that $\iota_{\bt}':=\Phi_{L_{\bt}}^{\bt\cdot\tilde{e}_{\bt}\cdot\bt^{-1}}\circ (dv_{\bt}+df_{\bt}):L_{\bt}\rightarrow \C^m$ is a SL conifold whose AC ends are asymptotic to the same cones $\mathcal{C}_i$.
\end{theorem}
\begin{proof}
The first claim is that, for each $\bt$, the linearization 
\begin{equation*}
\tilde{P}_{\bt}:=d\tilde{F}_{\bt}[0]:T_{e}\tilde{\mathcal{E}}_{\bt}\oplus E_0\oplus W^p_{3,\boldsymbol{\beta}_{\bt}}(L_{\bt})\rightarrow W^p_{1,\boldsymbol{\beta}_{\bt}-2}(L_{\bt})
\end{equation*}
of $\tilde{F}_{\bt}$ is an isomorphism and is uniformly invertible. As in Proposition \ref{prop:Q_estimates} (1), uniform invertibility of $\tilde{P}_{\bt}$ is equivalent to uniform invertibility of the operator
\begin{equation}
 \tilde{\Delta}_{g_{\bt}}:T_{e}\tilde{\E}_{\bt}\oplus E_0\oplus W^p_{3,\boldsymbol{\beta}_{\bt}}\rightarrow W^p_{1,\boldsymbol{\beta}_{\bt}-2},\ \ (\tilde{e},v,f)\mapsto \Delta_{g_{\bt}}(\chi_{\bt}(\tilde{e})+v+f).
\end{equation}
Corollary \ref{cor:laplaceresults} shows that, for fixed $\bt$ and when restricted to $W^p_{3,{\boldsymbol{\beta}_{\bt}}}$,  $\tilde{\Delta}_{g_{\bt}}$ is injective with cokernel of dimension 
\begin{equation*}
\mbox{dim(Coker$(\tilde{\Delta}_{g_{\bt}})$)}=d, \ \
\mbox{where } d:=\sum_{i=1}^s \left(1+2m+m^2-1-\mbox{dim}(G_i)\right).
\end{equation*}
It is also true, cf. \cite{pacini:sldefs} Theorem 5.3 for details, that $\tilde{\Delta}_{g_{\bt}}$ is injective on its full domain.
The idea is that the function $\chi_{\bt}(\tilde{e})+v+f$ lives in a Sobolev space of the form $W^p_{k,(\boldsymbol{\mu},\boldsymbol{\lambda})}$ with $\boldsymbol{\mu}=-\boldsymbol{\epsilon}$ and $\boldsymbol{\lambda}<0$. According to Corollary \ref{cor:laplaceresults}, on this space the Laplace operator is injective. This means that if $\tilde{\Delta}_{g_{\bt}}(\tilde{e},v,f)=0$ then $\chi_{\bt}(\tilde{e})+v+f=0$ so the infinitesimal Lagrangian deformation $d(\chi_{\bt}(\tilde{e})+v+f)$ is trivial. From here we can deduce that $\tilde{e}=e$, $v=0$ and $f=0$.

 Our definitions imply that $\mbox{dim}(\tilde{\mathcal{E}}_{\bt}+E_0)=d$. It follows that, on the full domain, the cokernel of $\tilde{\Delta}_{g_{\bt}}$ is empty so $\tilde{\Delta}_{g_{\bt}}$ is an isomorphism.

We now want to show that $\tilde{\Delta}_{g_{\bt}}$ is uniformly invertible. To this end it is convenient to make an explicit choice of the norm $\|\cdot\|_{\bt}$ on $T_{e}\tilde{\E}_{\bt}$. Notice that any $\tilde{e}\in T_{e}\tilde{\E}_{\bt}$, thought of as a vector field on $\C^m$, can be written as a sum of vector fields
$\tilde{e}=\tilde{e}_1+\dots+\tilde{e}_{\hat{s}}+\tilde{e}_1+\dots+\tilde{e}_s$, each with compact support in a neighbourhood of a singularity.
 We thus set
\begin{equation*}
 \|\tilde{e}\|:=\sum_{i=1}^{\hat{s}}\|\Delta_{\hat{g}}\,\chi(\tilde{e}_i)\|_
{W^p_{1,\hat{\boldsymbol{\beta}}-2}(\hat{g})}+\sum_{i=1}^s\|\Delta_g\,\chi(\tilde{e}_i)\|_{W^p_{1,\boldsymbol{\beta}-2}(g)}.
\end{equation*}
We then define $\|\tilde{e}\|_{\bt}$ as in Equation \ref{eq:tnorms}. More explicitly, given $\tilde{e}_i$ in a neighbourhood of a singularity in $\hat{S}^{**}$ as above, we find
\begin{equation*}\label{eq:tnorm_on_hatL}
\begin{split}
 \|\tilde{e}_i\|_{\bt}&=t_i^{2-\beta_i}\|\Delta_{\hat{g}}\,\chi(\tilde{e}_i)\|_{W^p_{1,\hat{\boldsymbol{\beta}}-2}}=t_i^{-\beta_i}\|\Delta_{\hat{g}}\,t_i^2\chi(\tilde{e}_i)\|_{W^p_{1,\hat{\boldsymbol{\beta}}-2}}\\
&=t_i^{2-\beta_i}\|\Delta_{t_i^2\hat{g}}\,\chi_{\bt}(\tilde{e}_i)\|_{W^p_{1,\hat{\boldsymbol{\beta}}-2}}=\|\Delta_{g_{\bt}}\,\chi_{\bt}(\tilde{e}_i)\|_{W^p_{1,\boldsymbol{\beta}_{\bt}-2}}.\\
\end{split}
\end{equation*}
Given instead $\tilde{e}_i$ in a neighbourhood of a singularity in $S^{**}$ as above, we find
\begin{equation*}\label{eq:tnorm_on_L}
 \|\tilde{e}_i\|_{\bt}=\|\Delta_g\,\chi(\tilde{e}_i)\|_{W^p_{1,\boldsymbol{\beta}-2}}=\|\Delta_{g_{\bt}}\,\chi_{\bt}(\tilde{e}_i)\|_{W^p_{1,\boldsymbol{\beta}_{\bt}-2}}.
\end{equation*}
Basically, we have chosen a norm on $T_{e}\tilde{\E}_{\bt}$ so that, restricted to this space, $\tilde{\Delta}_{g_{\bt}}$ is uniformly invertible by definition.

It follows from the definition of $E_0$ that any $v\in E_0$ has a unique decomposition $v=a_1 v_1+\dots+a_{\hat{s}} v_{\hat{s}}+b_1v_1+\cdots+b_sv_s$. Injectivity means that the map $\tilde{\Delta}_{g_{\bt}}:E_0\rightarrow W^p_{1,\boldsymbol{\beta}_{\bt}-2}$ is an isomorphism on its image. Since $E_0$ is finite-dimensional, it is a topological isomorphism. We conclude that, for appropriate $C>0$,
\begin{equation*}
\begin{split}
 \|\sum_{i=1}^sb_iv_i\|_{\bt}&=\|\sum_{i=1}^sb_i\,dv_i\|_{W^p_{2,\boldsymbol{\beta}_{\bt}-1}(g_{\bt})}=\|\sum_{i=1}^sb_i\,dv_i\|_{W^p_{2,\boldsymbol{\beta}-1}(g)}\\
&\leq C\|\sum_{i=1}^sb_i\,\Delta_g(v_i)\|_{W^p_{1,\boldsymbol{\beta}-2}(g)}=C\|\sum_{i=1}^sb_i\,\Delta_{g_{\bt}}(v_i)\|_{W^p_{1,\boldsymbol{\beta}_{\bt}-2}(g_{\bt})}.\\
\end{split}
\end{equation*}
 Likewise, (omitting for simplicity the subscripts of $t$ and $\beta$),
\begin{equation*}
\begin{split}
 \|\sum_{i=1}^{\hat{s}}a_iv_i\|_{\bt}&=\|\sum_{i=1}^{\hat{s}}a_i\,dv_i\|_{W^p_{2,\boldsymbol{\beta}_{\bt}-1}(g_{\bt})}=\|\sum_{i=1}^{\hat{s}}a_i\,dv_i\|_{W^p_{2,\hat{\boldsymbol{\beta}}-1}(t^2\hat{g})}=t^{-1}t^{1-\beta}\|\sum_{i=1}^{\hat{s}}a_i\,dv_i\|_{W^p_{2,\hat{\boldsymbol{\beta}}-1}(\hat{g})}\\
&\leq Ct^{-\beta}\|\sum_{i=1}^{\hat{s}}a_i\,\Delta_{\hat{g}}(v_i)\|_{W^p_{1,\hat{\boldsymbol{\beta}}-2}(\hat{g})}=Ct^{2-\beta}\|\sum_{i=1}^{\hat{s}}a_i\,\Delta_{t^2\hat{g}}(v_i)\|_{W^p_{1,\hat{\boldsymbol{\beta}}-2}(\hat{g})}\\
&=C\|\sum_{i=1}^{\hat{s}}a_i\,\Delta_{g_{\bt}}(v_i)\|_{W^p_{1,\boldsymbol{\beta}_{\bt}-2}(g_{\bt})}.\\
\end{split}
\end{equation*}
This shows that $\tilde{\Delta}_{g_{\bt}}$ is uniformly invertible on $E_0$.
We also know from Theorem \ref{thm:sum_injective} that $\tilde{\Delta}_{g_{\bt}}$ is uniformly invertible on $W^p_{3,\boldsymbol{\beta}_{\bt}}$, but we still have to argue that it is uniformly invertible on the three spaces together, \textit{i.e.} on its full domain. \footnote{To understand this issue the reader may want to consider the linear map 
\begin{equation*}
 A_t:\R^2\rightarrow\R^2, \ \ A_t\cdot(x,y):=(x+y,ty).
\end{equation*}
This map is uniformly injective on the subspaces $E_1:=\mbox{span}\{(1,0)\}$, $E_2:=\mbox{span}\{(0,1)\}$ but it is clearly not uniformly invertible as $t\rightarrow 0$.}

Choosing an appropriate cut-off function $\eta_t$ as in the proof of Theorem \ref{thm:sum_injective}, we find
\begin{align}
\|\sum_{i=1}^s\tilde{e}_i\|_{\bt}+&\|\sum_{i=1}^{\hat{s}}\tilde{e}_i\|_{\bt}+\|\sum_{i=1}^sb_iv_i\|_{\bt}+ \|\sum_{i=1}^{\hat{s}}a_iv_i\|_{\bt}+\|f\|_{W^p_{3,\boldsymbol{\beta}_{\bt}}(g_{\bt})}\label{eq:first}\\
&\leq\|\sum_{i=1}^s\tilde{e}_i\|_{\bt}+\|\sum_{i=1}^sb_iv_i\|_{\bt}+\|\eta_tf\|_{W^p_{3,\boldsymbol{\beta}_{\bt}}(g_{\bt})}\label{eq:second}\\
&\ \ +\|\sum_{i=1}^{\hat{s}}\tilde{e}_i\|_{\bt}+\|\sum_{i=1}^{\hat{s}}a_iv_i\|_{\bt}+\|(1-\eta_t)f\|_{W^p_{3,\boldsymbol{\beta}_{\bt}}(g_{\bt})}\label{eq:third}.
\end{align}
We now estimate the two quantities (\ref{eq:second}), (\ref{eq:third}) separately.
\begin{align}
(\ref{eq:second})
&\leq\|\sum_{i=1}^s\Delta_g\chi(\tilde{e}_i)\|_{W^p_{1,\boldsymbol{\beta}-2}(g)}+C\|\sum_{i=1}^s\Delta_g(b_iv_i)\|_{W^p_{1,\boldsymbol{\beta}-2}(g)}+C\|\Delta_g(\eta_tf)\|_{W^p_{1,\boldsymbol{\beta}-2}(g)}\nonumber\\
&\leq\|\sum_{i=1}^s\eta_t\cdot\Delta_g\chi(\tilde{e}_i)\|+C\|\sum_{i=1}^s\eta_t\cdot \Delta_g(b_iv_i)\|+C\|\eta_t\cdot\Delta_gf\|+\frac{C}{|\log t|}\|f\|_{W^p_{3,\boldsymbol\beta_{\bt}}(g_{\bt})}\nonumber\\
&\simeq C\|\eta_t\cdot\Delta_g\big(\sum_{i=1}^s\chi(\tilde{e}_i)+\sum_{i=1}^sb_iv_i+f\big)\|_{W^p_{1,\boldsymbol{\beta}-2}(g)}+\frac{C}{|\log t|}\|f\|_{W^p_{3,\boldsymbol\beta_{\bt}}(g_{\bt})}\nonumber\\
&=C\|\eta_t\cdot\Delta_{g_{\bt}}\big(\sum_{i=1}^s\chi_{\bt}(\tilde{e}_i)+\sum_{i=1}^sb_iv_i+f\big)\|_{W^p_{1,\boldsymbol{\beta}_{\bt}-2}(g_{\bt})}+\frac{C}{|\log t|}\|f\|_{W^p_{3,\boldsymbol\beta_{\bt}}(g_{\bt})}\nonumber\\
&=C\|\eta_t\cdot\Delta_{g_{\bt}}\big(\sum_{i=1}^s\chi_{\bt}(\tilde{e}_i)+\sum_{i=1}^sb_iv_i+\sum_{i=1}^{\hat{s}}\chi_{\bt}(\tilde{e}_i)+\sum_{i=1}^{\hat{s}}a_iv_i+f\big)\|+\frac{C}{|\log t|}\|f\|,\label{eq:fourth}
\end{align}
where the second and the last lines use our assumptions on the supports of $\eta_t$, $\tilde{e}_i$ and $v_i$, the $1/\log$ term arises as in the proof of Theorem \ref{thm:sum_injective} and $\simeq$ follows from Lemma \ref{l:norms_equivalent}. Likewise, 
\begin{align}
(\ref{eq:third})
&\leq t^{-\beta}\|\sum_{i=1}^{\hat{s}}\Delta_{\hat{g}}\,t_i^2\chi(\tilde{e}_i)\|+Ct^{-\beta}\|\sum_{i=1}^{\hat{s}}\Delta_{\hat{g}}(a_i v_i)\|+Ct^{-\beta}\|\Delta_{\hat{g}}((1-\eta_t)f)\|\nonumber\\
&\leq Ct^{-\beta}\|(1-\eta_t)\cdot\Delta_{\hat{g}}\big(\sum_{i=1}^{\hat{s}}\chi_{\bt}(\tilde{e}_i)+\sum_{i=1}^{\hat{s}}a_iv_i+f\big)\|_{W^p_{1,\hat{\boldsymbol{\beta}}-2}(\hat{g})}+\frac{C}{|\log t|}\|f\|_{W^p_{3,\boldsymbol{\beta}_{\bt}}(g_{\bt})}\nonumber\\
&\simeq C\|(1-\eta_t)\cdot\Delta_{g_{\bt}}\big(\sum_{i=1}^{\hat{s}}\chi_{\bt}(\tilde{e}_i)+\sum_{i=1}^{\hat{s}}a_iv_i+f\big)\|_{W^p_{1,\boldsymbol{\beta}_{\bt}-2}(g_{\bt})}+\frac{C}{|\log t|}\|f\|_{W^p_{3,\boldsymbol{\beta}_{\bt}}(g_{\bt})}\nonumber\\
&= C\|(1-\eta_t)\cdot\Delta_{g_{\bt}}\big(\sum_{i=1}^s\chi_{\bt}(\tilde{e}_i)+\sum_{i=1}^sb_iv_i+\sum_{i=1}^{\hat{s}}\chi_{\bt}(\tilde{e}_i)+\sum_{i=1}^{\hat{s}}a_iv_i+f\big)\|+\frac{C}{|\log t|}\|f\|.\label{eq:fifth}
\end{align}
We conclude that
\begin{equation*}
(\ref{eq:first})\leq(\ref{eq:fourth})+(\ref{eq:fifth}).
\end{equation*}
Setting $h:=\Delta_{g_{\bt}}\big(\sum_{i=1}^s\chi_{\bt}(\tilde{e}_i)+\sum_{i=1}^{\hat{s}}\chi_{\bt}(\tilde{e}_i)+\sum_{i=1}^sb_iv_i+\sum_{i=1}^{\hat{s}}a_iv_i+f\big)$ and moving the $1/\log$ terms to the left hand side we obtain
\begin{equation*}
 (\ref{eq:first})\leq\|\eta_th\|_{W^p_{1,\boldsymbol{\beta}_{{\bt}-2}}(g_{\bt})}+\|(1-\eta_t)h\|_{W^p_{1,\boldsymbol{\beta}_{{\bt}-2}}(g_{\bt})}.
\end{equation*}
We now use once more estimates of the form $\|\eta_th\|\leq \|h\|+(1/|\log t|)\|h\|$ to show that there exist constants $C$, $C'$ such that, for any $h\in W^p_{1,\boldsymbol{\beta}_{{\bt}-2}}(g_{\bt})$, 
\begin{equation*}
 C\|h\|\leq\|\eta_th\|+\|(1-\eta_t)h\|\leq C'\|h\|.
\end{equation*}
We conclude that, for our $h$, $(\ref{eq:first})\leq C\|h\|$, thus proving that $\tilde{\Delta}_{g_{\bt}}$ is uniformly invertible.

As in Theorem \ref{thm:ACSLgluing} we now define a map 
\begin{equation*}
\tilde{G}_{\bt}: \tilde{B}_{t^\alpha}\subset \tilde{\mathcal{E}}_{\bt} \times E_0\times W^p_{3,\boldsymbol{\beta_{\bt}}}(L_{\bt})\rightarrow \tilde{\mathcal{E}}_{\bt} \times E_0\times W^p_{3,\boldsymbol{\beta_{\bt}}}(L_{\bt}).
\end{equation*}
Using the fact that $\tilde{F}_{\bt}(e,0,0)=F_{\bt}(0)$ we can check that $\tilde{G}_{\bt}$ is a contraction which maps $\tilde{B}_{t^\alpha}$ into itself. This suffices to prove the theorem.
\end{proof}

\begin{example}\label{ex:t_independentsings}
It is known that not all isolated conical singularities admit AC SL desingularizations, cf. \cite{haskinspacini}. Thus, in general, it will not be possible to replace all conical SL singularities with smooth compact regions, completely desingularizing the conifold (see however Example \ref{ex:SL_doubling}, which replaces each singularity with a non-compact AC SL end). It follows that, in general, the best one can do is to let $S^{*}$ denote the set of CS ends which do admit desingularizations and let $S^{**}$ contain the others. If the corresponding cones are stable, Theorem \ref{thm:conifold_gluing} allows us to perform gluing on this configuration.
\end{example}

\begin{example}\label{ex:SL_desing+sings}
 Let $\mathcal{C}$ be a stable SL cone. Choose a point $x\in\mathcal{C}$ and a SL plane $\Pi$ such that $T_x\mathcal{C}\cup\Pi$ satisfies Lawlor's angle condition. Let $L$ denote the disjoint union of $\mathcal{C}$ and $\Pi$ and let $\iota$ denote the natural immersion. Theorem \ref{thm:conifold_gluing} allows us to resolve the intersection $x$ using a Lawlor neck. The result will be a SL conifold with one CS singularity asymptotic to $\mathcal{C}$ and two AC ends: one asymptotic to $\mathcal{C}$, the other planar.

In a similar way one can resolve the intersection points of a configuration of several stable SL cones, rotated so as to meet Lawlor's angle conditions at each intersection point. 
\end{example}

\begin{example}\label{ex:t_dependentsings}
Assume given a conifold $L$ with a singularity modelled on a cone $\mathcal{C}$. Assume also that there exists a conifold $\hat{L}$ with one AC end modelled on $\mathcal{C}$ and one CS end modelled on a different cone $\mathcal{C}'$. If we manage to glue $\hat{L}$ into $L$ we will have found a procedure for replacing one singularity with another. In our framework this situation is described by allowing $\hat{S}^{**}$ to contain at least one CS end. If the corresponding cone is stable, Theorem \ref{thm:conifold_gluing} will allow us to perform gluing on this configuration. In particular, the position of the new singularity in $\C^m$ will be rescaled with $\bt$. 

Unfortunately, no such conifolds $\hat{L}$ are currently known. However, we can use the conifolds constructed in Example \ref{ex:SL_desing+sings} for a similar purpose. For example, let $\hat{L}$ be the SL conifold with one CS singularity and two AC ends described in Example \ref{ex:SL_desing+sings}. Now choose a new SL conifold $L$ and a smooth point $x\in L$. We can use the planar end in $\hat{L}$ to glue $\hat{L}$ onto $L$, in a neighbourhood of $x$. The resulting conifold inherits from $\hat{L}$ both the singularity and the other end.
\end{example}

\begin{remark}
 In Theorem \ref{thm:ACSLgluing} the limitation on $\bt$ imposed by $M$ prevents one $t_i$ from becoming zero before the others. Geometrically, it prevents one component of $\hat{L}$ from collapsing to a point before the others. Theorem \ref{thm:conifold_gluing} shows how to deal with singularities so it seems reasonable to think that Theorem \ref{thm:ACSLgluing} might be extended so as to remove this limitation, using the same techniques as above. In particular, it would probably be necessary to assume that the corresponding singularity is stable and to allow it to move in $\C^m$ in the usual manner. More generally, it seems that the analogous limitation could be removed also from Theorem \ref{thm:conifold_gluing}. 
\end{remark}


\ 

{\bf Acknowledgments.} I would like to thank Dominic Joyce for many useful conversations. This work started while I was a Marie Curie EIF Fellow at the University of Oxford and continued at the Scuola Normale Superiore of Pisa, partially supported by a Marie Curie ERG grant. I also thank M. Haskins, J. Lotay and L. Mazzieri for useful comments. 

\bibliographystyle{amsplain}
\bibliography{slgluing}
\end{document}